\documentclass[a4paper,11pt,
 leqno
]{article}

\usepackage{comment}
\usepackage[utf8]{inputenc}
\usepackage[T1]{fontenc}
\usepackage{amsthm}
\usepackage{amssymb}
\usepackage{mathrsfs}
\usepackage{amsmath}
\usepackage{amsfonts}
\usepackage{mathtools}
\usepackage{graphicx}
\usepackage{float}
\usepackage{dsfont}
\usepackage{enumerate}
\usepackage{enumitem}
\usepackage[a4paper]{geometry}
\usepackage{csquotes}
\usepackage{IEEEtrantools}
\usepackage{appendix}
\usepackage{constants}
\usepackage[obeyDraft]{todonotes}
\usepackage{tikz-cd}
\usepackage{enumerate}
\usepackage{mathabx}
\usepackage{hyperref}
\hypersetup{linktocpage,
  colorlinks   = true, 
  urlcolor     = blue, 
  linkcolor    = blue,
  citecolor   = blue 
}

\newtheorem{The}{Theorem}[section]
\newtheorem{Lemme}[The]{Lemma}
\newtheorem{Prop}[The]{Proposition}
\newtheorem{Cor}[The]{Corollary}

\theoremstyle{definition}

\theoremstyle{remark}

\newtheorem{Rk}[The]{Remark}

\numberwithin{equation}{section}
\newcommand{\tend}[2]{\displaystyle\mathop{\longrightarrow}_{#1\rightarrow#2}}

\usepackage{caption}
\title{
\normalsize
\textbf{{
CRITICAL EXPONENTS FOR A PERCOLATION MODEL ON \\TRANSIENT GRAPHS
}}}
\author{}
\date{}
\makeatletter
\newcommand{\E}{\mathbb{E}}

\newcommand{\Q}{\mathbb{Q}}
\newcommand{\R}{\mathbb{R}}
\newcommand{\Z}{\mathbb{Z}}

\newcommand{\G}{\mathcal{G}}
\newcommand{\K}{\mathcal{K}}
\newcommand{\A}{\mathcal{A}}

\renewcommand{\P}{\mathbb{P}}
\newcommand{\eps}{\varepsilon}

\newcommand{\I}{{\cal I}}

\renewcommand{\phi}{\varphi}
\renewcommand{\tilde}{\widetilde}
\renewcommand{\hat}{\widehat}
\renewcommand{\epsilon}{\varepsilon}

\newconstantfamily{c}{symbol=c}

\usepackage{color}
\definecolor{Red}{rgb}{1,0,0}
\definecolor{Blue}{rgb}{0,0,1}
\definecolor{Olive}{rgb}{0.41,0.55,0.13}
\definecolor{Yarok}{rgb}{0,0.5,0}
\definecolor{Green}{rgb}{0,1,0}
\definecolor{MGreen}{rgb}{0,0.8,0}
\definecolor{DGreen}{rgb}{0,0.55,0}
\definecolor{Yellow}{rgb}{1,1,0}
\definecolor{Cyan}{rgb}{0,1,1}
\definecolor{Magenta}{rgb}{1,0,1}
\definecolor{Orange}{rgb}{1,.5,0}
\definecolor{Violet}{rgb}{.5,0,.5}
\definecolor{Purple}{rgb}{.75,0,.25}
\definecolor{Brown}{rgb}{.75,.5,.25}
\definecolor{Grey}{rgb}{.7,.7,.7}
\definecolor{Black}{rgb}{0,0,0}

\usepackage{titletoc}

\dottedcontents{section}[4em]{}{2.9em}{0.7pc}
\dottedcontents{subsection}[0em]{}{3.3em}{1pc}

\usepackage{multicol}
\usepackage{parcolumns}

\usepackage{tabu}
\usepackage{caption}
\begin{document}
\thispagestyle{empty}
\maketitle
\vspace{0.1cm}
\begin{center}
\vspace{-1.9cm}
Alexander Drewitz$^1$, Alexis Pr\'evost$^2$ and Pierre-Fran\c cois Rodriguez$^3$

\end{center}
\vspace{0.1cm}
\begin{abstract}
\centering
\begin{minipage}{0.85\textwidth}
We consider the bond percolation problem on a transient weighted graph induced by the excursion sets of the Gaussian free field on the corresponding cable system. Owing to the continuity of this setup and the strong Markov property of the field on the one hand, and the links with potential theory for the associated diffusion on the other, we rigorously determine the behavior of various key quantities related to the (near-)critical regime for this model. In particular, our results apply in case the base graph is the three-dimensional cubic lattice. They unveil the values of the associated critical exponents, which are explicit but not mean-field and consistent with predictions from scaling theory below the upper-critical dimension.
\end{minipage}
\end{abstract}

\vspace{7.5cm}
\begin{flushleft}

\noindent\rule{5cm}{0.4pt} \hfill January 2023 \\
\bigskip
\begin{multicols}{2}

$^1$Universit\"at zu K\"oln\\
Department Mathematik/Informatik \\
Weyertal 86--90 \\
50931 K\"oln, Germany. \\
\url{drewitz@math.uni-koeln.de}\\[2em]

$^2$University of Geneva\\
Section of Mathematics\\
24, rue du G\'enéral Dufour\\
1211 Genève 4, Suisse.
\\\url{alexis.prevost@unige.ch}\\[2em]

\columnbreak
\thispagestyle{empty}
\bigskip
\medskip
\hfill$^3$Imperial College London\\
\hfill Department of Mathematics\\
\hfill 180 Queen’s Gate\\
\hfill London SW7 2AZ, UK.\\
\hfill \url{p.rodriguez@imperial.ac.uk} 
\end{multicols}
\end{flushleft}

\newpage

\section{Introduction}

Critical phenomena represent a fascinating challenge for mathematicians and physicists alike.
An emblematic example is that of second-order phase transitions, especially in models that are both non-planar and remain below a certain upper-critical dimension (above which mean-field behavior is expected). In such ``intermediate'' dimensions, which are physically very relevant, the regime near the transition point remains largely uncharted territory.

The present article rigorously investigates this problem in a benchmark case. Namely, given a weighted graph $\G$, transient for the random walk on $\G$, we study the bond percolation model obtained by considering the clusters of $\G$ induced by the excursion sets of the Gaussian free field $\varphi$ on the continuous graph (or cable system) $\tilde{\G} \supset \G$ associated to $\G$, see~\eqref{eq:introGFF}--\eqref{eq:intro_h_*} below for definitions. On the lattice $\Z^d$, $d \geq 3$, the study of the corresponding discrete problem, i.e.~the percolation of excursion sets of $\varphi\vert_{\mathcal{G}}$, was initiated in \cite{LEBOWITZ1986194} and more recently reinvigorated in~\cite{MR3053773}. The corresponding cable system free field $\varphi$ and its connections with Poissonian ensembles of (continuous) loops and bi-infinite Brownian trajectories on $\tilde{\G}$ have recently been studied in \cite{MR3502602}, \cite{Sz-16}, \cite{LW18}, \cite{DrePreRod2},~\cite{MR4112719} and~\cite{werner2020clusters}. Among these links, those relating $\varphi$ to the model of random interlacements, introduced in~\cite{MR2680403}, stand out. For, as will turn out, the interlacements essentially set a characteristic length scale $\xi$ for the percolation problem we study.

Our main results, Theorems~\ref{T:cap} and \ref{T2} below and their consequences, Corollaries~\ref{T1},~\ref{C:captails} and~\ref{C:scalingrelation}, describe the near-critical regime of the phase transition for the above percolation model by  rigorously deriving various associated critical exponents. These exponents capture the behavior of the system at and near the critical point; see e.g.~Section 1 of \cite{kesten1987} or Sections 9.1--9.2 in \cite{MR1707339} regarding the heuristic picture for independent (Bernoulli) percolation. In essence, our results determine a unique set of exponents, listed in Table~\ref{tb:exponents} on p.\pageref{tb:exponents}, along with a related ``capacity exponent''~$\kappa$, see \eqref{eq:capacityexp}. In special cases, the numerical values of some of these exponents are implicitly contained in \cite{MR4112719} and \cite{MR3502602}. 

The exponents we derive all turn out to be explicit rational functions of two parameters alone: the first one, $\nu$, cf.~\eqref{eq:intro_Green} on p.\pageref{eq:intro_sizeball} and \eqref{eq:introGFF}, describes the decay of the Green's function for the underlying random walk, and thus controls the decay of correlations. The second parameter, $\alpha,$ is geometric and governs the volume-growth of the base graph (see condition \eqref{eq:intro_sizeball} on p.\pageref{eq:intro_sizeball}). In particular, these conditions do not depend on the local structure of the graph, which hints at the conjectured universality of the critical exponents. In the parlance of renormalization group theory \cite{PhysRevB.4.3174, PhysRevB.4.3184}, the set of exponents we infer for each pair $(\nu,\alpha)$ characterizes the ``fixed point'' corresponding to the ``universality class'' of this percolation problem.  Importantly, the resulting values satisfy \textit{all} scaling and hyperscaling relations, which are (heavily) over-determined by our findings, and approach the corresponding mean-field values as $\nu \uparrow 4$ in case the walk is diffusive, i.e.~$\alpha=\nu+2$. We defer a more thorough discussion of these matters to the end of this section (cf.~below \eqref{eq:xiexplained}). The long-range dependence of the model, manifest through $\nu$, presents the advantage of inducing a certain \textit{structure} on the field. This is in contrast to the much studied, but locally more \textit{amorphous} Bernoulli percolation model, for which celebrated results have been derived on two-dimensional lattices, see \cite{SW01} and references therein, or on $\Z^d$ for sufficiently large~$d$ (in the mean-field regime), following \cite{MR1043524}, cf.~\cite{HeHo17} for an extensive account; see also \cite{cerf2015}, \cite{duminilcopin2020upper} for recent progress on~$\Z^3$, and \cite{dewan2021upper} regarding excursion sets of continuous Gaussian fields with rapid correlation decay.

We now describe our results. We consider $\mathcal{G}= (G,\lambda )$ a weighted graph, where $G$ is a countable infinite set, $\lambda_{x,y}\in{[0,\infty)}$, $x,y\in{G},$ are non-negative weights with $\lambda_{x,y} =\lambda_{y,x} \geq 0$, $\lambda_{x,x}=0$, and an edge connects $x$ and $y$ if and only if $\lambda_{x,y}>0$. We assume that $\mathcal{G}$ is connected, locally finite and transient for the random walk on $\mathcal{G}$, which is the continuous-time Markov chain generated by
\begin{equation}
\label{eq:generator}
Lf(x)=\frac1{\lambda_x}\sum_{y \in G}\lambda_{x,y}(f(y)-f(x)),
\end{equation}
for suitable $f:G\to \R$, where $\lambda_x=\sum_{y\in G} \lambda_{x,y}$. We write $\widetilde{\mathcal{G}}$  for the metric graph (or cable system) associated to $\mathcal{G}$, obtained by replacing each edge by a one-dimensional segment of length $1/2\lambda_{x,y}$ and gluing these segments through their endpoints. We denote by $P_x$ the law of the Brownian motion on $\widetilde{\mathcal{G}}$ when starting at $x \in \widetilde{\mathcal{G}}$ and by $X_{\cdot}$ the corresponding canonical process. This diffusion can be defined in terms of its Dirichlet form or directly constructed from a corresponding discrete-time Markov chain by adding independent Brownian excursions on the edges beginning at a vertex; we refer e.g.~to Section 2.1 of~\cite{DrePreRod3} for details regarding the construction of the measure $P_x$. We denote by $g(\cdot,\cdot)$ the Green function associated to this Brownian motion, that is the density of the local times of $X_{\cdot}$ at infinity with respect to the natural Lebesgue measure on $\widetilde{\mathcal{G}}$, which attaches length ${1/2\lambda_{x,y}}$ to every cable. The corresponding Gaussian free field $\varphi=(\varphi_x)_{x\in \widetilde{\mathcal{G}}}$, with canonical law $\P$, is the unique continuous centered Gaussian field with covariance function
\begin{equation}
\label{eq:introGFF}
\E[\varphi_x\varphi_y]=g(x,y),\quad x,y \in \widetilde{\mathcal{G}}.
\end{equation}
In view of \eqref{eq:introGFF}, the behavior of $\varphi$ is deeply linked to that of the underlying diffusion, and our findings greatly benefit from this interplay, as will soon become clear (see for instance Theorem~\ref{T:cap} below).

In order to discuss geometric properties we further endow $G$ with a distance function $d(\cdot,\cdot)$. For many cases of interest, one can afford to simply choose $d= d_{\text{gr}}$, the graph distance on $\mathcal{G}$, i.e.~$d_{\text{gr}}(x,y)=1$ if and only if $\lambda_{x,y}>0$ (extended to a geodesic distance on $G$), but see \eqref{eq:intro_Green} below and the discussion following \eqref{UE}, which may require a different choice. We write $B(x,r)=\{ y \in G : d(x,y) \leq r\}$, $x \in G$, $r>0$, for the discrete balls in the distance $d$ and tacitly assume throughout that the sets $B(x,r)$ have finite cardinality for all $x \in G$ and $r>0$. In the sequel, we define $K\subset \tilde{\mathcal{G}}$ to be bounded if $K\cap G$ is a bounded (or equivalently, finite) set. 

We now consider
\begin{equation}
\label{eq:intro0}
0, \text{ an arbitrary point in } \G,
\end{equation}
and introduce, for $a\in \R$,
\begin{equation}
\label{eq:introKa}
\begin{split}
&\mathcal{K}^a\stackrel{\text{def.}}{=}\tilde{\mathcal{K}}^a\cap G, \text{ where }\\
&\tilde{\mathcal{K}}^a \stackrel{\text{def.}}{=} \text{ the connected component of $0$ in $\{ x \in \widetilde{\mathcal{G}}:  \varphi_x \geq a \}$}
\end{split}
\end{equation}
(with ${\mathcal{K}}^a=\tilde{\mathcal{K}}^a=\emptyset$ if $\varphi_0 <a$) and the percolation function 
\begin{equation}
\label{eq:intro_theta0}
{\theta}_0(a) \stackrel{\text{def.}}{=}\P({\mathcal{K}}^a \text{ is bounded}) \ (= \P(\tilde{\mathcal{K}}^a \text{ is bounded})), \quad a \in \R.
\end{equation}
One can also give an alternative (purely discrete) description of the random set $\mathcal{K}^a$ in \eqref{eq:introKa} without reference to $\tilde{\G}$ as follows. Conditionally on $(\varphi_x)_{x\in G}$, the field $\varphi$ (on $\tilde{\G}$) is obtained by adding independent Brownian bridges on each edge $\{x,y\}$ of length $1/2\lambda_{x,y}$ of a Brownian motion with variance 2 at time $1$, interpolating between the values of $\varphi$ at the endpoints (see e.g.~\cite[(2.5)--(2.8)]{DrePreRod} for the case of the Euclidean lattice; this discussion remains valid in the present setup of transient weighted graphs). In light of this, $\mathcal{K}^a$ can be viewed as the open cluster of $0$ in the following bond percolation model on $\mathcal{G}$: given the discrete Gaussian free field $(\varphi_x)_{x\in G}$, one opens each edge $\{x,y\}$ independently with conditional probability (see e.g.~\cite{MR1912205}, Chap.IV, \S 26, p.67)
\begin{equation}
\label{eq:cond_proba}
1-\exp\big\{-2\lambda_{x,y} (\varphi_x-a)_+(\varphi_y-a)_+ \big\} \qquad \text{(with }z_+ = z\vee0, \, z \in \R).
\end{equation}

In view of \eqref{eq:intro_theta0}, one defines the critical parameter associated to this percolation model as
\begin{equation}
\label{eq:intro_h_*}
a_* = a_* (\G) =\inf\{ a \in \R: \,  {\theta}_0(a) =1 \}.
\end{equation}
The regime $a> a_*$ will be referred to as subcritical and \eqref{eq:intro_h_*} implies that the probability for $\{ \varphi \geq a\}$ to contain an unbounded cluster (anywhere in $\tilde{\mathcal{G}}$) vanishes for such $a$. Correspondingly, this probability is strictly positive when $a< a_*$, which constitutes the supercritical regime. By adapting a soft indirect argument due to \cite{MR914444}, one knows that $a_* \geq 0$ for \textit{any} transient $ {\mathcal{G}}$. We will virtually always (except in \eqref{eq:lawcap_gen}) assume that
\begin{equation}
\label{T1_sign}
 {\theta}_0(a)\big|_{ a=0} = 1.
\end{equation}
As shown in our companion article \cite{DrePreRod3}, see Theorem~1.1,1) and Lemma~3.4,2) therein (see also Remark~\ref{R:cap},\ref{R:cap.sign} below), 
the condition \eqref{T1_sign} is generic in that it is satisfied for a wide range of graphs $\G$. For instance, 
\begin{equation}
\label{T1_signsat}
\text{any vertex-transitive graph $\G$ (with unit weights) satisfies \eqref{T1_sign}},
\end{equation}
see Corollary~1.2 in \cite{DrePreRod3}; for examples of graphs not verifying \eqref{T1_sign}, see Proposition 8.1 in~\cite{Pre1}. In combination with \eqref{T1_theta_0} below, \eqref{T1_sign} essentially settles the continuity question for this phase transition, which includes in particular all graphs in \eqref{T1_signsat}. Our first theorem concerns the observable $\textnormal{cap}(\tilde{\mathcal{K}}^a)$, with $\tilde{\mathcal{K}}^a$ as in \eqref{eq:introKa} and where $\text{cap}(\cdot)$ denotes capacity, see \eqref{eq:cap} below, which plays a prominent role in this context. In only assuming \eqref{T1_sign} (cf.~also \eqref{T1_signsat}), the following result holds under very mild conditions on~$ {\mathcal{G}}$.

\begin{The}
\label{T:cap} 
For all $a \in \R$ and $u \geq 0$,
\begin{equation}
\label{eq:lawcap_gen}
\E\big[e^{-u\textnormal{cap}(\tilde{\mathcal{K}}^a)}1\{ \emptyset \neq  \tilde{\mathcal{K}}^a \textnormal{ bounded} \}\big]= \P\big( \emptyset \neq  \tilde{\mathcal{K}}^{\sqrt{2u+ a^2}} \textnormal{ bounded} \big).\end{equation}
In particular, if \eqref{T1_sign} holds, then
\begin{equation}
\label{eq:lawcap}
\E\big[e^{-u\textnormal{cap}(\tilde{\mathcal{K}}^a)}1\{ \tilde{\mathcal{K}}^a \textnormal{ bounded} \}\big]=\Phi(a)+1-\Phi(\sqrt{2u+ a^2}), \text{ for all } a \in \R, \, u \geq 0,
\end{equation}
where $\Phi(a) =\P(\varphi_0 \leq a)$.
\end{The}

We refer to \eqref{eq:rad_5} below for the density corresponding to the Laplace transform in \eqref{eq:lawcap}. The emergence of the observable $\textnormal{cap}(\tilde{\mathcal{K}}^a)$ is an instance of the aforementioned interplay with potential theory for the underlying diffusion. Theorem~\ref{T:cap} has several important consequences, among which the following two immediate corollaries.
\begin{Cor}\label{T1}
If \eqref{T1_sign} holds, then 
\begin{equation}
\label{T1_theta_0}
 {\theta}_0(a) = 2\Phi(a\wedge 0), \quad a \in \R.
\end{equation}
In particular, $a_* =0$, the function ${\theta}_0$ is continuous on $\R,$ and
\begin{equation}
\label{T1_beta}
\lim_{a\to0^-} \frac{1- {\theta}_0(a)}{|a|}=\sqrt{\frac{2}{\pi g}}, \end{equation}
where $g=g(0,0)$.
\end{Cor}

In the special case $G=\Z^d$, $d \geq3$ (with unit weights), the formula \eqref{T1_theta_0} was shown in \cite{MR4112719}, albeit by different methods (see also \cite{LW18} for a version of this result on finite graphs). Along with the other findings of~\cite{MR4112719}, cf.~\eqref{eq:psi_0nu=1} and Remark~\ref{R:rad},\ref{R:rad.2} below, these all turn out to be immediate consequences of Theorem~\ref{T:cap}. In view of~\eqref{T1_signsat}, these results are in fact true in far greater generality, and underlying them is the fundamental quantity $\textnormal{cap}(\tilde{\mathcal{K}}^a)$, which is integrable.

The Laplace functional \eqref{eq:lawcap} entails all the information about the capacity of bounded clusters at any height, including at and near the critical point $a_*= 0$, as reflected by: 
\begin{Cor}
\label{C:captails}
If \eqref{T1_sign} holds and $a_N$ 
satisfies $\lim_N N^{1/2}a_N =a_{\infty}\in{[-\infty,\infty]},$  then
\begin{equation}
\label{eq:captails}
    \sqrt{N} \P\big(\mathrm{{\rm cap}}(\tilde{\mathcal{K}}^{a_N})\geq N, \,  \tilde{\mathcal{K}}^{a_N} \textnormal{ bounded} \big)\tend{N}{\infty} \frac{1}{2\pi\sqrt{g}}\int_1^{\infty} t^{-3/2}\exp\Big(-\frac{a_\infty^2t}{2}\Big) \,  {\rm d}t.
\end{equation}
\end{Cor}

Theorem \ref{T:cap} and its corollaries are proved in Section~\ref{S:capandthetageneral} using an approach involving differential formulas developed in Section~\ref{S:diffform}, see in particular~Lemma \ref{L:df2} and Corollary~\ref{C:df}. The derivation of these formulas relies on the strong Markov property for $\varphi$ and a sweeping identity, which makes $\textnormal{cap}(\tilde{\mathcal{K}}^a)$ naturally appear, see for instance \eqref{eq:deriv1.1} or \eqref{eq:Cdf1} and Remark~\ref{R:cap_special}.

The appeal of Theorem~\ref{T:cap} and Corollaries~\ref{T1} and~\ref{C:captails} is in no small part due to the level of generality in which they hold. We  forewarn the perceptive reader not to mistake \eqref{T1_beta} as an indication of perpetual mean-field behavior, cf.~also \eqref{eq:intro_beta} below. Indeed, the results following below will show otherwise. In order to gain further insights, we make an additional assumption on $\mathcal{G}$, namely that
there exist an exponent $\nu>0$ and constants $ c,c' \in (0 ,\infty)$ (possibly depending on $\nu$) such that
\begin{equation}
\label{eq:intro_Green}\tag{$G_{\nu}$}
\begin{split}
&c\leq g(x,x)\leq c' 
\text{ and }  
c d(x,y)^{-\nu}\leq g(x,y)\leq c' d(x,y)^{-\nu} \quad \text{ for all }x \neq y\in{G},
\end{split}
\end{equation}
where $d(\cdot,\cdot)$ refers to the distance introduced below \eqref{eq:introGFF}. The condition \eqref{eq:intro_Green} actually implies \eqref{T1_sign}, as follows by combining Corollary~3.3,1) and Lemma~3.4,2) in \cite{DrePreRod3}.

We will also often require the graph to be $\alpha$-Ahlfors regular, i.e.~there exist a positive exponent $\alpha$ and 
$ c,c' \in (0 ,\infty)$ (possibly depending on $\alpha$) such that the volume growth condition
\begin{equation}
\label{eq:intro_sizeball} \tag{$V_{\alpha}$}
cr^{\alpha}\leq \lambda(B(x,r))\leq c'r^{\alpha}\quad \text{ for all }x\in{G}\text{ and }r\geq1,
\end{equation}
is satisfied (recall that $B(x,r)$ refers to the discrete ball of radius $r$ around $x \in G$, cf.~above \eqref{eq:intro0}). 
Moreover, we will at times rely on two additional technical assumptions (see also Remark~\ref{R:rad},\ref{R:rad.3} regarding a possible weakening), which we gather here for later reference:
\begin{align}
&\label{eq:ellipticity} 
\ \, \lambda_{x,y}/\lambda_x\geq \Cl[c]{c:ellipticity}\text{ for all }x\sim y\in{G} \text{ (i.e. if $\lambda_{x,y}>0$);}\\[0.3em]
&\label{eq:intro_shortgeodesic}
\begin{array}{l}
\text{there exists an infinite geodesic }(0=y_0,y_1,\dots)\text{ for $d_{\text{gr}}$} \\[0.3em]
\text{such that }d_{\text{gr}}(y_k,y_p)\leq \Cl[c]{Cgeo}d(y_k,y_p)\text{ for all }k,p\geq 0.
\end{array}
\end{align}
Condition \eqref{eq:ellipticity} is a standard ellipticity assumption in this context, which together with \eqref{eq:intro_Green} and \eqref{eq:intro_sizeball} forms a natural set of requirements from the perspective of the walk. Indeed, in case $d=d_{\text{gr}}$ the results of \cite{MR1853353} imply that \eqref{eq:intro_Green}, \eqref{eq:intro_sizeball} and \eqref{eq:ellipticity} are equivalent to upper and lower Gaussian (in case $\beta= \alpha-\nu=2$) or sub-Gaussian (in case $\beta= \alpha-\nu>2$; note that $\beta \geq 2$ always holds, cf.~\eqref{eq:condonalphanu} below) estimates on the heat kernel $q_t$ of the walk on $G$ of the form
\begin{equation}
\label{UE}
  ct^{-\frac{\alpha}{\beta}}\exp\Big\{-\Big(\frac{d(x,y)^{\beta}}{c't}\Big)^{\frac{1}{\beta-1}}\Big\} \leq   q_t(x,y)\leq  \tilde{c}t^{-\frac{\alpha}{\beta}}\exp\Big\{-\Big(\frac{d(x,y)^{\beta}}{\tilde{c}'t}\Big)^{\frac{1}{\beta-1}}\Big\}
\end{equation}
for all $x,y \in G$ and $t \geq 1\vee d(x,y)$.
 Condition \eqref{eq:intro_shortgeodesic}, which always holds in case $d=d_{\text{gr}}$ (see \cite{MR864581} regarding the existence of infinite geodesic ``rays'' for $d_{\text{gr}}$) is tailored to certain chaining arguments that will be used to build long connections in $\{ \varphi \geq a\}$. Its necessity is further explained in Remark~\ref{R:LB},\ref{rk:shortgeodesicsisneeded}. In fact, \eqref{eq:ellipticity} can often be weakened and together with \eqref{eq:intro_shortgeodesic}, the two conditions are in a sense complementary, see Remark~\ref{R:rad},\ref{R:rad.3} for more on this.

A canonical example satisfying all of \eqref{eq:intro_Green}, \eqref{eq:intro_sizeball}, \eqref{eq:ellipticity} and \eqref{eq:intro_shortgeodesic} is the Euclidean lattice $G= \Z^d$, $d \geq 3$, with unit weights and for the choices $d(\cdot,\cdot)=d_{\text{gr}}(\cdot,\cdot)$, $\nu=d-2$ and $\alpha=d.$ In particular, the emblematic case $G=\Z^3$ corresponds to $\nu=1$. More generally, this setup allows for disordered (random) uniformly elliptic weights $c \leq 
\lambda_{x,y} \leq c^{-1}$ (in fact, \eqref{eq:intro_Green} and \eqref{eq:ellipticity} alone only require $\lambda_{x,y} \geq c$, and \eqref{eq:intro_sizeball} implies $\lambda_{x,y} \leq c'$; see~Lemma 2.3 in \cite{DrePreRod2}). Our results then hold in a quenched sense, i.e.~for almost all realizations of $\lambda$.

Furthermore, all four conditions hold for instance for the examples discussed in~(1.4) of~\cite{DrePreRod2}, which include Cayley graphs of suitable volume growth, as well as various fractal graphs (possibly sub-diffusive). The flexibility in the choice of the distance $d$ takes into account that the heat flow on $\G$ may well propagate differently in different ``directions'', for instance if $\G$ has a product structure, which typically requires choosing $d\neq d_{\text{gr}}$ for \eqref{eq:intro_Green} to hold; see Proposition 3.5 in \cite{DrePreRod2} for more on this, as well as~\cite{DrePreRod2} and references therein for further examples. An instructive case in point is the graph $G= \text{Sierp}\times \Z$ considered in \cite{MR2891880} (endowed with unit weights), where $\text{Sierp}$ is the graphical Sierpinski gasket, whose projection on $\text{Sierp}$ is sub-diffusive, and which is a canonical example of graph with $\nu<1,$ see \cite{MR2177164, MR1378848}.

Note that, since $\G$ is assumed to be transient, once \eqref{eq:intro_sizeball}, \eqref{eq:intro_Green} and \eqref{eq:ellipticity} are satisfied, combining Theorem 1 and Proposition 3(a) in \cite{MR2076770} (see also (2.10) in \cite{DrePreRod2} for as to why our assumptions entail $\lambda_{x,y} \geq c$ for $x\sim y$, as required in \cite{MR2076770}), and Proposition 6.3 in \cite{MR1853353} one necessarily has, in case~$d=d_{\text{gr}}$,
 \begin{equation}
\label{eq:condonalphanu}
0<\nu\leq\alpha-2 \text{ (and in particular, $\alpha>2$)}.
\end{equation}
Moreover, combining Theorem 2 and Proposition 3(d) in \cite{MR2076770}, together with Theorem~2.1 and (4.2) in \cite{MR1853353}, one knows that for any set values $\alpha$ and $\nu$ satisfying \eqref{eq:condonalphanu}, there exists a graph satisfying \eqref{eq:intro_Green}, \eqref{eq:intro_sizeball} and \eqref{eq:ellipticity} (the latter follows by inspection of  \cite{MR2076770}, see p.13 therein), as well as \eqref{eq:intro_shortgeodesic} since $d=d_{\text{gr}}$ for these graphs. In the sequel, whenever we assume  \eqref{eq:intro_sizeball}, \eqref{eq:intro_Green} to hold simultaneously (for some distance function $d$), we tacitly assume \eqref{eq:condonalphanu} to be true.

Now, assuming only \eqref{eq:intro_Green} to hold (see p.\pageref{eq:intro_Green}), we consider the quantity
\begin{equation}
\label{eq:intro_psi}
\psi(a,r)\stackrel{\text{def.}}{=}\P(r \leq \text{rad}({\mathcal{K}}^a)<\infty), \quad \text{ for }r>0, \, a\in \R
\end{equation}
(cf.\ \eqref{eq:introKa} for the definition of ${\mathcal{K}}^a$), where 
$\text{rad}(A)\stackrel{\text{def.}}{=}\sup\{{d}(0,x): x \in A \}$, for $A \subset {{G}}$ with $0\in A$, as well as the
truncated two-point function
\begin{equation}
\label{eq:eq:deftauhf}
\tau_a^{\text{tr}}(0,x)\stackrel{\text{def.}}{=} \P(x\in {\mathcal{K}}^a, \, {\mathcal{K}}^a \text{ bounded})\ (=\P(x\in \tilde{\mathcal{K}}^a, \, \tilde{\mathcal{K}}^a \text{ bounded})), \text{ $a \in \R$, $x\in G$.}
\end{equation}
The quantity $\P(x \in \tilde{\mathcal{K}}^0)$, $x \in \tilde{\G}$, admits an exact formula, first observed in \cite{MR3502602}, which follows by combining Propositions~5.2 and~2.1 therein. Under \eqref{eq:intro_Green} (and \eqref{eq:intro_sizeball}) it yields that for all $x \in G$,
\begin{equation}
\label{eq:2ptatcriticality}
\tau_0^{\text{tr}}(0,x)=\frac2\pi \arcsin \Big(\textstyle\frac{g(0,x)}{\sqrt{g(0,0)g(x,x)}}\Big) \asymp d(0,x)^{2-\alpha-\eta} \text{ as $d(0,x)\to \infty$}, \text{ with } \eta=\nu-\alpha+2,
\end{equation}
 where $f\asymp g$ means that $cf \leq g \leq c'f$ for some constants $c,c'\in (0,\infty)$ (see the end of this introduction for our convention regarding constants). The arguably cumbersome parametrisation in \eqref{eq:2ptatcriticality} follows standard convention. It is arranged so that $\E[|{\mathcal{K}}^a \cap B_r|] \asymp r^{2-\eta}$, where $B_r=B(0,r)$, whence $\eta$ captures the discrepancy from mean-field behavior, cf.~the discussion below.
 
 With regards to $\psi(0,\cdot)$, by comparison with the capacity functional, i.e.~using Corollary~\ref{C:captails} in case $a_N\equiv 0$, see Remark~\ref{R:rhobounds} below, it is a simple consequence that for all $r \geq1$,
\begin{align}
&\label{eq:psi_0nu<1}
cr^{-\nu/2} \leq \psi(0,r) \leq c' r^{-\nu/2}, \quad \text{ if $\nu <1$,} \\
&\label{eq:psi_0nu=1}
cr^{-\nu/2} \leq \psi(0,r) \leq c' (r/(\log r)^{1\{ \nu=1\}})^{-1/2}, \quad \text{ if $\nu \geq1$};
\end{align}
see also \cite{MR4112719} for \eqref{eq:psi_0nu=1} when $G=\Z^d$, $d \geq3$, derived therein together with bounds on the critical window; see also Remark~\ref{R:rad},\ref{R:rad.2} below regarding improvements on the latter.

Our second main result gives precise estimates on $\psi(a,r)$ (and similarly for $\tau_a^{\text{tr}}(0,x)$), quantitative in $a$ and $r.$

\begin{The}[under \eqref{eq:intro_Green} and \eqref{eq:ellipticity}]
\label{T2} With  
\begin{equation}
\label{eq:defxi}
\xi(a)\stackrel{\textnormal{def.}}{=}|a|^{-\frac2{\nu}} \quad \text{ $($with the convention $\xi(0)=\infty)$},
\end{equation}
the following~hold:
\begin{enumerate}[label=(\roman*)]
\item \label{T2:nu<1}If $\nu < 1$, then for all $a \in \R$ and $r\geq 1$,
\begin{equation}
\label{eq:corrlength1}
\Cl[c]{c1_xi} \psi(0,r) \exp \Big\{-\Cl[c]{C2_xi} \Big(\frac{ r}{ \xi(a)}\Big)^\nu\Big\} \leq \psi(a,r) \leq \psi(0,r)\exp \Big\{-\Cl[c]{c4_xi}  \Big(\frac{ r}{ \xi(a)}\Big)^\nu\Big\}.
\end{equation}
\item \label{T2:nu>1}If $\nu \geq 1$, then for all $a \in \R$ and $r\geq1$, 
\begin{equation}
\label{eq:corrlength2}
  \psi(a,r) \leq\psi(0,r) \times \begin{cases}
  \exp \Big\{-\Cr{c4_xi}  \frac{ (r/\xi(a))}{ \log( r \vee 2)} \Big\}, &  \text{ if $\nu = 1,$} \\[0.8em]
   \exp \big\{-\Cr{c4_xi}  ra^2 \big\}, & \text{ if $\nu > 1.$} 
  \end{cases}
  \end{equation}
Furthermore, if \eqref{eq:intro_sizeball} and \eqref{eq:intro_shortgeodesic} are also satisfied, then for $\nu=1$ and all $|a|\leq c$, 
\begin{equation}
\label{eq:corrlength2lbpsi}
{\psi}(a,r) \geq \Cr{c1_xi}\psi(0,r) \times 
 \exp \Big\{-  \Cr{C2_xi}\frac{(r/ \xi(a))}{ \log ((r/\xi(a))\vee2)} \Big\},  \quad  \text{ if $\textstyle \frac{r}{\xi(a)}\notin{(1,(\log\xi(a))^{\Cl[c]{c:defect3}})}$,}
\end{equation}
with $\Cr{c:defect3}\in (0,1)$. Further, if $ \psi(0,r) \asymp r^{-1/2}$ $($cf.~\eqref{eq:psi_0nu=1}$)$ then \eqref{eq:corrlength2lbpsi} holds for all $r\geq1.$
\end{enumerate}
Moreover, the upper bounds in \eqref{eq:corrlength1}, \eqref{eq:corrlength2} remain valid upon replacing $ \psi(a,r)$ by $\tau_a^{\textnormal{tr}}(0,x)$ everywhere, with $r\stackrel{\text{def.}}{=}d(0,x) \geq 1;$ furthermore, in case \eqref{eq:intro_sizeball} holds and $d=d_{\textnormal{gr}}$, the lower bounds in \eqref{eq:corrlength1}
 remain valid for $|a| \leq c$, as well as
 \eqref{eq:corrlength2lbpsi} for $r\geq \xi(a)(\log\xi(a))^{\Cr{c:defect3}}$.
 \end{The}

The role of $\xi$ above as a natural length scale for the percolation problem \eqref{eq:introKa} confirms a prediction of \cite{PhysRevB.29.387, PhysRevB.27.413}. Indeed, for $\nu \leq 1$, \eqref{eq:corrlength1}, \eqref{eq:corrlength2} and \eqref{eq:corrlength2lbpsi} exhibit $\xi$ as the right correlation length in this model, with exponent $\nu_{c}$ (not to be confused with the parameter $\nu$ from \eqref{eq:intro_Green}, whence the subscript) defined as
\begin{equation} \label{eq:nuDeffirst}
\nu_{c} \stackrel{\text{def.}}{=} - \lim_{a \to 0 } \frac{\log(\xi(a))}{\log(a)}  \ \Big( = \frac{2}{\nu} \Big).
\end{equation}
In fact, \cite{PhysRevB.29.387, PhysRevB.27.413} conjecture that $\nu_c = 2/\nu$ is the correct correlation length exponent for any long-range percolation model with correlation decay exponent $\nu.$ We refer to \eqref{eq:defcorrlength}--\eqref{eq:intro_nu} below and to Corollary~\ref{Corcorrlength} for a more careful 
treatment of the correlation length, as well as to the discussion around \eqref{eq:defRI}--\eqref{eq:xiexplained} and to Theorem~\ref{Thm:locuniq} and Corollary~\ref{C:locuniq} below for further insight into the length scale $\xi=\xi(a)$ introduced in \eqref{eq:defxi}.

Regarding lower bounds for related quantities $\tilde{\psi}$, $\tilde{\tau}_a^{\textnormal{tr}}$, cf.~\eqref{eq:intro_psitilde}, which include the regime $\nu>1$, we refer to the discussion around Theorem~\ref{T:psitilde} at the very end of this introduction and to the recent article~\cite{GRS20} concerning results related to~\eqref{eq:corrlength2} and \eqref{eq:corrlength2lbpsi} for the discrete problem~on~$\Z^3,$ which witness $\xi=\xi(a)$ in yielding the bounds $ c\xi(a)^{-1} \leq -\frac{\log(r)}r\log {\psi}(a,r)\leq c'\xi(a)^{-1}$ valid for all large enough $r\geq R(a)$ but lack any quantitative control on $R(a)$.

The version of \eqref{eq:corrlength2} for $\tau_a^{\text{tr}}$ including the sharp pre-factor as $a \to 0$ will be crucial for our purposes, see Corollary~\ref{C:scalingrelation} below. Upper bounds for $\psi(a,r)$ akin to \eqref{eq:corrlength2}, but without the correct pre-factor $\psi(0,r)$ were derived in \cite{MR4112719}. In essence, all bounds for $\psi$ in case $\nu<1$ in Theorem~\ref{T2} as well as all off-critical \textit{upper} bounds (when $r/\xi(a) \geq c$) can straightorwardly be deduced either by direct comparison between $\text{rad}({\mathcal{K}}^a)$ and $\text{cap}(\tilde{\mathcal{K}}^a)$ in combination with Corollary~\ref{C:captails}, or, in the case of \eqref{eq:corrlength2}, by means of a suitable differential inequality.
 We return to the lower bounds \eqref{eq:corrlength2lbpsi} on $\psi$ and $\tau_{a}^{\textnormal{tr}}$ (as well as \eqref{eq:corrlength1} in case of $\tau_{a}^{\textnormal{tr}}$, for which comparison estimates with the capacity observable already fail) shortly, which illuminate~\eqref{eq:defxi} and rely on different ideas.

We now discuss important consequences of Theorem~\ref{T2} with regards to volume observables. For this purpose, let $|{\mathcal{K}}^a|=|\tilde{\mathcal{K}}^a \cap G|$ denote the volume (cardinality) of ${\mathcal{K}}^a$. The following result, which follows from Theorem~\ref{T2}, relates a quantity $\gamma$ governing the divergence of the expected volume of ${\mathcal{K}}^a$ (when bounded) as $a$ approaches $0$ with the exponents $\nu_c$
from \eqref{eq:nuDeffirst} and $\eta$ introduced in \eqref{eq:2ptatcriticality}. Its meaning in the context of scaling theory is further explained in the discussion at the end of this introduction.

\begin{Cor}[Scaling relation]\label{C:scalingrelation} For $\nu \leq 1$, if \eqref{eq:intro_Green}, \eqref{eq:intro_sizeball}, \eqref{eq:ellipticity} hold and $d=d_{\textnormal{gr}}$, the limit 
\begin{equation}
\label{eq:gamma1}
 \gamma \stackrel{\textnormal{def.}}{=}- \lim_{a \to 0} \frac{\log(\E[|\mathcal K^a|1\{|\mathcal{K}^a|<\infty\}])}{\log |a|} \end{equation}
exists and 
\begin{equation}
\label{eq:scaling_gamma}
\gamma= \nu_{c}(2-\eta) \ \Big(= \frac{2\alpha}\nu-2 \Big).
\end{equation}
For $\nu < 1$ one even has the stronger result $\E[|\K^a|1\{|\K^a|<\infty\}]\asymp |a|^{-\frac{2\alpha}\nu+2}$ as $a \to 0$ (recall below \eqref{eq:2ptatcriticality} for the definition of $\asymp$).
\end{Cor}

We refer to Proposition~\ref{P:scalinggamma} for the precise bounds on $\E[|\K^a|1\{|\K^a|<\infty\}]$ and to Remark~\ref{rk:1stscalingrelation},\ref{rk:1stscalingrelationvolumerenormalized} for related results regarding a ``renormalized'' volume observable. The ``softer'' conclusions of Corollary~\ref{C:scalingrelation}, which witness the correct scaling factor $\xi$ and integrability in $r/\xi$, will follow from the ``hard'' estimates of Theorem~\ref{T2}. Namely, we use the versions for $\tau_a^{\textnormal{tr}}(a,x)$ of \eqref{eq:corrlength1} in case $\nu<1$ and of \eqref{eq:corrlength2} and \eqref{eq:corrlength2lbpsi} in case $\nu =1$, while following the heuristics behind the scaling equality \eqref{eq:scaling_gamma}, see for instance \cite{MR1707339}, Chap.~9, to deduce Corollary~\ref{C:scalingrelation}. The fact that the exponent $\Cr{c:defect3}$ appearing in \eqref{eq:corrlength2lbpsi} is less than $1$ is absolutely instrumental in obtaining \eqref{eq:scaling_gamma} when $\nu=1$.

We now return to the lower bound(s) in \eqref{eq:corrlength2lbpsi} (and \eqref{eq:corrlength2lb} below) and their proofs, which are instructive. In both cases, we rely on a change-of-measure argument, somewhat similar to the one used in \cite{GRS20}, but quantitative (the arguments in~\cite{GRS20} operate at fixed level $a$ as $r\to \infty$); see also \cite{BDZ95}, \cite{Sz15}, for arguments of this kind in various contexts involving $\varphi$. We modify $\P$ so as to shift a given level $a \in (0, 1]$ to $-a$, which is (slightly) supercritical, in an appropriate region. This effectively translates the problem into building the desired long connection to distance $r$ at the new level $-a$ with sizeable probability. The intuitive renormalization picture is that this ought to happen by stacking boxes of side length roughly equal to $\xi(-a)=\xi(a)$ as given by \eqref{eq:defxi}, which ``start to see a good chunk'' of $\{\varphi \geq - a\}$. 

The approach delineated above yields the bound \eqref{eq:corrlength2lb} below for $\tilde{\psi}$. The bound  \eqref{eq:corrlength2lbpsi} is more subtle and requires amendments to this strategy. In essence (see also Fig.~\ref{F:LB} on p.\ \pageref{F:LB}), we explore a piece of the cluster of $\mathcal{K}^{a}$ inside $B_{\xi(a)}$, then apply the Markov property and perform the change of measure in the complement of the explored region, without getting too close to its boundary. On a suitable event, the explored part (as opposed to the single point $0$) is sufficiently ``visible'' for the gluing constructions performed below (for essentially the same reasons as those explained around \eqref{eq:xiexplained}). The explored part thereby manifests itself precisely as multiplicative ``critical cost'' $\psi(0,r)$ in the lower bound \eqref{eq:corrlength2lbpsi}. However, establishing this rigorously requires some care since the (Dirichlet) boundary condition forced by the exploration acts as a trap, which has the tendency to ``kill'' connections in the system $\mathcal{I}^u$ of ``highways'' used below. An important role in this context will be played by certain ``bridge'' trajectories, which emanate from the explored region and link to the net of highways.

Our approach to building the highways is driven by two key estimates, summarized in \eqref{eq:squares} and \eqref{eq:RSW} below, which can be regarded as partial substitutes for two essential ingredients that are usually available in planar settings at criticality: (i) squares of arbitrary size are crossed with probability $1/2$ (duality symmetry), and (ii) rectangles are crossed with sizeable probability across all scales (a ``Russo-Seymour-Welsh''-type bound); see~e.g.~\cite{MR1707339}, Chap.~11, see also \cite{koehlerschindler2020crossing} for latest developments in this direction. 

Our replacements for (i) and (ii) harvest a powerful and profound link between $\varphi$ and the random interlacement sets $(\mathcal{I}^u)_{u>0}$ on $\tilde{\G}$, see e.g.~Section~2.5 in \cite{DrePreRod3} for their precise definition in the present context. The random sets $\mathcal{I}^u \subset \tilde{\mathcal{G}}$, $u>0,$ can be jointly defined in such a way that $\mathcal{I}^u$ is increasing in $u$ and  its law is characterized by the property that
\begin{equation}
\label{eq:defRI}
\mathbb{P}(\mathcal{I}^u \cap K = \emptyset)=e^{-u\text{cap}(K)}, \text{ for  compact } K \subset \tilde{\mathcal{G}}
\end{equation}
(see the beginning of Section~\ref{S:diffform} below regarding compactness). In fact, $\mathcal{I}^u$ is realized as the trace of a Poisson cloud of bi-infinite transient continuous trajectories with intensity measured by $u$, and thus has only unbounded connected components. We will only use the fact here that whenever \eqref{T1_sign} holds, there exists a coupling $\mathbb{Q}$ of $(\mathcal{I}^u)_{u>0}$ and $\varphi$ such that
\begin{equation}
\label{eq:weakisom}
\text{$\mathbb{Q}$-a.s.,} \quad \mathcal{I}^{\frac{a^2}{2}} \subset \{ \varphi \geq -a \}, \text{ for all } a>0
\end{equation}
(this follows by adapting the result of \cite{MR2892408} to the cable system $\tilde{\mathcal{G}}$ and using continuity as first noted in \cite{MR3502602}).
The inclusion \eqref{eq:weakisom} hints at $\mathcal{I}^{\frac{a^2}{2}}$ typically forming the ``backbone'' of percolating clusters in $\{ \varphi \geq -a \}$, see \cite{Sz-16}. Correspondingly, our key estimates at scale $\xi=\xi(a)$, cf.~\eqref{eq:defxi}, assert that, if \eqref{eq:intro_Green}, \eqref{eq:intro_sizeball} and \eqref{eq:ellipticity} hold, one has (assuming $\nu\leq1$ to avoid unnecessary clutter)
\begin{align}
&\label{eq:squares} 
\inf_{a\in (0,1]} \P({B}_r\stackrel{\geq -a}{ \longleftrightarrow} \partial {B}_{4r})\big\vert_{r=  \xi(a)} \geq c,\\
&\label{eq:RSW} 
 \inf_{a\in (0,1]} \P(\text{LocUniq}(a,r))\big\vert_{r= \xi(a)} \geq c,
\end{align}
and the infima in question converge to $1$ in the limit $\lambda\to \infty$ upon choosing $r= \lambda \xi(a)$ instead, see \eqref{eq:lb2alphalarger2nubis} below; here, roughly speaking, $\text{LocUniq}(a,r)$ can be characterized through its complement, the ``absence of local uniqueness'' event $\text{LocUniq}(a,r)^{\mathsf{c}}$ that there exist two points in $\mathcal{I}^{\frac{a^2}{2}}\cap {B}_{2r}$ which are not connected by a continuous path of $\mathcal{I}^{\frac{a^2}{2}}$ within $ B_{4r}$ (see \eqref{eq:lb1} for the exact definition). 
The bound \eqref{eq:squares} follows readily by combining \eqref{eq:weakisom}, \eqref{eq:defRI} and the two-sided estimate $\text{cap}({B}_r)\asymp r^\nu$ for $r \geq 1$ (see (3.11) in \cite{DrePreRod2}), and $\xi(\cdot)$ given by \eqref{eq:defxi} emerges naturally~as
\begin{equation}
\label{eq:xiexplained}
e^{-u\text{cap}({B}_r)}\vert_{u=\frac{a^2}{2},r=\xi(a)} \stackrel{\eqref{eq:defxi}}{\asymp} 1.
\end{equation}
Our contribution is thus to obtain~\eqref{eq:RSW}, which follows from a sharp bound on  $\P(\text{LocUniq}(a,r)^{\mathsf{c}})$ ``in terms of $a^2r^{\nu}$,'' for $\nu \leq 1$ and more generally if $\alpha > 2\nu$, cf.~\eqref{eq:condonalphanu}, with logarithmic corrections when $\alpha = 2\nu.$ Estimates of this flavor, albeit non-optimal in $r$ and non-quantitative in~$u$, were first derived in \cite{MR2819660}. The precise estimate we obtain, which is of independent interest, is stated in Theorem~\ref{Thm:locuniq} below.

\bigskip

One can then attempt to give a complete overview of the critical exponents associated to the phase transition \eqref{eq:introKa}--\eqref{eq:intro_h_*} on the basis of scaling theory, the corresponding system of equations being now (over-)determined. We refer the reader to Sections 9.1--9.2 of \cite{MR1707339}, or Section 1 in \cite{kesten1987} regarding heuristics. Corollary \ref{T1}, see in particular \eqref{T1_beta}, implies that in very broad generality -- namely assuming \eqref{T1_sign} only (see also~\eqref{T1_signsat}), which guarantees that $a=0(=a_*)$ is critical, cf.~\cite{DrePreRod3} for a thorough investigation into the validity of this assumption -- one has
\begin{equation}
\label{eq:intro_beta}
\beta=1, 
\end{equation}
where $\beta$ is defined via 
\begin{equation} \label{eq:betaDef}
1- {\theta}_0(a-a_*)= 1- {\theta}_0(a) \sim c|a|^\beta, \quad \text{as } a \to 0^-,
\end{equation}
and $\sim$ means that the ratio of both sides tends to $1$ in the given limit (often, one more cautiously expects that $ \frac{\log(1- {\theta}_0(a-a_*))}{\log |a|}\to \beta$, see e.g.~(1.3) and~(1.8) in \cite{kesten1987}).
Under the assumption \eqref{eq:intro_Green}, it further follows from \eqref{eq:corrlength1} and \eqref{eq:psi_0nu=1} that 
\begin{equation}
\label{eq:intro_rho}
\rho=\frac2\nu, \text{ for } \nu \leq 1 \text{ and }\rho\in{\Big[\frac2\nu,2\Big]}\text{ for }\nu>1,
\end{equation}
where $\rho$ is the one-arm exponent at criticality, i.e.,
\begin{equation} \label{eq:rhoDef}
- \frac{\log \psi(0,\cdot)}{\log r}\to \frac1{\rho} \quad \text{as} \quad r \to \infty \quad \text{(with $\psi$ as in \eqref{eq:intro_psi}).}
\end{equation} 

Next, with correlation length exponent $\nu_c$ given by \eqref{eq:nuDeffirst}, see also \eqref{eq:defxi} in Theorem~\ref{T2}, the results of Corollary~\ref{C:scalingrelation} guarantee for $\nu \leq 1$ the existence of the volume exponent $\gamma$ near criticality defined by \eqref{eq:gamma1} (in fact, one would typically consider the limits $a \searrow 0$ and $a \nearrow 0$ separately) and \eqref{eq:scaling_gamma} is an instance of a scaling relation relating the exponents $\gamma$, $\nu_c$ and $\eta$ from \eqref{eq:2ptatcriticality}. Further to \eqref{eq:scaling_gamma}, scaling theory predicts the relations (in case of \eqref{eq:intro_hyperscaling} at least so long as $\alpha$ or $\nu$ remain below a certain upper-critical value)
\begin{align}
&\Delta=\delta \beta, \,  2-\alpha_{c}= \beta(\delta+1)=\gamma+2\beta, \quad \text{ (scaling)} \label{eq:intro_scaling}\\
&\alpha\nu_{c} = 2-\alpha_{c}, \, \alpha \rho=\delta+1, \quad \text{ (hyperscaling)}. \label{eq:intro_hyperscaling}
\end{align}
Here, $\beta$, $\nu_{c}$, $\eta$, $\rho$ and $\gamma$ have been introduced in \eqref{eq:betaDef}, \eqref{eq:nuDef}, \eqref{eq:2ptatcriticality}, \eqref{eq:rhoDef} and in \eqref{eq:gamma1}, respectively (see also \eqref{eq:deltaDef} below regarding $\delta$). We refer the reader to (1.2) and (1.5) of~\cite{kesten1987} concerning the quantities supposedly described by $\alpha_{c}$ and $\Delta$ in the context of Bernoulli percolation (assuming $\Delta_k=\Delta$ for all $k \geq2$ in the notation of \cite{kesten1987}, see also (9.7) in \cite{MR1707339}), and further to Chap.~9 of \cite{MR1707339} for an explanation of the heuristics behind \eqref{eq:intro_scaling} and \eqref{eq:intro_hyperscaling} on $\Z^d,$ $d\geq2.$ These readily generalize to any graph satisfying \eqref{eq:intro_sizeball} except for the informal derivation of the relation $\gamma= \nu_{c}(2-\eta),$ for which some control on the size of the boundary of a ball is needed. The heuristic behind Corollary~\ref{C:scalingrelation}, cf.~the proof of Proposition~\ref{P:scalinggamma}, indicates that this scaling relation should also hold for different percolation models on any graphs satisfying \eqref{eq:intro_sizeball}. 

Assuming all of the relations \eqref{eq:scaling_gamma}, \eqref{eq:intro_scaling} and \eqref{eq:intro_hyperscaling} to hold, the values of any two exponents are typically sufficient in order to determine a unique set of exponents. Feeding e.g.~\eqref{eq:intro_beta} and \eqref{eq:intro_rho} into \eqref{eq:scaling_gamma}, \eqref{eq:intro_scaling}, \eqref{eq:intro_hyperscaling} yields the following:\\
\begin{center}
\captionsetup[table]{position=top}
\begin{tabu}{ |c||c|c|c|c|c|c|c|c||c| } 
 \hline Exponent & $\alpha_{c}$ & $\beta$ & $\gamma$ & $\delta$ & $\Delta$  & $\rho$ & $\nu_{c}$ & $\eta$ & $\kappa$\\
 \hline
 Value & $2-\frac{2\alpha}\nu$ & 1$^{\ast}$  & $\frac{2\alpha}\nu-2$ & $\frac{2\alpha}\nu-1$&$\frac{2\alpha}\nu-1$&$\frac2\nu$&$\frac2\nu$&$\nu-\alpha+2^{\ast}$&$\frac12^{\ast}
 $\\
 \hline
\end{tabu}\\[1em]
{\footnotesize\captionof{table}{critical exponents as a function of the parameters $\nu \, (\leq 1)$ and $\alpha$ in \eqref{eq:intro_Green}, \eqref{eq:intro_sizeball}. Values with an asterisk hold without restriction on $\nu >0$ (and even in much greater generality, cf.~\eqref{T1_signsat}). 
\label{tb:exponents}}}
\end{center}
Some comments are in order. First of all, and crucially so, the values for $\nu_{c}$, $\eta$ and $\gamma$ thereby obtained are consistent with \eqref{eq:intro_nu} (and \eqref{eq:nuDeffirst}), \eqref{eq:2ptatcriticality} and \eqref{eq:scaling_gamma}. It is further remarkable that the exponents in Table~\ref{tb:exponents} are rational functions of $\alpha$ and $\nu$, and, in case the random walk is for instance diffusive -- that is if $\alpha=\nu+2,$ cf.~\eqref{UE} -- which applies e.g.~to $\Z^d,$ $d\geq3$ (with unit weights), all exponents can be expressed as functions of the sole parameter $\nu>0$ that governs the Green's function decay in \eqref{eq:intro_Green}. Moreover, in view of Corollary~\ref{C:captails}, we may add a ``capacity exponent'' $\kappa$ to the list, whence
\begin{equation}
\label{eq:capacityexp}
\P(\mathrm{{\rm cap}}({\mathcal{K}}^{0})\geq N ) \asymp N^{-\kappa}, \quad \text{as }N \to \infty, \quad \text{with }\kappa=1/2,
\end{equation}
as soon as the base graph $\mathcal{G}$ satisfies \eqref{T1_sign} and \eqref{eq:ellipticity}. Indeed, \eqref{eq:capacityexp} is obtained from the corresponding asymptotics for the random variable $\text{cap}(\tilde{\mathcal{K}}^{0})$, implied by \eqref{eq:captails} and valid under \eqref{T1_sign} only, using that $\ \frac{\mathrm{{\rm cap}}({\mathcal{K}}^{0})}{\mathrm{{\rm cap}}(\tilde{\mathcal{K}}^{0})} \in [\Cr{c:ellipticity}, 1]$ (with $\Cr{c:ellipticity}$ as in \eqref{eq:ellipticity}), which follows readily from \eqref{eq:balayage} below with the choices $K= {\mathcal{K}}^{0}$, $K'=\tilde{{\mathcal{K}}}^0$ upon integrating over $x\in\tilde{\G}$.

We now list one more consequence of the above results regarding the volume of~$\mathcal{K}^a$ at the critical point when $\nu\leq1.$
Recall that $|{\mathcal{K}}^a|=|\tilde{\mathcal{K}}^a \cap G|$, cf.~\eqref{eq:introKa}.
\begin{Cor}
\label{Cor:upperboundvolume} $(\nu \leq 1).$
If \eqref{eq:intro_Green}, \eqref{eq:intro_sizeball} and \eqref{eq:ellipticity} hold, there exists $c=c(\alpha, \nu) \in (0,\infty)$ such that, with $\tilde{c}= 1\{\nu=1\}/2$, one has
\begin{equation}
\label{eq:upperboundvolume}
\P(|\K^0|\geq n)\leq cn^{-\frac{\nu}{2\alpha-\nu}}\log(n)^{\tilde{c}}, \text{ for all $n \geq 1$}.
\end{equation}
\end{Cor}
In particular, assuming the existence of a volume exponent $\delta$ at criticality given by 
\begin{equation} \label{eq:deltaDef}
-\frac{\log(\P(|\K^0|\geq n))}{\log n}\rightarrow1/\delta \quad \text{as} \quad n \to \infty,
\end{equation}
we deduce from \eqref{eq:upperboundvolume} that $\delta\leq\frac{2\alpha}{\nu}-1$, for $\nu \leq 1$ (if $\delta$ exists). 
In view of the value for $\delta$ listed in Table~\ref{tb:exponents}, the upper bound in \eqref{eq:upperboundvolume} is thus presumably sharp up to logarithmic corrections. The bound \eqref{eq:upperboundvolume} follows readily by combining \eqref{eq:psi_0nu=1}, \eqref{eq:2ptatcriticality} and a first-moment argument. The short proof, given at the end of Section~\ref{sec:radiusLB}, is an adaptation of the argument giving Prop.~7.1 in~\cite{HS20}. We thank T. Hutchcroft for pointing out this reference to us.

We now discuss extensions of \eqref{eq:corrlength2lbpsi} to the regime $\nu>1$. Rather than working with $\psi$ defined in \eqref{eq:intro_psi} directly (but see Proposition~\ref{prop:lowerbounds}), we consider the function 
\begin{equation}
\label{eq:intro_psitilde}
\tilde{\psi}(a,r)\stackrel{\text{def.}}{=} \P\big(\{{B}_{\xi(a)} \stackrel{\geq a}{\longleftrightarrow} \partial_{\textnormal{in}} {B}_r\} \setminus \{ {B}_{\xi(a)} \stackrel{\geq a}{\longleftrightarrow} \infty\}\big),
\end{equation}
(depending on $\xi$ given by \eqref{eq:defxi}), where $B_r=B(0,r)$ refers to the discrete ball centered at $0$ in the metric $d$, cf.~above \eqref{eq:intro0}, $\partial_{\textnormal{in}} B_r=\{ x\in B_r: \exists\, y \in G\setminus B_r \text{ s.t.~}\lambda_{x,y}\neq 0 \}$ and with hopefully obvious notation $\{A \stackrel{\geq a}{\longleftrightarrow}  B\}$ is the event that $A$ and $B$ are connected by a path of open edges in the description given above \eqref{eq:cond_proba}, or equivalently by a (continuous) path in $\{ \varphi \geq a\}$. 
\begin{The}
\label{T:psitilde}
If \eqref{eq:intro_Green}, \eqref{eq:intro_sizeball}, \eqref{eq:ellipticity} and \eqref{eq:intro_shortgeodesic} hold, one has for all $|a| \leq c$ and $r\geq1,$ 
\begin{equation}
\label{eq:corrlength2lb}
\tilde{\psi}(a,r) \geq  \begin{cases}
 \tilde{\Cr{c1_xi}}\exp \Big\{- \tilde{ \Cr{C2_xi}}\frac{(\frac{r}{ \xi(a)})^{\nu}}{ (\log (\frac{r}{\xi(a)}\vee2))^{1\{\nu=1\}}} \Big\},  & \text{if $\nu\leq1,$}
 \\[0.8em]
  \tilde{\Cr{c1_xi}}\exp \big\{-  \tilde{\Cr{C2_xi}}\frac{r}{\xi(a)} (\log (\frac{r}{\xi(a)} \vee2))^{\Cl[c]{c:lbdefect2}} \big\}, &  \text{if $1< \nu < \frac{\alpha}{2},$}  \\[0.8em]
 \tilde{\Cr{c1_xi}}\exp \big\{-  \tilde{\Cr{C2_xi}}\frac{r}{\xi(a)} (\log (\frac{r}{\xi(a)} \vee2))^{\Cr{c:lbdefect2}} (\log  \xi )^{\Cl[c]{c:lbdefect4}} \big\}, &  \text{if $ \nu = \frac{\alpha}{2},$ $\frac{r}{\xi} \geq c (\log \xi)^{\frac{3}{\nu}}$,}
\end{cases}
\end{equation}
with $\xi(a)$ as in \eqref{eq:defxi} and where $\Cr{c:lbdefect2}= (1-1/\nu)(2\nu+1 +1\{\alpha=2\nu\})$ and $\Cr{c:lbdefect4}=2(1-\frac{1}{\nu})$. If $d=d_{\textnormal{gr}}$, \eqref{eq:corrlength2lb} remains valid when replacing $\tilde{\psi}(a,r)$ by $\tilde{\tau}^{\textnormal{tr}}_a(0,x)$, see \eqref{tautilde}, with $r\stackrel{\text{def.}}{=}d(0,x) \geq 1$.
\end{The}

When $1\leq\nu<\alpha/2,$ the bounds \eqref{eq:corrlength2lb} remain valid for $\psi$ in place of $\tilde{\psi}$ (as well as ${\tau}^{\textnormal{tr}}_a$), with the correct prefactor, if one assumes the lower bound in \eqref{eq:psi_0nu=1} to be sharp, see Proposition~\ref{prop:lowerbounds} below. Much as in \eqref{eq:gamma1}, one can consider a ``renormalized'' volume observable, which roughly speaking counts the number of balls of radius $\xi$ in an (approximate) tiling of $\tilde{\G}$ visited by $\tilde{\mathcal{K}}^a$, see \eqref{eq:volumerenorm} below. This quantity is expectedly of order unity in case $\xi$ is the correct correlation length scale in the problem. The bounds \eqref{eq:corrlength2lb} yield a lower bound of constant order uniform in $a$ as $a\to 0$, for all $0<\nu<\frac{\alpha}{2}$, see Remark~\ref{rk:1stscalingrelation},\ref{rk:1stscalingrelationvolumerenormalized}; see also Remark~\ref{rk:1stscalingrelation},\ref{rk:1stscalingrelationvolumerenormalizednot} for corresponding lower bounds on $\E[|\mathcal K^a|1\{|\mathcal{K}^a|<\infty\}]$ in the regime $1<\nu<\frac{\alpha}{2}$, which depend on the true behavior of $\psi(0,r)$ in \eqref{eq:psi_0nu=1} and yield a potentially sharp estimate on $\gamma$ in case the lower bound in  \eqref{eq:psi_0nu=1} is exact.

We now briefly return to matters regarding the correlation length for the present model. In view of Theorems \ref{T2} and \ref{T:psitilde}, one may expect off-critical bounds of the following form: under sensible assumptions on $\mathcal{G}$ (including, at the very least, \eqref{eq:intro_Green}) and for some functions $f_{\nu}: [1,\infty)\to [0,\infty)$ and ${\xi}': [-1,1] \setminus \{ 0\} \to (0,\infty)$, one has, for all $r\geq1$ and $|a|\leq c$ with $r/{\xi}'(a)\geq1$ (and even without the last restriction),
\begin{equation}
\label{eq:defcorrlength}
\psi(0,r) \exp \Big\{-cf_{\nu}\Big(\frac{r}{\xi'(a)}\Big)\Big\} \leq \psi(a,r) \leq \psi(0,r) \exp \Big\{-c'f_{\nu}\Big(\frac{r}{\xi'(a)}\Big)\Big\}.
\end{equation}
The correlation length is perhaps most intuitively defined as the quantity $\xi'(\cdot)$ satisfying \eqref{eq:defcorrlength} (assuming such a bound to hold), or a similar two-sided estimate for the truncated two-point function $\tau_a^{\textnormal{tr}}(0,x)$ from \eqref{eq:eq:deftauhf} instead of $\psi(a,r)$, with $d(0,x)$ in place of $r$ (or possibly a different distance function, intrinsic to $\mathcal{K}^a$). Associated to $\xi'$ is a correlation length exponent, which we define somewhat loosely to be such that 
\begin{equation} \label{eq:nuDef}
-\frac{\log(\xi'(a))}{\log(a)}\to\nu_{c} \quad \text{as} \quad a\to 0 \quad \text{(with $\xi'(\cdot)$ such that \eqref{eq:defcorrlength} holds)}
\end{equation}
(if this limit exists). We refer to Corollary~\ref{Corcorrlength} below which asserts that, assuming \eqref{eq:defcorrlength} to hold, \eqref{eq:nuDef} is consistent with \eqref{eq:nuDeffirst}, i.e.~$\xi' \asymp \xi$ for $\nu \leq 1$, and deduces from \eqref{eq:corrlength2} and \eqref{eq:corrlength2lb} that
\begin{equation}
\label{eq:intro_nu}
\nu_{c}\in{\Big[\frac2\nu,2\Big]}\text{ for }1 < \nu \le \alpha/2. 
\end{equation}

Finally, we note that the values in Table~\ref{tb:exponents} converge towards those corresponding to a mean-field regime as $\nu \uparrow 4$ and $\alpha\uparrow 6,$ which corresponds on $\Z^d$ to $d\uparrow 6$. In fact, one knows by \eqref{eq:2ptatcriticality} above and (1.16) in \cite{Aizenman:1984aa}, see also Exercise~4.1 in \cite{HeHo17},  that the triangle condition holds if $G=\Z^d$ when $d > 6.$ 
In view of \cite{Aizenman:1984aa, barsky1991} or Theorem~4.1 in \cite{HeHo17}, this indicates that $\beta=\gamma=1$ and $\delta=2$ likely hold for such~$d$, i.e.~these exponents expectedly attain their mean-field values, see also \cite{werner2020clusters} for related results. Note that if a mean-field regime is to appear for sufficiently large values of $\nu,$ it can only happen for $\nu\geq4$ by \eqref{eq:intro_rho}. 

We conclude by observing that knowing the values of $\rho,$ $\beta,$ $\nu_{c}$, $\eta$ as well as \eqref{eq:scaling_gamma} (roughly the status quo for $\nu \leq 1$), the scaling relations \eqref{eq:intro_scaling} alone are enough to obtain all remaining exponents $\alpha_{c},\delta$ and $\Delta$, and the hyperscaling relations \eqref{eq:intro_hyperscaling} are then automatically verified.

\medskip
The remainder of this article is organized as follows. In Section~\ref{S:diffform}, we derive certain key differential formulas, see Lemma~\ref{L:df2} and Corollary~\ref{C:df}, which are applied in Section~\ref{S:capandthetageneral} to deduce Theorem~\ref{T:cap} and Corollaries~\ref{T1} and~\ref{C:captails}. Section~\ref{sec:radius} concerns comparison estimates and upper bounds for the connectivity functions considered in Theorem~\ref{T2}. The outstanding lower bounds, e.g.~of \eqref{eq:corrlength2lbpsi} (part of Theorem~\ref{T2}) are split over Sections~\ref{sec:LB} and \ref{sec:lemmaProofs}. They rely on a sharp local uniqueness estimate (cf.~the discussion around~\eqref{eq:RSW}),  which is derived separately in Section~\ref{sec:radiusLB}, see in particular Theorem~\ref{Thm:locuniq} therein, which is of independent interest. The various pieces are gathered in Section~\ref{sec:denouement} to yield the proof of Theorem~\ref{T2}. Its various consequences, including the proofs of Corollary~\ref{C:scalingrelation}, and of Theorem~\ref{T:psitilde}, are presented at the end of Section~\ref{sec:denouement}.

Throughout, $c,c',\tilde{c}, \tilde{c}',\dots$ denote generic positive constants that change from place to place and may depend implicitly on the parameters $\alpha$ and $\nu$ in  \eqref{eq:intro_Green}, \eqref{eq:intro_sizeball}, whenever these conditions are assumed to hold (they also implicitly depend on the specific values of the constants $c,c'$ appearing in \eqref{eq:intro_Green}, \eqref{eq:intro_sizeball}, which we assume fixed once and for all). Numbered constants $c_1,c_2,\tilde{c}_1, \tilde{c}_2,\dots$ are defined upon first appearance in the text and remain fixed until the end.

\section{Differential formulas}
\label{S:diffform}

In this section, we develop certain formulas involving derivatives with respect to the parameter $a$ of fairly generic random variables of the excursion set $\{\varphi \geq a\}$ for the free field $\varphi$ on $\tilde{\G}$, cf.~\eqref{eq:introGFF}. We then specialize to functionals of the cluster $\tilde{\mathcal{K}}^a$ (recall~\eqref{eq:introKa}), see Lemma~\ref{L:df2} and Corollary~\ref{C:df} below. These results will play a central role in the sequel. 

It will be convenient to introduce an auxiliary geodesic distance $\tilde{d}$ on $\tilde{\mathcal{G}}$ attaching length $1$ to every cable of $\tilde{\mathcal{G}}$ (thus $\tilde{d}$ interpolates $d_{\text{gr}}$, the graph distance on $G$). We refer to topological properties of subsets of $\tilde{\mathcal{G}}$ below as relative to the topology induced by $\tilde{d}$ and denote by $\partial K$ the boundary of a set $K\subset \tilde{\mathcal{G}}$. Note that $\tilde{\mathcal{K}}^a$ is bounded in the sense defined above \eqref{eq:intro0} if and only if it is $\tilde{d}$-bounded. 

We now briefly review a few selected elements of potential theory for the diffusion $X$ under $P_x$ that will be needed below. For $U \subset \tilde{\mathcal{G}}$ open, we write $g_U$ for the Green function of $X_{\cdot}$ killed outside $U$, whence $g=g_{ \tilde{\mathcal{G}}}$ and the two are related by
\begin{equation}
\label{eq:killedGreen}
g_U(x,y)= g(x,y)- E_{x}[g(X_{T_U},y)1\{T_U< \infty\}], \quad x,y \in \tilde{\mathcal{G}},
\end{equation}
where $T_U= \inf\{ t \geq 0: X_t \notin U\}$ denotes the exit time from $U$. The identity \eqref{eq:killedGreen} is an immediate consequence of the Markov property. For compact $K\subset \tilde{\mathcal{G}}$, we write $e_{K}= e_{K,\tilde{\G}}$ for the equilibrium measure of $K$ relative to $\tilde{\G}$, which is supported on a finite set included in $\partial K$ (see for instance (2.16) in \cite{DrePreRod3} for its definition in the present context; we only add the subscript $\tilde{\G}$ to our notation in Sections~\ref{sec:radiusLB}--\ref{sec:denouement}, in which $\tilde{\G}$ and other cable systems are considered simultaneously, cf.~\eqref{eq:G_K}). Its total mass
 \begin{equation}
 \label{eq:cap}
 \text{cap}(K) \stackrel{\text{def.}}{=} \int {\rm d}e_K \ (< \infty)
 \end{equation}
 is the capacity of $K$. We now introduce the equilibrium potential $h_K(x)=P_x(H_K< \infty)$, for $x \in \tilde{\mathcal{G}}$, with $H_K = T_{\tilde{\G}\setminus K}= \inf\{t \geq 0 : X_t \in K \}$ denoting the entrance time of $X_{\cdot}$ in $K$, and more generally, for suitable $f:\tilde{\mathcal{G}} \to \R$,
 \begin{equation}
\label{eq:hdef}
h_K^f(x)\stackrel{\text{def.}}{=} E_x\big[f(X_{H_K})1\{ H_K< \infty\}\big], \text{ $x\in \tilde{\mathcal{G}}$}\quad \text{(whence $h_K=h_K^{f=1}$).} 
\end{equation} 
 By suitable extension of~(1.7) in \cite{MR3053773}, one obtains that
\begin{equation}
\label{eq:entranceGreenequi}
Ge_K = h_K \text{ on $\tilde{\G}$};
\end{equation}
here $G\mu(x) =\int g(x,y) \, {\rm d}\mu(y)$ is the potential of $\mu$, for a measure $\mu$ with compact support in $\tilde{\mathcal{G}}$.  For later purposes, we also record the following sweeping identity, see Section 2 of~\cite{DrePreRod3}, valid for compact sets $K, K' \subset \tilde{\mathcal{G}}$ with $K\subset K'$:
\begin{equation}
\label{eq:balayage}
P_{e_{K'}}(X_{H_K}=x,H_K<\infty)=e_K(x)\text{ for all }x\in\tilde{\G},
\end{equation}
where $P_{\mu}= \int P_x \, {\rm d}\mu(x).$ More generally, in view of \eqref{eq:hdef}, we obtain for suitable $f:\tilde{\G}\rightarrow\R$,
\begin{equation}
\label{eq:consequencebalayage}
\langle e_{K'},h_{K}^f \rangle = \langle e_K,f\rangle, \quad \text{for compact $K,K'$ with $K\subset  K'$},
\end{equation}
writing $\langle\mu,f\rangle =\int f{\rm d}\mu$ for the canonical dual pairing. We now introduce the linear functional
\begin{equation}
\label{eq:df2}
\begin{split}
M_K\stackrel{\text{def.}}{=} \langle e_K, \varphi \rangle, 
\end{split}
\end{equation}
which will play a central role in the sequel. Note that $M_K$ is Gaussian with mean $\E[M_K]=0$ and combining \eqref{eq:introGFF}, \eqref{eq:cap} and \eqref{eq:entranceGreenequi}, one finds that
\begin{equation}
\label{eq:df3}
\begin{split}
\E[M_K^2] = \text{cap}(K).
\end{split}
\end{equation}
We are interested in derivatives (with respect to a real parameter $a \in \R$) of random variables 
\begin{equation}
\label{eq:df0}
\begin{split}
&\text{$F_K^{(a)} = F_K^{(a)}(\varphi)$ with $F_K^{(a)}(\varphi)= F_K^{(0)}(\varphi -a)$ for all $a \in \R$}\\
&\text{and $\Vert F_K^{(0)}\Vert_{\infty} < \infty$, $F_K^{(0)}(\varphi)\in \sigma(1\{\varphi_x \geq 0 \}, x\in K)$},
\end{split}
\end{equation}
 where, with hopefully obvious notation, $\varphi -a$ refers to the field shifted by $-a$ in each coordinate, and $K \subset \widetilde{\mathcal{G}}$ is compact  and connected (for $\tilde d$). For such $K$, let $h^{-a}\equiv h_{K}^{f=-a}$, see \eqref{eq:hdef} for notation, so that
\begin{equation}
\label{eq:df1}
\begin{split}
&h^{-a}(x)=-ah(x), \text{ for }a\in \R, \, x \in \widetilde{\mathcal{G}} \quad \text{ (with }h(x)\equiv h_K(x)=P_x(H_K< \infty)\text{)}.
\end{split}
\end{equation}
One checks using \eqref{eq:df2}, \eqref{eq:df3} and applying the Cameron-Martin formula, see e.g.~\cite{janson_1997}, Theorem 14.1, that $\varphi+ h^{-a}$ has the same law under $\P$ as $\varphi$ under $\P_a$, where
\begin{equation}
\label{eq:df4}
\begin{split}
\frac{{\rm d}\P_a}{{\rm d}\P} =\exp \Big\{ -aM_K -\frac{a^2}{2}\text{cap}(K)\Big\}
\end{split}
\end{equation}
(to obtain this, one applies (14.6) in \cite{janson_1997} with the choice $\xi= -aM_K (\in L^2(\P))$, noting that $:e^{\xi}:$ is precisely the right-hand side of \eqref{eq:df4}, see also Theorem 3.33 in \cite{janson_1997}, and observing that, by means of (14.3) in \cite{janson_1997} and \eqref{eq:entranceGreenequi}, \eqref{eq:df1} above, one can rewrite $\rho_{\xi}(\varphi_{x})= \varphi_{x}-a\mathbb{E}[M_K\varphi_{x}]= \varphi_{x}-a(Ge_K)(x)= \varphi_{x} + h^{-a}(x),$ $x \in \tilde{\G})$. 
From this, one readily infers the following
\begin{Lemme}[under \eqref{eq:df0}]  \label{L:df1} 
\begin{align}
&\frac{{\rm d}}{{\rm d}a} \E\big[ F_K^{(a)}\big] = -\E\big[M_K \cdot F_K^{(a)} \big],\label{eq:deriv1}\\
&\frac{{\rm d}^2}{{\rm d}a^2} \E\big[ F_K^{(a)}\big] = \textnormal{Cov}_{\P}\big(M_K^2 , F_K^{(a)} \big) \label{eq:deriv2}.
\end{align}
\end{Lemme}

\begin{proof}
Regarding the first derivative, by \eqref{eq:df0} and \eqref{eq:df1}, one has that $ F_K^{(a)} (\varphi)=  F_K^{(0)}(\varphi-a) = F_K^{(0)}(\varphi+h^{-a}) $ since $h^{-a}=-a $ on $K$. Hence, applying a change of measure and using \eqref{eq:df4} gives 
 \begin{equation*}
\begin{split}
\frac{{\rm d}}{{\rm d}a} \E\big[ F_K^{(a)}\big] & =\frac{{\rm d}}{{\rm d}a} \E_a \big[F_K^{(0)}(\varphi)  \big]\\
&=\frac{{\rm d}}{{\rm d}a} \E \Big[\exp \Big\{ -aM_K -\frac{a^2}{2}\text{cap}(K)\Big\} F_K^{(0)}(\varphi)  \Big]\\
&= \E \Big[\big(-M_K -a \text{cap}(K) \big)\exp \Big\{ -aM_K -\frac{a^2}{2}\text{cap}(K)\Big\} F_K^{(0)}(\varphi)  \Big]\\
&=  \E_a \Big[\big(-  \langle e_K, \varphi-h^{-a} \rangle \big) F_K^{(0)}(\varphi)  \Big] 
= - \E \big[M_K \cdot F_K^{(0)}(\varphi+h^{-a})  \big] = -\E\big[M_K \cdot F_K^{(a)} \big].
\end{split}
\end{equation*}
Similarly, for the second derivative, one obtains from \eqref{eq:deriv1} and by change of measure
\begin{equation*}
\begin{split}
-\frac{{\rm d}^2}{{\rm d}a^2} \E\big[ F_K^{(a)}\big] &=
\frac{{\rm d}}{{\rm d}a} \E \Big[(M_K +a \text{cap}(K) )  \exp \Big\{ -aM_K -\frac{a^2}{2}\text{cap}(K)\Big\} F_K^{(0)}(\varphi)  \Big]\\
&= \E \Big[\big( \text{cap}(K)-(M_K +a \text{cap}(K) )^2  \big) \exp \Big\{ -aM_K -\frac{a^2}{2}\text{cap}(K)\Big\} F_K^{(0)}(\varphi)  \Big]\\
&=  \E_a \Big[\big(-  \langle e_K, \varphi-h^{-a} \rangle^2 +  \text{cap}(K) \big) F_K^{(0)}(\varphi)  \Big]\\
& =  \E \big[\big(- M_K^2 +  \text{cap}(K) \big) F_K^{(a)}(\varphi)  \big],
\end{split}
\end{equation*}
from which \eqref{eq:deriv2} follows on account of \eqref{eq:df3}.
\end{proof}
\begin{Rk}
Analogues of the differential equalities \eqref{eq:deriv1} and \eqref{eq:deriv2} hold for the (discrete) Gaussian free field on $G$. These can be obtained as direct consequences of  \eqref{eq:deriv1} and \eqref{eq:deriv2}, by considering $K$ an arbitrary finite subset of $G$ and noting that $\varphi$ extends the discrete free field on $G$. 
\end{Rk}
Whereas so far everything applies to $\mathcal{G}$ itself, the next calculation is specific to $\tilde{\G}$. For compact, connected $K\subset \tilde{\mathcal{G}}$ containing $0$ (cf.~\eqref{eq:intro0}), write
\begin{equation}
\label{eq:E_N}
\E_K[\, \cdot \, ]\stackrel{{\rm def.}}{=}\E[(\cdot)1\{\tilde{\mathcal{K}}^a \subset \ring{K} \}]
\end{equation}
(cf.\ \eqref{eq:introKa} for the definition of $\tilde{\mathcal{K}}^a$), where $\ring{K} = K \setminus\partial K$. Recall the strong Markov property of $\varphi$ (see e.g.~\cite[(1.19)]{Sz-16} for details): for $O\subset\tilde{\G}$ open, let $\mathcal{A}_O$ denote the $\sigma$-algebra $\sigma({\phi}_x,\,x\in{O})$. For compact $K\subset\tilde{\G}$ we consider $\mathcal{A}_K^+=\bigcap_{\eps>0}\A_{K^\eps}$, where $K^\eps$ is the open $\eps$-ball around $K$ for the distance $\tilde{d}.$ We define a (random) set $\K$ to be \textit{compatible} if $\K$ is a compact connected subset of $\tilde{\G}$ and $\{\K\subset O\}\in{\A_O}$ for any open set $O\subset\tilde{\G},$ and let
\begin{equation}
\label{eq:Markov1}
    \begin{split}
    \mathcal{A}_{\K}^+\stackrel{\text{def.}}{=}\big\{A\in{\A_{\tilde{\G}}}:\,A\cap\{\K\subset K\}\in{\A_K^+}\ &\text{for all compact~connected }K\subset\tilde{\G}  \text{ with } \ring{K} \neq \emptyset\big\}.
        \end{split}
\end{equation}
The Markov property then asserts that for any compatible $\K,$ conditionally on $\mathcal{A}_{\K}^+$,
\begin{equation}
\label{eq:Markov2}
    (\phi_x)_{x\in{\tilde{\G}}}\text{ is a Gaussian field with mean }h_{\K}^{\phi}\text{ and covariance }g_{\tilde{\mathcal{G}} \setminus \K},
\end{equation}
with $h_{\K}^{\phi}$ as defined in \eqref{eq:hdef} and $g_{\tilde{\mathcal{G}} \setminus \K}$ above \eqref{eq:killedGreen}. The following lemma is key. A useful variant can be found in Remark~\ref{R:L:df2bis},2) below. With regards to measurability below, recall that $\P$ in \eqref{eq:introGFF} refers to the canonical law of the Gaussian free field on the space $\R^{\widetilde{\mathcal{G}}}$, endowed with its canonical $\sigma$-algebra generated by the canonical coordinate maps $\varphi_x : \R^{\widetilde{\mathcal{G}}} \to \R,$ for $x \in \widetilde{\mathcal{G}}.$

\begin{Lemme}[$K \subset \tilde{\mathcal{G}}$ compact, connected, $0 \in K$] \label{L:df2} For all bounded $F:2^{\widetilde{\mathcal{G}}}\to \R$ such that $F(\emptyset)=0$ and $\varphi \mapsto F(\tilde{\mathcal{K}}^a(\varphi))1\{\tilde{\mathcal{K}}^a(\varphi) \textnormal{ bounded}\}$ 
is measurable for all $a \in \R $, one~has
\begin{align}
&\frac{{\rm d}}{{\rm d}a} \E_K\big[ F(\tilde{\mathcal{K}}^a) \big] = -a \E_K\big[\textnormal{cap}(\tilde{\mathcal{K}}^a)F(\tilde{\mathcal{K}}^a) \big]. \label{eq:deriv1.1} 
\end{align}
\end{Lemme}

\begin{Rk}
\label{R:cap_special}
The formulas \eqref{eq:deriv1.1} and \eqref{eq:deriv2.1} below indicate the special role played by the observable $\textnormal{cap}(\tilde{\mathcal{K}}^a)$, as derivatives of generic functionals $F(\tilde{\mathcal{K}}^a)$ under $ \E_K$ involve interaction terms between $F(\tilde{\mathcal{K}}^a)$ and the capacity functional.
\end{Rk}
\begin{proof}
Let 
\begin{equation}
\label{eq:Ka}
\tilde{\mathcal{K}}^a_K=\{x\in \tilde{\mathcal{G}}: 0 \leftrightarrow x \text{ in } \{ \varphi \geq a \} \cap K\}.
\end{equation}
 We will use the fact that, for any measurable function $f:\R\to \R$ with $f(M_{K}) \in L^1(\P)$ (see \eqref{eq:df2} for notation), one obtains the following as a consequence of the strong Markov property: for all $a \in \R$, $\P$-a.s.\ on the event $\{\varphi_0 \geq a \}$,
\begin{equation}
\label{eq:df5}
\E\big[ f(M_{K}) \, \big| \, \mathcal{A}_{\tilde{\mathcal{K}}^a_K}^+\big]= E\big[ f\big(\mathcal{N}(M_{\tilde{\mathcal{K}}^a_K}, \textnormal{cap}(K)- \textnormal{cap}(\tilde{\mathcal{K}}^a_K))\big) \big],
\end{equation}
where, conditionally on $\varphi$, $\mathcal{N}(\cdot,\cdot)$ is a Gaussian random variable with the given mean and variance under $E[\, \cdot \,]$. To deduce \eqref{eq:df5}, one observes that, on the respective event and conditionally on $ \mathcal{A}_{\tilde{\mathcal{K}}^a_K}^+$,  by \eqref{eq:Markov2} the random variable $M_{K}$ is Gaussian with mean (see \eqref{eq:hdef} for notation)
\begin{equation*}
\langle e_{K}, h_{\tilde{\mathcal{K}}^a_K}^{\varphi}  \rangle \stackrel{\eqref{eq:consequencebalayage}}{=}\langle e_{\tilde{\mathcal{K}}^a_K}, \varphi \rangle= M_{\tilde{\mathcal{K}}^a_K}
\end{equation*}
and variance (using the notation $(G_U \mu)(\cdot)= \int g_U(\cdot,x)\, {\rm d}\mu(x)$)
\begin{equation*}
\begin{split}
\big\langle e_{K}, G_{\tilde{\mathcal{G}}\setminus \tilde{\mathcal{K}}^a_K} e_{K}\big\rangle
& \stackrel{\eqref{eq:killedGreen}}{=} \big\langle e_{K}, G e_{K}\big\rangle- \big\langle e_{K},  E_{e_{K}}[g(\cdot, X_{H_{\tilde{\mathcal{K}}^a_K}})1\{H_{\tilde{\mathcal{K}}^a} < \infty\}] \big\rangle\\
& \stackrel{\eqref{eq:balayage}}{=}  \big\langle e_{K}, G e_{K}\big\rangle-  \big\langle e_{K}, G e_{\tilde{\mathcal{K}}^a_K}\big\rangle \stackrel{\eqref{eq:entranceGreenequi}, \eqref{eq:cap}}{=} \textnormal{cap}(K)- \textnormal{cap}(\tilde{\mathcal{K}}^a_K).
\end{split}
\end{equation*}
Moreover, since $\phi=a$ on the support of $e_{\tilde{\mathcal{K}}^a}$ (which is contained in $\partial \tilde{\mathcal{K}}^a$), on the event $\{\tilde{\K}^a\subset\ring{K},\phi_0\geq a\}$ we have that
\begin{equation}
\label{eq:df52}
M_{\tilde{\K}_K^a}\stackrel{ \tilde{\mathcal{K}}^a = \tilde{\mathcal{K}}^a_K}{=}
\langle e_{\tilde{\mathcal{K}}^a}, \phi  \rangle
\stackrel{\emptyset \neq \tilde{\mathcal{K}}^a}{=} \langle e_{\tilde{\mathcal{K}}^a}, a  \rangle
\stackrel{\eqref{eq:cap}}{=} a \textnormal{cap}(\tilde{\mathcal{K}}^a).
\end{equation}
With \eqref{eq:df5} and \eqref{eq:df52} at hand, one then obtains \eqref{eq:deriv1.1} by applying the formula \eqref{eq:deriv1} with the choice $F_{K}^{(a)}= F(\tilde{\mathcal{K}}^a)1\{\tilde{\mathcal{K}}^a \subset \ring{K} \}= F(\tilde{\mathcal{K}}^a)1\{\tilde{\mathcal{K}}^a \subset \ring{K}, \varphi_0 \geq a\} $ (the last equality holds since $F(\emptyset)=0$ by assumption), which satisfies \eqref{eq:df0}, by conditioning on $ \mathcal{A}_{\tilde{\mathcal{K}}^a_K}^+,$ using \eqref{eq:df5} with $f(x)=x$ and \eqref{eq:df52}, and noting that $F_{K}^{(a)}$ is $\mathcal{A}_{\tilde{\mathcal{K}}^a_K}^+$-measurable.   
 \end{proof}
\begin{Rk}\label{R:L:df2bis} 

 \begin{enumerate}[label=\arabic*)]
 \item Proceeding similarly as above, starting from \eqref{eq:deriv2} (for the same choice of $F_{K}^{(a)}$), using \eqref{eq:df5} and \eqref{eq:df52}, and observing that
\begin{equation*}
\textnormal{Cov}_{\P}\big(M_{K}^2 , F_{K}^{(a)} \big) \stackrel{\eqref{eq:df5},\eqref{eq:df52}, \eqref{eq:df3}}{=} \hspace{-1mm}\E\big[\big(  \textnormal{cap}(K)- \textnormal{cap}(\tilde{\mathcal{K}}^a) +a^2 \textnormal{cap}(\tilde{\mathcal{K}}^a)^2 \big) F_{K}^{(a)} \big]-  \textnormal{cap}(K)\cdot \E[F_{K}^{(a)}],
\end{equation*}
one deduces upon cancelling terms proportional to $  \textnormal{cap}(K)$, in view of \eqref{eq:E_N}, that
\begin{align}
&\frac{{\rm d}^2}{{\rm d}a^2} \E_K\big[ F(\tilde{\mathcal{K}}^a) \big] = \E_K\big[\textnormal{cap}(\tilde{\mathcal{K}}^a)\big(a^2 \textnormal{cap}(\tilde{\mathcal{K}}^a)-1 \big)F(\tilde{\mathcal{K}}^a) \big]. \label{eq:deriv2.1}
\end{align}

\item By slightly modifying the argument of Lemma \ref{L:df2}, one further obtains the following. Let $K  \subset \tilde{\mathcal{G}}$ be compact and connected, $0 \in K$ and $\tilde{\mathcal{K}}^a_K$ be as in \eqref{eq:Ka}.
For all $F:2^{\tilde{\mathcal{G}}}\to \R_+$ measurable such that $F(\emptyset)=0$ and $F(\tilde{\mathcal{K}}_K^a) \in L^1(\P)$ for all $a \in \R$, one has
\begin{align}
-\frac{{\rm d}}{{\rm d}a} \E\big[ F(\tilde{\mathcal{K}}^a_K) \big] &\geq a\E\big[\textnormal{cap}(\tilde{\mathcal{K}}^a_K)F(\tilde{\mathcal{K}}^a_K) \big] \label{eq:deriv1.1bis} \tag{\ref{eq:deriv1.1}'} \quad \text{and}\\
\frac{{\rm d}^2}{{\rm d}a^2} \E\big[ F(\tilde{\mathcal{K}}^a_K) \big] &\geq \E_K\big[\textnormal{cap}(\tilde{\mathcal{K}}_K^a)\big(a^2 \textnormal{cap}(\tilde{\mathcal{K}}_K^a)-1 \big)F(\tilde{\mathcal{K}}^a_K) \big], \text{ if $a>0$}. \label{eq:deriv2.1bis} \tag{\ref{eq:deriv2.1}'}
\end{align}
To obtain \eqref{eq:deriv1.1bis}, \eqref{eq:deriv2.1bis}, one proceeds as in the proof of Lemma~\ref{L:df2}, but in absence of the event $\{ \tilde{\mathcal{K}}^a \subset \ring{K} \} $, cf.~\eqref{eq:E_N}, the conditional mean $M_{\tilde{\K}^a_K}$ of $M_{K}$ given $ \mathcal{A}_{\tilde{\mathcal{K}}^a_K}^+$ on the event $\{\phi_0\geq a\},$ see \eqref{eq:df5}, verifies $M_{\tilde{\K}^a_K}=\langle e_{\tilde{\mathcal{K}}^a_K}, \varphi \rangle\geq a \textnormal{cap}(\tilde{\mathcal{K}}^a_K)$ since $\varphi \geq a$ on the support of $e_{\tilde{\mathcal{K}}^a_K}$ (part of $\partial \tilde{\mathcal{K}}^a_K$). 
\end{enumerate}
\end{Rk}

Next, we proceed to take the limit $K \nearrow \tilde{\mathcal{G}}$ under suitable assumptions. For $F$ satisfying the conditions of Lemma~\ref{L:df2}, we define  
\begin{equation}
\label{eq:psi_F}
\psi_F(a)=\E\big[ F(\tilde{\mathcal{K}}^a) 1\{ \tilde{\mathcal{K}}^a \, \text{bounded} \}\big], \quad  a \in \R.
\end{equation}
where boundedness is relative to $\tilde{d}$, see the beginning of this section. The following result will a-posteriori (once Theorem~\ref{T:cap} is proved) be strengthened under suitable assumptions on $\G$, see Corollary~\ref{C:df'} in the next section.
\begin{Cor}\label{C:df}
Let $I \subset \R$ be a closed interval, $\lambda_I$ denote the Lebesgue measure on $I$ and $F:2^{\tilde{\G}}\rightarrow\R$ be a function satisfying the assumptions of Lemma~\ref{L:df2}. If 
\begin{equation}
\label{eq:Cdf1}
Z_F \in L^1(\lambda_I \times \P), \text{ where } Z_F(a,\varphi)\stackrel{\textnormal{def.}}{=}-a \textnormal{cap}( \tilde{\mathcal{K}}^a(\varphi)) F( \tilde{\mathcal{K}}^a(\varphi)) 
1\{  \tilde{\mathcal{K}}^a(\varphi) \textnormal{ bounded}\},
\end{equation}
then for all $a,b \in I$, with $\phi_F(v)= \E[Z_F(v,\cdot)] $, one has
\begin{equation}
\label{eq:Cdf2}
\psi_F(b)-\psi_F(a)=\int_a^b  \phi_F(v) \, {\rm d}v.
\end{equation}
\end{Cor}
\begin{proof} Abbreviate $\psi\equiv \psi_F$, $\phi\equiv \phi_F$ and let $K_N \subset \tilde{\mathcal{G}}$ with $0 \in K_N$, $N \geq 0$, be an increasing sequence of compact sets exhausting $\tilde{\mathcal{G}}$. For each $N$, defining $\psi^{(N)}(a)=\E\big[ F(\tilde{\mathcal{K}}^a) 1\{ \tilde{\mathcal{K}}^a \subset \ring{K}_N \}\big]$ and $\phi^{(N)}(a)=-a \E\big[ F(\tilde{\mathcal{K}}^a) \text{cap}(\tilde{\mathcal{K}}^a) 1\{ \tilde{\mathcal{K}}^a \subset \ring{K}_N \}\big]$, one obtains for all $a,b \in I$, integrating \eqref{eq:deriv1.1} with $K = K_N,$
\begin{equation}
\label{eq:Cdf3}
\psi^{(N)}(b)-\psi^{(N)}(a)=\int_a^b  \phi^{(N)}(v) \, {\rm d}v.
\end{equation}
Since $\phi\mapsto F(\tilde{\K}^a)1\{\tilde{\K}^a\text{ bounded}\}\in{L^\infty(\P)}$ for all $a\in{I},$  in view of \eqref{eq:psi_F} one infers $\psi^{(N)}(a) \stackrel{N}{\longrightarrow}\psi(a)$ for all $a\in I$ by bounded convergence. One then uses that 
\begin{equation*}
\begin{split}
\Big|\int_a^b (\phi- \phi^{(N)})(v) \, {\rm d}v \Big| \leq & \,\E\Big[ \int_a^b {\rm d}v \, |v| F(\tilde{\mathcal{K}}^v) \text{cap}(\tilde{\mathcal{K}}^v) 1\big\{ \tilde{\mathcal{K}}^v \text{ bounded, } \tilde{\mathcal{K}}^v \cap (K_N)^{\mathsf{c}} \neq \emptyset \big\}\Big]\\[0.3em]
&\ \stackrel{N}{\longrightarrow} 0 \text{ (by \eqref{eq:Cdf1} and dom. convergence)}
\end{split}
\end{equation*}
in order to deduce \eqref{eq:Cdf2} from \eqref{eq:Cdf3} by passing to the limit.
\end{proof}
\begin{Rk}
One can formulate analogous conditions for \eqref{eq:deriv2.1} allowing to take the limit $K \nearrow \tilde{\mathcal{G}}$. The resulting formula is more delicate to manipulate, but instructive. Indeed, the minus sign present in \eqref{eq:deriv2.1} (and in the corresponding limiting formula) may cause cancellations; see Remark \ref{R:theta_0},2) in the next paragraph for an example.
\end{Rk}

\section{Cluster capacity and the function \texorpdfstring{${\theta}_0$}{theta0}}
\label{S:capandthetageneral}
As a first application of the above differential formulas, we prove Theorem~\ref{T:cap} and Corollaries~\ref{T1} and~\ref{C:captails}. It is now clear that $F(\tilde{\mathcal{K}}^a)\equiv f(\textnormal{cap}(\tilde{\mathcal{K}}^a))$ for suitable $f: \mathbb{R} \to \mathbb{R}$ looks to be a promising choice since \eqref{eq:deriv1.1} or \eqref{eq:Cdf2} yield an autonomous  system of differential equations  in $(a,\textnormal{cap}(\tilde{\mathcal{K}}^a))$. Moreover, as noted in Remark \ref{R:cap_special}, the utility of formulas such as \eqref{eq:deriv1.1}, \eqref{eq:deriv2.1} or \eqref{eq:Cdf2} for more general functionals $F(\cdot)$ largely rests on having access to information about the capacity functional.

A key ingredient is the following result. We recall that $g=g(0) (=\textnormal{cap}(\{0\})^{-1})$ and denote by $\mu_a$  the law (on $\{0\} \cup (g^{-1}, \infty)$) of the random variable ${\textnormal{cap}(\tilde{\mathcal{K}}^a) 1\{  \emptyset \neq  \tilde{\mathcal{K}}^{a} \textnormal{ bounded}  \}}$ under~$\P$.
\begin{Lemme} 
\label{L:capdifflevel}
For all $a,b \in \R$,
\begin{equation}
\label{eq:capdifflevel}
	\frac{  \mathrm{d}\mu_a}{ \mathrm{d}\mu_b}(t)= \exp\Big\{-\frac{(a^2-b^2)t}{2}\, \Big\}, \quad t\in (g^{-1},\infty). 
\end{equation}
\end{Lemme}

\begin{proof} We assume that $a,b \geq 0$. The case $a,b \leq 0$ is treated similarly, and the remaining cases follow by splitting the relevant interval at $0$. Consider
\begin{equation}
\label{eq:Fforcap}
F(\tilde{\mathcal{K}}^a)=1\{ \textnormal{cap}(\tilde{\mathcal{K}}^a) \in A \}, \ a \in \R,
\end{equation}
for $A \subset \R$, bounded, measurable, such that $\P( \textnormal{cap}(\tilde{\mathcal{K}}^a) \in A) >0$ and with $0 \notin A$. The latter implies that $F(\emptyset)=0$. Clearly, the map $\phi\mapsto F(\tilde{\K}^a)1\{\tilde{\K}^a\text{ bounded}\}\in{L^{\infty}(\P)}$ for all $a\in{\R}$ and
$|Z_F| \leq a\sup A$, whence \eqref{eq:Cdf1} is satisfied for any bounded interval $I$. Thus, Corollary \ref{C:df} applies, and \eqref{eq:Cdf2} yields that $\psi_F(a)= \P( \textnormal{cap}(\tilde{\mathcal{K}}^a) \in A, \, \tilde{\mathcal{K}}^a \textnormal{ bounded})$ is differentiable a.e.~in $a\in \R$, with derivative
\begin{equation}
\label{eq:cap1}
\frac{{\rm d}}{{\rm d}a} \P\big(\textnormal{cap}(\tilde{\mathcal{K}}^a) \in A, \, \tilde{\mathcal{K}}^a \textnormal{ bounded}\big)=-a\E\big[\textnormal{cap}(\tilde{\mathcal{K}}^a) 1\{ \textnormal{cap}(\tilde{\mathcal{K}}^a) \in A, \, \tilde{\mathcal{K}}^a \textnormal{ bounded}\}\big].
\end{equation}
Specializing to the case $A=(t-\varepsilon,t]$ for some $t> g^{-1}$ and $\varepsilon < t$, \eqref{eq:cap1} implies that
\begin{equation}
\label{eq:cap2}
-\frac{{\rm d}}{{\rm d} a} \log \mu_a \big((t-\varepsilon, t]\big)=a\E\big[\textnormal{cap}(\tilde{\mathcal{K}}^a) \, \big| \, t-\varepsilon < \textnormal{cap}(\tilde{\mathcal{K}}^a) \leq t, \, \tilde{\mathcal{K}}^a \textnormal{ bounded}\}\big] \in (a(t-\varepsilon),at],
\end{equation}
from which we infer
\begin{equation}
\label{eq:cap3}
 \mu_b \big((t-\varepsilon, t]\big)= \exp \Big( \int_{a}^b \frac{\mathrm{d}\log  \mu_v \big((t-\varepsilon, t]\big)}{\mathrm{d}v}  \  \mathrm{d}v \Big)  \cdot   \mu_a \big((t-\varepsilon, t]\big).
\end{equation}
Substituting the bounds \eqref{eq:cap2} into \eqref{eq:cap3} one obtains that, assuming without loss of generality that $b>a$,
$$
e^{-t \frac{(b^2-a^2)}{2}}\leq \frac{ \frac{1}{\varepsilon } \mu_b \big((t-\varepsilon, t]\big)}{ \frac{1}{\varepsilon } \mu_a \big((t-\varepsilon, t]\big)}\leq e^{-(t-\varepsilon) \frac{(b^2-a^2)}{2}}, \text{ for all } t>g^{-1}, \, \varepsilon < t,
$$
from which \eqref{eq:capdifflevel} follows by letting $\varepsilon \to 0$.
\end{proof}
 
 We now first give the 
 \begin{proof}[Proof of Theorem \ref{T:cap}]
For all $a \in \R$ and $u \geq 0$, changing levels from $a$ to $\sqrt{a^2+2u}$, one obtains that
\begin{equation*}
\begin{split}
\E\big[e^{-u\textnormal{cap}(\tilde{\mathcal{K}}^a)}1\{ \emptyset \neq \tilde{\mathcal{K}}^a \textnormal{ bounded} \}\big] &= \int_{g^{-1}}^{\infty} e^{-ut} \, \text{d}\mu_a(t) \\
&\hspace{-0.5em}\stackrel{\eqref{eq:capdifflevel}}{=}  \int_{g^{-1}}^{\infty} \text{d}\mu_{\sqrt{a^2+2u}}(t) = \P\big( \emptyset \neq \tilde{\mathcal{K}}^{\sqrt{a^2+2u}} \textnormal{ bounded} \big),
\end{split}
\end{equation*}
which entails \eqref{eq:lawcap_gen}. The identity \eqref{eq:lawcap} is then an immediate consequence of \eqref{eq:lawcap_gen} since $\E[e^{-u\textnormal{cap}(\tilde{\mathcal{K}}^a)}1\{ \tilde{\mathcal{K}}^a =\emptyset\}]= \P(\varphi_0 < a) =\Phi(a)$, \eqref{T1_sign} implies that $ \tilde{\mathcal{K}}^{\sqrt{2u+ a^2}}$ is bounded $\P$-a.s.~and $\P(  \tilde{\mathcal{K}}^{\sqrt{2u+ a^2}} \neq \emptyset) = \P(\varphi_0 \geq \sqrt{2u+ a^2}) =1-\Phi(\sqrt{2u+ a^2})$.
 \end{proof}

\begin{proof}[Proof of Corollary~\ref{C:captails}]
One has the identity, valid for all $u \geq 0$, $a\in \R$ (see Lemma 5.2 in \cite{DrePreRod3} for a proof), $\int_0^{\infty} \rho_a(t)e^{-ut} \, {\rm d}t = 1-\Phi(\sqrt{2u+a^2})$, where
 \begin{equation}
  \label{eq:rad_5}
\rho_a(t)=\frac1{2\pi}\frac{1}{t\sqrt{g(t-g^{-1})}}e^{-a^2t/2}1\{ t > g^{-1}\}.
 \end{equation}
 In view of \eqref{eq:lawcap}, one thus obtains from \eqref{eq:rad_5} that for all $a \in \R$,
  \begin{equation}
  \label{eq:rad_6}
\text{$\text{cap}( \tilde{\mathcal{K}}^a)$ has density $\rho_a(\cdot)$ under $\P((\, \cdot \, ),  \text{ $ \emptyset \neq \tilde{\mathcal{K}}^a$ bounded})$.}
 \end{equation}
 The tail estimate \eqref{eq:captails} then readily follows from \eqref{eq:rad_5}.
\end{proof}

\begin{Rk}
\label{R:cap}
 \begin{enumerate}[label=\arabic*)]
\item \label{R:cap.sign}By adapting the argument yielding Theorem~\ref{T:cap} above, one also obtains, without further assumption on $\mathcal{G}$, that for all $a \geq 0$,
\begin{equation}
\label{eq:cap4}
\text{cap}(\tilde{\mathcal{K}}^a)<\infty, \ \mathbb{P}\text{-a.s.},
\end{equation}
as implied by Theorem 3.1 in \cite{DrePreRod3}. In particular, together with  Lemma~3.4,2) of \cite{DrePreRod3}, \eqref{eq:cap4} readily yields that \eqref{T1_sign} holds on any vertex-transitive graph. We now briefly explain how to deduce \eqref{eq:cap4}. Rather than applying \eqref{eq:Cdf2} (which builds on \eqref{eq:deriv1.1}) with $F(\cdot)$ given by \eqref{eq:Fforcap} as in the proof of Theorem~\ref{T:cap}, one uses \eqref{eq:deriv1.1bis} with $F(\cdot)=1\{\text{cap}(\cdot) \in (s,t] \}$ for $g^{-1}\leq s < t< \infty$ (so that $F(\emptyset)=0$), to find instead of Lemma~\ref{L:capdifflevel} that
\begin{equation}
\label{eq:capfin1}
\P( s< \textnormal{cap}(\tilde{\mathcal{K}}_K^b) \leq t)\leq \P( s< \textnormal{cap}(\tilde{\mathcal{K}}_K^a) \leq t) \exp\Big\{-\frac{(b^2-a^2)s}{2}\, \Big\}, \text{ for }a< b,
\end{equation}
with $\tilde{\mathcal{K}}_K^a$ as defined in \eqref{eq:Ka}. Letting first $t \to \infty$, then $K \nearrow \tilde{\mathcal{G}}$ using monotonicity of $\text{cap}(\cdot)$ and finally $s \to \infty$ in \eqref{eq:capfin1} (say with $a=0$) yields \eqref{eq:cap4} for $a > 0$. To treat the case $a=0$, one uses \eqref{eq:capfin1} again with $s= g^{-1}$ and lets $K \nearrow \tilde{\mathcal{G}}$, then $t\to\infty$ and $b \downarrow 0$. The left-hand side of \eqref{eq:capfin1} thereby converges to $\P( \varphi_0 \geq 0)=\frac12$ and the right-hand side to $\P(\textnormal{cap}(\tilde{\mathcal{K}}^0) < \infty)- \P( \varphi_0 < 0)$. The claim \eqref{eq:cap4} for $a=0$ follows. 

\item We refer to our companion article \cite{DrePreRod3}, see in particular Theorem~3.9 therein, for an alternative approach to the above results by entirely different means; namely, exploiting a certain isomorphism theorem, due to \cite{Sz-16}, relating $\varphi$ and random interlacements on $\tilde{\G}$, which is shown in Theorem 1.1,2) of \cite{DrePreRod3} to hold under the sole assumption \eqref{T1_sign}, and turns out to be equivalent to~\eqref{eq:lawcap}.
\item Note that, if \eqref{T1_sign} holds, then by \eqref{eq:lawcap}
\begin{equation}
\label{eq:cap5}
\E\big[e^{-u\textnormal{cap}(\tilde{\mathcal{K}}^a)}\big]=\Phi(a)+1-\Phi(\sqrt{2u+ a^2}), \text{ for all } a \geq 0, \, u \geq 0.
\end{equation}
Assume on the contrary that \eqref{eq:cap5} holds. By \eqref{eq:lawcap_gen} (which always holds), \eqref{eq:cap5} can be equivalently recast as
\begin{equation}
\label{eq:cap6}
\E\big[e^{-u\textnormal{cap}(\tilde{\mathcal{K}}^a)}1\{ \tilde{\mathcal{K}}^a \textnormal{ unbounded} \}\big]= \P\big(  \tilde{\mathcal{K}}^{\sqrt{2u+ a^2}} \textnormal{ unbounded} \big).
\end{equation}
One readily deduces from \eqref{eq:cap6} with $a=0$ and \eqref{eq:cap4} that, if \eqref{T1_sign} does not hold, then $ \tilde{\mathcal{K}}^{\sqrt{2u}}$ is unbounded with positive probability for all $u \geq 0$, thus recovering the dichotomy $a_*\in \{0,\infty \}$ implied by Corollary 3.11 of \cite{DrePreRod3}.
\end{enumerate}
\end{Rk}

We now proceed with the
\begin{proof}[Proof of Corollary \ref{T1}]
Choosing $u = 0$ in \eqref{eq:lawcap} and observing that $\Phi(a)+1-\Phi(|a|)= 2\Phi(a\wedge 0)$, the claim \eqref{T1_theta_0} follows. The remaining conclusions are immediate consequences of \eqref{T1_theta_0} and the fact that $a_* \geq 0$, see above \eqref{T1_sign}.
\end{proof}

As a further consequence of Theorem \ref{T:cap} one obtains the following improvement of Corollary~\ref{C:df} under \eqref{T1_sign}. 

\begin{Cor}[Differential formula]\label{C:df'}
If \eqref{T1_sign} holds and $F$ satisfies the assumptions of Lemma~\ref{L:df2}, then \eqref{eq:Cdf2} holds for all $a,b \in \R$.
\end{Cor}
\begin{proof}
Taking derivatives in $u$ in \eqref{eq:lawcap} and setting $u=0$, one finds that
\begin{equation}
\label{eq:dens4}
\E[\text{cap}(\tilde{\mathcal{K}}^a)  1\{ \tilde{\mathcal{K}}^a \, \text{bounded} \} ]= \frac1{|a|} f(a), \text{ for all }a \in \R\setminus \{ 0\}.
\end{equation}
where $f(\cdot)=\Phi'(\cdot)$ denotes the density of $\varphi_0$. Hence,
\begin{align*}
\int_{\R}\E[| Z_F(a,\cdot )|] \, {\rm d}a \stackrel{\eqref{eq:Cdf1}}{\leq} \Vert F\Vert_{\infty}  \int_{\R}|a| \, \E\big[ \text{cap}(\tilde{\mathcal{K}}^a)  1\{ \tilde{\mathcal{K}}^a \, \text{bounded} \}\big] \, {\rm d}a 
\stackrel{\eqref{eq:dens4}}{=}  \Vert F\Vert_{\infty}  < \infty,
\end{align*}
i.e., $ Z_F \in L^1(\R \times \P)$ (in spite of the divergence in \eqref{eq:dens4} when $a\to 0$). Thus, condition~\eqref{eq:Cdf1} holds and the claim follows by applying Corollary~\ref{C:df}.
\end{proof}

 \begin{Rk}
 \label{R:theta_0}
 \begin{enumerate}[label=\arabic*)]
 \item One can alternatively deduce Theorem~\ref{T1} as an application of Corollary~\ref{C:df'}. Consider 
 \begin{equation}
\label{eq:dens1}
F( \tilde{\mathcal{K}}^a )= 1\{ \tilde{\mathcal{K}}^a \neq \emptyset\}  \stackrel{\eqref{eq:introKa}}{=}1\{ \varphi_0 \geq a\}
\end{equation}
(in particular $F(\emptyset)=0$), whence 
\begin{equation}
\label{eq:dens2}
{\theta}_0(a)\stackrel{\eqref{eq:intro_theta0}}{=} \P(\tilde{\mathcal{K}}^a \text{ is $(\tilde{d}$-)bounded}, \, \varphi_0 \ge a)+ \P(\varphi_0 < \, a ) \stackrel{\eqref{eq:psi_F}}{=}\psi_F(a)+ \Phi(a).
\end{equation}
By \eqref{eq:Cdf1}, \eqref{eq:dens4} and \eqref{eq:dens1}, one finds that $\E[Z_F(a,\cdot)]= - \frac{a}{|a|} f(a)$, for all $a \neq 0$, which extends to a piecewise continuous function of $a \in \R$. Thus, applying Corollary~\ref{C:df'}, which applies to $F$ in \eqref{eq:dens1}, it follows that for all $a \in \R$,
\begin{equation*}
\begin{split}
{\theta}_0(a)&\, \stackrel{ \eqref{T1_sign}}{=} 1+ {\theta}_0(a)- {\theta}_0(0)\stackrel{ \eqref{eq:dens2}}{=} 1+ \psi_F(a) -\psi_F(0) + \Phi(a)-\Phi(0)\\
&\stackrel{\eqref{eq:Cdf2}}{=} 1+ \int_0^a \big( -\frac{v}{|v|}f(v)\big) \, dv + \Phi(a)-\Phi(0) = 1+ \big(1-\text{sign}(a)\big)\big( \Phi(a)-\Phi(0) \big),
\end{split}
\end{equation*}
which is \eqref{T1_theta_0} ($\Phi(0)=\frac12$). 
\item One could also obtain \eqref{T1_theta_0} with the help of \eqref{eq:deriv2.1} (but using more information, i.e. the second moment $\E[\text{cap}(\tilde{\mathcal{K}}^a)^2]$, for $a>0$). Indeed, one can pass to the limit in \eqref{eq:deriv2.1} with $F$ given by \eqref{eq:dens1}. One then obtains, in view of \eqref{eq:psi_F} and \eqref{eq:dens2}, that for all $a \neq 0$,
\begin{equation}
\label{eq:dens5}
{\theta}_0''(a)=\psi_F''(a)+\Phi''(a)= \E[\text{cap}(\tilde{\mathcal{K}}^a) (a^2 \text{cap}(\tilde{\mathcal{K}}^a)-1 ) 1\{ \tilde{\mathcal{K}}^a \, \text{bounded} \} ]+ \Phi''(a).
\end{equation}
By means of \eqref{eq:lawcap}, one computes, for $a\neq0$, with $g=g(0,0)$,
\begin{align*}
\E[\text{cap}(\tilde{\mathcal{K}}^a)^2  1\{ \tilde{\mathcal{K}}^a \, \text{bounded} \} ]&= \frac{{\rm d}}{{\rm d}u}\Big(- f\big( \sqrt{2u+a^2} \big) \cdot \frac{1}{\sqrt{2u+a^2} } \Big)\Big|_{u=0}
\\&= f(a) \Big(\frac{1}{|a|^3}+\frac1g\cdot \frac1{|a|}\Big).
\end{align*}
From this and \eqref{eq:dens4}, one thus obtains in \eqref{eq:dens5}, noting that $\Phi''(a)=f'(a)=-\frac{a}{g}f(a)$, that
\begin{equation}
\label{eq:dens6}
{\theta}_0''(a)= f(a)  \Big(\frac{a^2}{|a|^3}+\frac1g\cdot |a|-\frac1{|a|}\Big) +  \Phi''(a)=\frac1gf(a)\big(|a|-a\big)
\end{equation}
for all $a \neq 0$, which readily gives \eqref{T1_theta_0}; one notes the perfect cancellation in \eqref{eq:dens6}.
\end{enumerate}
 \end{Rk}

 \section{Connectivity upper bounds}
 \label{sec:radius}

In this short section, we derive the upper bounds \eqref{eq:corrlength1}, \eqref{eq:corrlength2} on the truncated radius and two-point functions $\psi$ and $\tau_{a}^{\textnormal{tr}}$, introduced in \eqref{eq:intro_psi} and \eqref{eq:eq:deftauhf},
and even the full strength of \eqref{eq:corrlength1} in case of $\psi$. This corresponds to a certain choice of $F$ in \eqref{eq:psi_F}, see for instance \eqref{eq:rad_F} below. In one way or another, all the results of this section revolve around the idea of comparing with the cluster capacity observable and thus rely on the information supplied by Theorem~\ref{T:cap}, which was derived in the previous section. We also show how comparison with $\text{cap}( \tilde{\mathcal{K}}^a)$ immediately yields the estimates \eqref{eq:psi_0nu<1}, \eqref{eq:psi_0nu=1} on $\psi$ at criticality, together with bounds on the critical window, see Remarks~\ref{R:rhobounds} and~\ref{R:rad},\ref{R:rad.2} below.

We now introduce suitable balls on $\tilde{\G}$, which will be used throughout the remainder of this article. Recalling the discrete balls $B(x,r)\subset G$ (relative to $d$, cf.~above \eqref{eq:intro0}; note that these are not necessarily connected in nearest-neighbor sense), we define $\tilde{B}(x,r) \subset \tilde{\G}$ for $r\geq0$ and $x \in G$ as consisting of $B(x,r)$ and all the cables joining any pair of neighbors in $B(x,r)$ (i.e.~any $x,y \in B(x,r)$ s.t.~$\lambda_{x,y}>0$). We abbreviate $\tilde{B}_r = \tilde{B}(0,r)$. 
Since $B(x,r)$ is finite by assumption, the sets $B(x,r)$, $\tilde{B}(x,r)$, for $x\in G$, $r\geq0$, are compact in the sense  of Section~\ref{S:diffform} (see the beginning of that section). Moreover, whenever \eqref{eq:intro_Green} and \eqref{eq:ellipticity} hold, one knows by (2.8) of~\cite{DrePreRod2} that $ d(x,y) \leq \Cr{c:ball1} d_{\text{gr}}(x,y) $ for $x,y \in G$ hence $ B_{d_{\text{gr}}}(x,r) \subset B(x,\Cl[c]{c:ball1}r)$ for any $x \in {G}$ and $r > 0$ (here $B_{d_{\text{gr}}}(x,r)$, $x \in G$, $r \geq 0$, refers to the discrete ball with respect to $d_{\text{gr}}$ instead of $d$).  
 
 Throughout the remainder of this section, we assume that \eqref{eq:intro_Green} and \eqref{eq:ellipticity} are in force. Let $f_\nu: \R_+ \to \R_+$ be defined as $f_{\nu}(r)=r^{\nu}$ if $\nu<1$, $f_{\nu}(r)=\frac{r}{\log(r\vee 2)}$ if $\nu=1$ and $f_{\nu}(r)=r$ if $\nu >1$. One has the following inclusions.
 
 \begin{Lemme}[under \eqref{eq:intro_Green} and \eqref{eq:ellipticity}]\label{L:capcompare}
 For all $\nu>0$, there exist $\Cl[c]{c_radLB}, \Cl[c]{c_radUB}\in (0,\infty)$ depending on $\nu$ only such that for all $a \in \R$ and $r \geq1$, with $A(a,r)= \{ r\leq  \text{rad}( {\mathcal{K}}^a)< \infty\}$,
 \begin{align}
&A(a,r)\supset\big\{ \Cr{c_radLB}r^{\nu } \leq \textnormal{cap}(\tilde{\mathcal{K}}^a) < \infty \big\} , \label{eq:rad_1}\\
&A(a,r)\subset\big\{ \Cr{c_radUB}f_{\nu}(r) \leq \textnormal{cap}(\tilde{\mathcal{K}}^a) < \infty \big\} . \label{eq:rad_2}
 \end{align}
 \end{Lemme}
\begin{proof} Recalling the definition of $\text{rad}(\cdot)$ from below \eqref{eq:intro_psi}, if $ \text{rad}( {\mathcal{K}}^a) < r$, then $\tilde{\mathcal{K}}^a$ is included in $ \{ z \in G: d_{\text{gr}}(z, B(0,r))\leq 1\}$ union with all cables between neighboring pairs of points in this set. Thus, if $ \text{rad}( {\mathcal{K}}^a) < r$, then $\tilde{\mathcal{K}}^a \subset \tilde{B}_{r+ \Cr{c:ball1}}$ (cf.~above \eqref{eq:killedGreen} regarding $\Cr{c:ball1}$), hence by monotonicity of $\text{cap}(\cdot)$,
 \begin{equation}
  \label{eq:rad_3}
  \text{cap}( \tilde{\mathcal{K}}^a) \leq  \text{cap}( \tilde{B}_{(1+ \Cr{c:ball1})r}) \leq \Cr{c_radLB}r^{\nu}, \text{ for all $r \geq 1$,}
  \end{equation}
see for instance (3.11) in \cite{DrePreRod2} and (2.16) in \cite{DrePreRod3} regarding the last inequality, which relies solely on \eqref{eq:intro_Green} and \eqref{eq:ellipticity}. In the opposite direction, when  $ \text{rad}( {\mathcal{K}}^a) \geq r$, one has 
 \begin{equation}
 \label{eq:rad_4}
 \text{cap}( \tilde{\mathcal{K}}^a) \geq  \text{cap}( {\mathcal{K}}^a) \geq \inf_{\substack{A \subset G \text{ connected}\\\text{rad}(A) \geq r}} \text{cap}(A) \geq \Cr{c_radUB}f_{\nu}(r),
 \end{equation}
 see for instance Lemma~3.2 in \cite{DrePreRod2} regarding the last bound. Here, connectedness is meant with respect to $d_{\text{gr}}$, and $\mathcal{K}^a$ is connected by definition, see \eqref{eq:introKa}. Together, \eqref{eq:rad_3} and \eqref{eq:rad_4} also imply that $ \text{rad}( {\mathcal{K}}^a)=\infty$ if and only if $ \text{cap}( \tilde{\mathcal{K}}^a)=\infty$, and \eqref{eq:rad_1}, \eqref{eq:rad_2} follow.
 \end{proof}
 
 \begin{Rk}\label{R:rhobounds} As a first application of~\eqref{eq:rad_1}, \eqref{eq:rad_2} and Corollary~\ref{T1}, we deduce the bounds \eqref{eq:psi_0nu<1} and \eqref{eq:psi_0nu=1}. Using \eqref{eq:captails} with $a_N=0$, one first notes that for all $b >0$, $s \geq1,$
 \begin{equation}
   \label{eq:rad_7}
   \P(bs \leq   \text{cap}(\tilde{\mathcal{K}}^0) < \infty ) \asymp_{b}\int_{b s}^{\infty}t^{-3/2} \, {\rm d}t \asymp_{b} s^{-1/2},
 \end{equation}
 where $f\asymp_{b} g$ means that $cf \leq g \leq c'f$ for some constants $c,c'\in (0,\infty)$ depending only on $b$ and $\nu$. Together with \eqref{eq:rad_1} and \eqref{eq:rad_2}, the asymptotics \eqref{eq:rad_7} give  $\psi (0,r)\asymp r^{-\nu/2}$ when $\nu<1$ (recall the notation from \eqref{eq:intro_psi}), which is \eqref{eq:psi_0nu<1}. Similarly \eqref{eq:rad_1}, \eqref{eq:rad_2} and \eqref{eq:rad_7} yield \eqref{eq:psi_0nu=1} in case $\nu\geq 1$. 
 \end{Rk}
 
 Next, we give the
 \begin{proof}[Proof of \eqref{eq:corrlength1} and \eqref{eq:corrlength2}.]
 For all $a \in \R$, one has, for $b>0$, $r  \geq 1$ and $\nu>0$, using \eqref{L:capdifflevel}, \eqref{eq:rad_5} and \eqref{eq:rad_6},
  \begin{equation}
   \label{eq:rad_9}
   \begin{split}
  &e^{-b a^2 r^{\nu}} \P(b r^{\nu} \leq   \text{cap}(\tilde{\mathcal{K}}^0) < \infty ) \geq \P(b r^{\nu} \leq   \text{cap}(\tilde{\mathcal{K}}^a) < \infty ) \\
  &  \quad \geq e^{- 2 b a^2 r^{\nu}}  \int_{b r^{\nu}}^{2b r^{\nu}}\rho_0(t) \, {\rm d}t \geq c(b,\nu)  e^{- 2 b a^2 r^{\nu}}  \P(b r^{\nu} \leq   \text{cap}(\tilde{\mathcal{K}}^0) < \infty ),
  \end{split}
   \end{equation}
where we also used \eqref{eq:rad_7} in the last step. From \eqref{eq:rad_1},  \eqref{eq:rad_2} and \eqref{eq:rad_9}, together with \eqref{eq:psi_0nu<1} one readily deduces the lower bound in \eqref{eq:corrlength1}, and also the upper bound if one allows for a constant $c(>1)$ in front of $\psi(0,r)$. Such direct comparisons fail to yield the right order for both upper and lower bound when $\nu \geq 1$, see Remark \ref{R:rad},\ref{R:rad.1} below. 

We now give an argument which yields the desired upper bounds in \eqref{eq:corrlength1} and \eqref{eq:corrlength2}. For $\kappa >0$, we consider the function (for arbitrary $\nu > 0$)
\begin{equation}
   \label{eq:rad_10}
\tau_{\kappa}(a)=e^{\kappa a^2 f_{\nu}( r )}\psi(a,r), \text{ for $a \in \R$}
\end{equation}
(which implicitly depends on $r>0$)
with $f_{\nu}$ as defined above \eqref{eq:rad_1}. We will show the following simple result.

\begin{Lemme}
\label{L:rad_UB}
$(\nu > 0)$. There exists $\kappa_1(\nu)> 0$ such that, if $\kappa \in (0, \kappa_1]$, with $ \tau_{\kappa}'= \frac{{\rm d}}{{\rm d}a}  \tau_{\kappa}$,
\begin{equation}
\label{eq:rad_11}
\textnormal{sign}(a) \tau_{\kappa}'(a) \leq 0, \text{ for a.e. }a \in \R.
\end{equation}
(In particular $ \tau_{\kappa}$ is a.e. $C^1$ on $\R$).
\end{Lemme}
By integrating the differential inequality  \eqref{eq:rad_11} between $0$ and $a \in \R,$ one immediately deduces in view of \eqref{eq:rad_10} that $\psi(a,r) \leq \psi(0,r) e^{- \kappa_1 a^2 f_{\nu}( r )}$, from which the upper bounds in \eqref{eq:corrlength1} and \eqref{eq:corrlength2} follow since $a^2r^{\nu}= (r/\xi(a))^{\nu}$. It thus remains to give the
\begin{proof}[Proof of Lemma \ref{L:rad_UB}] 
  We consider 
 \begin{equation}
 \label{eq:rad_F}
 F(\tilde{\mathcal{K}}^a)=1_{A(a,r)}, \quad \text{ recalling that } A(a,r)=\{ r\leq  \text{rad}( {\mathcal{K}}^a)< \infty\}, \text{ for }r>0,\, a \in \R,
 \end{equation}
and study the corresponding observable $ \psi_F(a)=\psi(a,r)$ (see \eqref{eq:intro_psi} and \eqref{eq:psi_F} for notation). 

One first observes, using \eqref{eq:dens4}, that the condition \eqref{eq:Cdf1} is satisfied with $I=\R$ for $F$ given by \eqref{eq:rad_F}. Moreover, since $F$ is bounded and $F(\emptyset)=0$, \eqref{eq:Cdf2} applies and one deduces that for (almost) all $a \in \R \setminus \{0\}$,
\begin{equation}
\label{eq:rad_12}
\frac{{\rm d}}{{\rm d}a}\psi(a,r)= \E[Z_F(a,\cdot)]=-a \E[ \text{cap}(\tilde{\mathcal{K}}^a) 1\{ r\leq  \text{rad}( {\mathcal{K}}^a)< \infty \}].
\end{equation}
Hence, for all $\kappa >0$ and a.e. $a >0$,
\begin{equation}
\label{eq:rad_13}
\begin{split}
\tau_{\kappa}'(a)&= a\big(2 \kappa f_{\nu}(r) \psi(a,r) -   \E[ \text{cap}(\tilde{\mathcal{K}}^a) 1\{ r\leq  \text{rad}( {\mathcal{K}}^a)< \infty \}] \big)e^{\kappa a^2 f_{\nu}( r )}\\
&\leq a\lambda \big( \psi(a,r) -   \P(  \text{cap}(\tilde{\mathcal{K}}^a) \geq \lambda , \, r\leq  \text{rad}( {\mathcal{K}}^a)< \infty ) \big)e^{\kappa a^2 f_{\nu}( r )},
\end{split}
\end{equation}
where $\lambda = 2 \kappa f_{\nu}(r)$. But due to \eqref{eq:rad_2}, one knows that $ \text{rad}( {\mathcal{K}}^a) \geq r$ implies $ \text{cap}(\tilde{\mathcal{K}}^a) \geq \lambda$
whenever $\kappa \leq \kappa_1 \stackrel{\text{def.}}{=} \Cr{c_radUB}/2$, whence \eqref{eq:rad_13} gives $\tau_{\kappa}'(a)\leq0$ for almost all $a>0$ and \eqref{eq:rad_11} follows by symmetry. 
\end{proof}
With Lemma~\ref{L:rad_UB} shown, the proof of \eqref{eq:corrlength1} and \eqref{eq:corrlength2} is complete.
\end{proof}

 \begin{Rk}
 \label{R:rad}
 \begin{enumerate}[label=\arabic*)]
 \item\label{R:twopointUB} We briefly describe how to adapt the above arguments to yield the versions of the upper bounds in  \eqref{eq:corrlength1} and \eqref{eq:corrlength2} for $\tau_a^{\textnormal{tr}}$. For any $x\in{G\setminus B_r},$ defining $\hat{A}(a,x)=\{0\stackrel{\geq a}{\longleftrightarrow} x, {\mathcal{K}}^a\text{ bounded}\}$ the inclusion \eqref{eq:rad_2} still holds when replacing $A(a,r)$ by $\hat{A}(a,x)$ (indeed $\hat{A}(a,x)\subset A(a,r)$). Hence, mimicking the proof of  \eqref{eq:rad_11}, but using $\hat{F}=1_{\hat{A}(a,x)}$ instead of $F$, cf.~\eqref{eq:rad_F}, one finds
 that $\textnormal{sign}(a) \hat{\tau}_{\kappa}'(a) \leq 0$ for $\kappa$ small enough, where $\hat{\tau}_{\kappa}(a)=e^{\kappa a^2f_{\nu}(r)}\tau_a^{\textnormal{tr}}(a,x),$ and the analogues for $\tau_a^{\textnormal{tr}}(a,x)$ of \eqref{eq:corrlength2} and of the upper bound in \eqref{eq:corrlength1} readily follow. Note that, as opposed to $\psi$, we do not claim here the version for $\tau_a^{\textnormal{tr}}$ of the (off-critical) lower bounds in \eqref{eq:corrlength1} asserted as part of Theorem~\ref{T2}. These will be supplied, along with the proofs of \eqref{eq:corrlength2lbpsi} and \eqref{eq:corrlength2lb}, by a separate argument in Section~\ref{sec:denouement}. 
\item Proceeding similarly as in Lemma \ref{L:rad_UB}, but using \eqref{eq:rad_1} instead of \eqref{eq:rad_2}, one can easily prove that for all $\nu>0$ there exists $\kappa_2(\nu)<\infty$ such that for all $\kappa\geq \kappa_2$
\begin{equation}
\label{eq:rad_112}
\textnormal{sign}(a) \tilde{\tau}_{\kappa}'(a) \geq 0, \text{ for a.e. }a \in \R,
\end{equation}
where $\tilde{\tau}_{\kappa}(a)=e^{\kappa a^2r^{\nu}}\psi(a,r).$ This directly implies that $\psi(a,r)$ and $\psi(0,r)$ are of the same order when $r\leq t\xi,$ for any choice of $\nu>0$ and $t>0.$ Indeed, for $\kappa\geq \kappa_2(\nu),$
\begin{equation}
\label{eq:rad_113}
	\psi(0,r)\geq \psi(a,r)= \tilde{\tau}_{\kappa}(a)e^{-\kappa a^2r^{\nu}}\geq c\tilde{\tau}_{\kappa}(a) \stackrel{\eqref{eq:rad_112}}{\geq} c\tilde{\tau}_{\kappa}(0)=c\psi(0,r)
\end{equation}
for all $r\geq1$ and $a\in\R$ with $r\leq t\xi.$

 \item \label{R:rad.1} When $\nu \geq1$, \eqref{eq:rad_1} and \eqref{eq:rad_9} yield the lower bound $\psi(a,r) \geq e^{-ca^2 r^{\nu}}\P(c' r \leq   \text{cap}(\tilde{\mathcal{K}}^0) < \infty )$, which does not exhibit the desired leading exponential order, cf.\ Corollary \ref{Corcorrlength}. Regarding the upper bound, one has, for $a \neq 0$, $r >0$,
 \begin{equation*}
 \begin{split}
 \psi(a,r) &\stackrel{ \eqref{eq:rad_2}}{\leq} \P( \Cr{c_radUB}f_{\nu}(r) \leq   \text{cap}(\tilde{\mathcal{K}}^a) < \infty ) \stackrel{ \eqref{eq:rad_5}}{\leq} e^{-  \Cr{c_radUB} a^2 f_{\nu}(r)}  \P( \Cr{c_radUB}f_{\nu}(r) \leq   \text{cap}(\tilde{\mathcal{K}}^0) < \infty ) \\& \stackrel{ \eqref{eq:rad_1}}{\leq} e^{-  \Cr{c_radUB} a^2 f_{\nu}(r)} \psi\big(0, cf_{\nu}(r)^{\frac1\nu}\big)  \stackrel{\eqref{eq:psi_0nu=1}}{\leq} c(\log( r \vee 2) )^{1_{\nu=1}}r^{\frac12(\nu-\frac1\nu)}e^{-  \Cr{c_radUB} a^2 f_{\nu}(r)} \psi(0,r),
 \end{split}
 \end{equation*}
 which has the correct exponential order, cf. \eqref{eq:corrlength2}, but is only pertinent sufficiently ``far away'' from criticality, i.e.\ in the regime of parameters $ c a^2 f_{\nu}(r) \geq  \log\log ( r \vee 2) $ when $\nu=1$ or $ca^2f_{\nu}(r)\geq \log r$ if $\nu>1$ (rather than $ c a^2 f_{\nu}(r) \geq 1$).
\item \label{R:rad.2} (Critical window). Suppose \eqref{eq:intro_Green} and \eqref{eq:ellipticity} hold. If $\nu < 1$ then
\eqref{eq:corrlength1} implies in particular that 
$$
\frac{\psi(a,r)}{\psi(0,r)} \to 0 \text{ if and only if } |a|r^{\nu/2} \to \infty.
 $$
In case $\nu=1$ and $a>0$, one can deduce good bounds on the critical window as follows: $\psi(a,r)\geq c\psi(0,r)(\geq cr^{-1/2})$ if $r\leq a^{-2}$ on account of \eqref{eq:rad_113} and \eqref{eq:psi_0nu=1}, and $\psi(a,r)/\psi(0,r)\to 0$ as $r/(\xi(a)\log(r))\to \infty$ on account of \eqref{eq:corrlength2}. In particular, in case $G=\Z^3$ (with unit weights), this improves on the bounds (12) and (13) from Theorem 6 in \cite{MR4112719}. Similarly, in the supercritical regime $a<0$ (cf.~(14) and (15) in \cite{MR4112719}), using that $\P(0 \stackrel{ \geq a}{\longleftrightarrow} \partial B_r) \leq \psi(0,r) + (1-\theta_0(a))$, for all $r >0$, one finds with the help of \eqref{T1_beta} that
$$
\P(0 \stackrel{ \geq a}{\longleftrightarrow} \partial B_r) \leq \psi(0,r)+ ca \ \Big( \stackrel{\eqref{eq:psi_0nu=1}}{\leq} c'\big(\frac{\log r}{r}\big)^{1/2}\Big), \text{ if $\frac{r}{\log(r)} \leq \xi(a)$, $a \in[-1,0]$,} 
$$
and similarly since $\P(0 \stackrel{ \geq a}{\longleftrightarrow} \partial B_r) \geq 1-\theta_0(a) $ that $\frac{\P(0 \stackrel{ \geq a}{\longleftrightarrow} \partial B_r)}{\P(0 \stackrel{ \geq 0}{\longleftrightarrow} \partial B_r)} \geq \frac{c|a|}{\psi(0,r)} \to \infty$ for $a \in[-1,0]$ as $r/(\xi(a)\log(r))\to \infty$  using \eqref{T1_beta} and \eqref{eq:psi_0nu=1}.
\item \label{R:rad.3} (The condition \eqref{eq:ellipticity}). The only place \eqref{eq:ellipticity} entered the proof of \eqref{eq:corrlength1} and \eqref{eq:corrlength2} is through Lemma~\ref{L:capcompare}, specifically to obtain \eqref{eq:rad_3} and \eqref{eq:rad_4}. Inspecting the proofs of (3.11) and of Lemma 3.2 in \cite{DrePreRod2} shows that  \eqref{eq:ellipticity} is only used to deduce that (cf.~(2.8) in \cite{DrePreRod2})
\begin{equation}\label{eq:ellipticity'}
d \leq c d_{\textnormal{gr}}, 
\end{equation}
Thus, Lemma~\ref{L:capcompare}, as well as \eqref{eq:corrlength1}, \eqref{eq:corrlength2} continue to hold upon replacing \eqref{eq:ellipticity} by \eqref{eq:ellipticity'}. Condition~\eqref{eq:ellipticity'} may be better suited to deal with examples $(G,\lambda)$ in which one tinkers more severely with the conductances (indeed the requirements \eqref{eq:intro_Green} and \eqref{eq:ellipticity} imply a uniform lower ellipticity bound $\lambda_{x,y} \geq c$, see (2.10) in \cite{DrePreRod2}).
  \end{enumerate}
 \end{Rk}

 \section{Local uniqueness at the critical scale}
 \label{sec:radiusLB}

  We now derive a suitable local uniqueness estimate at scale $\xi$, cf.~\eqref{eq:defxi}, which in particular will imply bounds like \eqref{eq:RSW}, see Corollary~\ref{C:locuniq} below. This estimate really concerns connections in the interlacement set $\mathcal{I}^u$, $u>0$, from which useful results for $\varphi$ can be gleaned by means of the coupling in \eqref{eq:weakisom}. Its general form (quantitative in the parameter $u>0$ and a generic length scale $R \geq 1$) is stated in Theorem~\ref{Thm:locuniq}. Weaker results of this kind have been derived on $\Z^d,$ $d\geq3$, in Proposition 1 of \cite{MR2819660}, see also Lemma~3.2 in~\cite{DrePreRod} for a quantitative bound in $u$ valid in the regime $R \geq u^{-\frac{1}{\varepsilon}}$ for $\varepsilon \ll 1$, and extended to any graph satisfying \eqref{eq:intro_Green}, \eqref{eq:intro_sizeball} and \eqref{eq:ellipticity} in Section~4 of \cite{DrePreRod2}. All these bounds however, are too weak for our purpose, notably because they do not cover the regime of scale $R\approx u^{-\frac1{\nu}}$ when $u\ll 1$, which corresponds to $R\approx \xi$ in view of  \eqref{eq:weakisom} and \eqref{eq:defxi}. The scale $ u^{-\frac1{\nu}}$ forms a natural barrier, being the smallest radius for which balls become ``visible'' for an interlacement trajectory in $\mathcal{I}^u$, cf.~\eqref{eq:RIdefG_K} and \eqref{eq:capBd} below.

In the sequel we tacitly assume that $K \subset \tilde{\G}$ is a compact set. For such $K$, let 
\begin{equation} 
\label{eq:G_K}
\tilde{\G}_K\stackrel{\text{def.}}{=} \text{the unbounded connected component of } \tilde{\G} \setminus K
\end{equation} 
(see \eqref{choicec18} below regarding its uniqueness). The following results, in particular Theorem~\ref{Thm:locuniq} below, are of independent interest, already in case $K=\emptyset$ (whence $\tilde{\G}_K= \tilde{\G}$). 
For the purposes we have in mind, the removal of $K$ in \eqref{eq:G_K} should be thought of as corresponding to the exploration of part of the cluster $\tilde{\mathcal{K}}^a$ in \eqref{eq:introKa}. In view of the Markov property \eqref{eq:Markov2} for the free field, the explored region $K$ will effectively act as a Dirichlet boundary condition for $X$. 
Accordingly, we consider $\overline{\P}_{\tilde{\G}_K}$, the canonical law of the interlacement process on $\widetilde{\mathcal{G}}_K$ and $\mathcal{I}^u \subset \widetilde{\mathcal{G}}_K$ the interlacement set at level $u>0$, whose distribution is characterized by the property that
\begin{equation}
\label{eq:RIdefG_K}
\overline{\P}_{\tilde{\G}_K}(\mathcal{I}^u \cap C = \emptyset)= \exp\{ -u \mathrm{cap}_{\tilde{\G}_K}(C)\} , \text{ for all compact $C \subset \tilde{\G}_K$}
\end{equation}
(see the paragraph following Corollary~\ref{C:locuniq} regarding $\mathrm{cap}_{\tilde{\G}_K}(\cdot)$, see also Section~2.5 of \cite{DrePreRod3} for the definition of $\overline{\P}_{\tilde{\G}_K}$ in this context). In particular, \eqref{eq:RIdefG_K} reduces to \eqref{eq:defRI} when $K=\emptyset$.

Let $\widehat{\mathcal{I}}^u$ denote the set of edges of $G$ traversed entirely by at least one of the trajectories in the support of the interlacement point process at level $u.$ For $z\in{G}$ and $R>0,$ let $B_E(z,R)$ refer to the set of edges of $G$ whose endpoints are both contained in $B(z,R).$ For $z \in G$ as well as $u, R >0$ and $\lambda>1$, we introduce the event
 \begin{equation}
 \label{eq:lb1}
 \text{LocUniq}_{u,R,\lambda}(z) \stackrel{\text{def.}}{=}
  \bigcap_{x,y \in \mathcal{I}^u \cap B(z,R)}\{ x \leftrightarrow y \text{ in }  \hat{\mathcal{I}}^u \cap {B}_E(z,\lambda R) \}.
 \end{equation}
Note that the event in \eqref{eq:lb1} implies a ``local uniqueness'' for interlacements both on the discrete graph $G$ and on the cable system, in that if $\{ x \leftrightarrow y \text{ in }  \hat{\mathcal{I}}^u \cap {B}_E(z,\lambda R) \}$ occurs, then $x$ and $y$ are connected by both a discrete path in $\mathcal{I}^u\cap B(z,\lambda R)$ and a continuous path in $\I^u\cap\tilde{B}(z,\lambda R).$  
 
 \begin{The}[under \eqref{eq:intro_Green}, \eqref{eq:intro_sizeball}, \eqref{eq:ellipticity}]
\label{Thm:locuniq} 
 \medskip
\noindent There exist $\Cl[c]{clocuniq} \in (0,1),$ $\Cl[c]{c:locuniq1}, \Cl[c]{cvalphagnuK} \in (1,\infty)$ such that for all $u > 0$, $R\geq 1,$ compacts $K\subset \tilde{B}(0,R)$ and $z \in G$ with ${d}(z,0)\geq\Cr{cvalphagnuK}R\cdot 1\{K \neq \emptyset \},$ 
  \begin{equation}
 \label{eq:lb2alphalarger2nu}
\overline{\P}_{\tilde{\G}_K}\big(\textnormal{LocUniq}_{u, R}(z)^{\mathsf{c}}\big) \leq c\exp\bigg\{ -\Big(\frac{\Cr{clocuniq}(u\wedge1)R^{\nu}}{\log(R)^{2\cdot 1\{\alpha=2\nu\}}}\Big)^{\frac{1}{2\nu+1}}\bigg\}, \quad \text{ if $\alpha\geq2\nu$,}
  \end{equation}
  where $\textnormal{LocUniq}_{u, R}(z)$ denotes the event in \eqref{eq:lb1} with the choice $\lambda=\Cr{c:locuniq1}$.
 \end{The}

We refer to Remark~\ref{R:locuniqdiff} below regarding an upper bound for $\overline{\P}_{\tilde{\G}_K}(\textnormal{LocUniq}_{u, R}(z)^{\mathsf{c}})$ in the regime $\alpha<2\nu$ and the discrepancy between the two cases. In particular, Theorem~\ref{Thm:locuniq} yields the following instructive estimate with regards to the definition of $\xi(\cdot)$ in \eqref{eq:defxi}, which follows immediately from \eqref{eq:lb2alphalarger2nu} with $K=\emptyset.$ 
\begin{Cor}[under \eqref{eq:intro_Green}, \eqref{eq:intro_sizeball}, \eqref{eq:ellipticity}]  \label{C:locuniq}
If $\alpha > 2\nu$, then for all $u>0,$ $s > 1$, $z \in G,$ one has
\begin{equation}
\label{eq:lb2alphalarger2nubis}
\overline{\P}_{\tilde{\G}}\big(\textnormal{LocUniq}_{u, s u^{-1/\nu}}(z)^{\mathsf{c}} \big) \leq c \exp(-c'{s}^{\frac{\nu}{2\nu+1}}). \end{equation}
\end{Cor}
 
We now prepare the ground for the proof of Theorem~\ref{Thm:locuniq}. Throughout the remainder of this section, we tacitly assume \eqref{eq:intro_Green}, \eqref{eq:intro_sizeball} and \eqref{eq:ellipticity} to hold. We write $P^{\tilde{\G}_K}_\cdot$ for the canonical law of the Brownian motion $X$ on $\tilde{\G}$ 
killed when exiting $\tilde{\G}_K$, see \eqref{eq:G_K}, i.e.~of the process $X_{\cdot \wedge H_K}$ under $P_{\cdot}$ (for convenience, entering a cemetery state $\Delta \notin \tilde{\G}$ upon being killed). In particular, $P_{\cdot}=P_{\cdot}^{\tilde\G}$. Associated to this process is the capacity functional $\mathrm{cap}_{\tilde{\G}_K}(\cdot)$, defined similarly as $\mathrm{cap}(\cdot)=\mathrm{cap}_{\tilde{\G}}(\cdot)$ in \eqref{eq:cap}. Indeed, $\mathrm{cap}_{\tilde{\G}_K}(\cdot)$ is given by (2.19) in \cite{DrePreRod3} if one regards $\tilde{\G}_K$ as the cable system associated to the graph $\G_K$ with vertex set $G_K = (\tilde{\G}_K \cap G) \cup \partial K$, killing measure $\kappa_x = \infty$ if $x \in (G_K \cap \partial K)$ and $0$ otherwise, and weights $\lambda^K_{x,y}$, $x,y \in G_K$, given by $\lambda^K_{x,y}= \lambda_{x,y}$ whenever $x,y \in G$ and $\lambda^K_{x,y} = \frac{1}{2\rho(x,y)}$ if $x \in G$, $y \in \partial K$, where $\rho(\cdot,\cdot)$ denotes the Euclidean distance on the cable of $\tilde{\G}$ containing $x$ and $y$ (viewed as a line segment of length $1/2\lambda_{x,z}$ with $z \in G$ the corresponding other endpoint).

For later reference, we record the following estimates on $\tilde{\G}_K$ which mirror \eqref{eq:intro_Green} away from the boundary. Using that $ g_{\tilde{\G}_{\hat{K}}}(x,y) \leq g_{\tilde{\G}_K}(x,y) \leq g(x,y)$ for all $x,y \in \tilde{\G}_K$, where $\widehat{K}$ denotes the union of $K$ and the closure of all cables intersected by $K$ (so in particular $ \partial \widehat{K} \subset G$) and applying Lemma~3.1 in~\cite{DrePreRod2}, it follows that there exists $\Cl[c]{c:boundonNk} \geq 1$ such that if $d(z,0)\geq \Cr{c:boundonNk}R$ and $K\subset \tilde{B}(0,R)$,
\begin{equation}
\label{eq:GreenK}
c\leq g_{\tilde{\G}_K}(x,x)\leq c' \text{ and }  c d(x,y)^{-\nu}\leq g_{\tilde{\G}_K}(x,y)\leq c'd(x,y)^{-\nu}, \quad \text{ for all }x \neq y\in{B(z,R)}, 
\end{equation}
where $g_{\tilde{\G}_K}$ denotes the Green function killed outside $\tilde{\G}_K$, cf.~\eqref{eq:killedGreen}. Proceeding similarly as in the argument leading to~(3.11) in \cite{DrePreRod2}, \eqref{eq:GreenK} (and $\eqref{eq:intro_Green}$ in case $K=\emptyset$) then yields
  \begin{equation} \label{eq:capBd}
c R^\nu \le  \mathrm{cap}_{\tilde{\G}_K}(B(z,R)) \le c'R^\nu, \quad  \text{for all }R\geq1, \, K\subset \tilde{B}(0,R), \, d(z,0)\geq \Cr{c:boundonNk}R \cdot 1\{ K \neq \emptyset\}.
 \end{equation}

For $x\in G$ and $R>0$, we define the set ${\mathcal{C}}(x,R)$ as consisting of the vertices $z\in{B(x,R)}$ visited by the diffusion $X$ before the first time it exits $B(x,R).$ We begin with a lower bound on  $\mathrm{cap}_{\tilde{\G}_K}({\mathcal{C}}(x,R))$, which for $K=\emptyset$ can be viewed as refining Proposition~4.7 in~\cite{DrePreRod2}. 

\begin{Lemme}[$K\subset \tilde \G$ compact]
\label{lem:4.7}
For $x\in G$, $R,t\geq1$ with $R^{\nu}\geq 2t$ and $B(x,R)\subset\tilde{\G}_K$, 
\begin{equation}
    \label{eq:capifalpha>2nu}
    P_x^{\tilde{\G}_K}\Big(\mathrm{cap}_{\tilde{\G}_K}({\mathcal{C}}(x,R))\leq \frac{cR^{\nu}}{t\log(R)^{1\{\alpha=2\nu\}}}\Big)\leq c \exp(-c't^{\frac1\nu}), \text{  if $\alpha\geq2\nu$.}
\end{equation}
\end{Lemme}
\begin{proof}
Let $Z=(Z_n)_{n\geq 0}$ denote the discrete-time skeleton of the trace process on $G$ of the diffusion $X$ under $P_x=P_x^{\tilde{\G}}$ (cf.~(2.4) and below in \cite{DrePreRod3} for the definition), which has the law of the discrete-time Markov chain with transition probabilities induced by \eqref{eq:generator}, and write $Z_{[0,t]}=\{ Z_n : 0\leq n \leq t \}$, $t \geq 0$. By Lemma~4.4  in \cite{DrePreRod2} applied in the case $N=1$ and since $\alpha\geq 2\nu$ is equivalent to $\alpha/\beta\leq2$, where $\beta=\alpha-\nu$, with equality if and only if $\alpha= 2\nu$, there exist positive constants $c$ and $\Cl[c]{c:2ndmom_RW}$ such that $P_x^{\tilde{\G}} \big (\text{cap}(Z_{[0,(R^{\nu}/t)^{\beta/\nu}]}) \geq c R^{\nu}/(t\log(R)^{1\{\alpha=2\nu\}})\big) \geq \Cr{c:2ndmom_RW}$, for all $1 \leq t\leq \frac12R^{\nu}$. Hence, by the Markov property, we get that for all such $t$ and all $M>0$,
 \begin{equation}
 \label{eq:lb6}
  P_x^{\tilde{\G}}\big(\text{cap}_{\tilde{\G}}\big(Z_{[0,M(R^{\nu}/t)^{\beta/\nu}]}\big) \leq cR^{\nu}/(t\log(R)^{1\{\alpha=2\nu\}}) \big) \leq (1-\Cr{c:2ndmom_RW})^{cM}.
  \end{equation}
 Moreover, by (3.17) in \cite{DrePreRod2},
   \begin{equation}
 \label{eq:lb7}
  P_x^{\tilde{\G}}\big(Z_{[0,M(R^{\nu}/t)^{\beta/\nu}]} \cap B(x,R)^{\mathsf{c}} \neq \emptyset \big) \leq  ce^{-c'(\frac{t^{\beta/\nu}}{M})^{1/(\beta-1)}}, \text{ for all $M\geq1$, $1 \leq  t\leq R^{\nu}/2$}.
  \end{equation}
Combining \eqref{eq:lb6} and \eqref{eq:lb7} with the choice $M=t^{\frac1\nu},$ and noticing that ${\mathcal{C}}(x,R) \supset Z_{[0,M(R^{\nu}/t)^{\beta/\nu}]}$ under the complement of the event appearing on the left-hand side of \eqref{eq:lb7}, we obtain \begin{equation}
\label{eq:capifalpha>2nuG}
    P_x^{\tilde{\G}}\Big(\mathrm{cap}_{\tilde{\G}}({\mathcal{C}}(x,R))\leq \frac{cR^{\nu}}{t\log(R)^{1\{\alpha=2\nu\}}}\Big)\leq c \exp(-c't^{\frac1\nu}), \text{  if $\alpha\geq2\nu$.}
\end{equation}
Since $B(x,R)\subset\tilde{\G}_K,$ the law of $Z$ until the first exit time of $B(x,R)$ is the same under $P_x^{\tilde{\G}}$ and $P_x^{\tilde{\G}_K},$ and so \eqref{eq:capifalpha>2nuG} still holds when replacing $P_x^{\tilde{\G}}$ by $P_x^{\tilde{\G}_K}.$ As $\mathrm{cap}_{\tilde{\G}_K}(A)\geq\mathrm{cap}_{\tilde{\G}}(A)$ for all $A\subset \tilde{\G}_K,$ \eqref{eq:capifalpha>2nu} follows. 
\end{proof}

We now define $\widehat{\mathcal{C}}^u(x,R)$ under $P_x^{\tilde{\G}_K}\otimes \overline{\P}_{\tilde{\G}_K}$ as the union of ${\mathcal{C}}(x,R)$ and the vertices $y\in{B(x,R)}$ connected to any vertex in ${\mathcal{C}}(x,R)$ by a path of edges in $\widehat{\I}^u\cap B_E(x,R)$ (see above \eqref{eq:lb1} for the definition of $\widehat{\I}^u$). The attentive reader will notice that the following proof of Theorem~\ref{Thm:locuniq} could avoid the use of $\widehat{\mathcal{C}}^u(x,R)$ and be reformulated in terms of ${\mathcal{C}}(x,R)$ only. The use of $\hat{\mathcal C}^u(x,R)$ is justified by a possible extension to the case $\alpha<2\nu,$ see Remark \ref{R:locuniqdiff} for details, and it creates little extra difficulty in the proof of Theorem \ref{Thm:locuniq}, which still applies when $\alpha<2\nu$.

\begin{proof}[Proof of Theorem~\ref{Thm:locuniq}]
Let $z \in G$ and abbreviate ${B}^z={B}^z (z,R)$ and for $\lambda > 1$, ${B}^z_{\lambda}= B(z,\lambda R)$, $B^z_{E,\lambda } = B_E(z,\lambda R)$, see above \eqref{eq:lb1} for notation. Throughout the proof, given $u > 0$ we tacitly assume that $R$ is large enough so that $(u\wedge 1)R^{\nu} \geq 1.$ For $u>0$, we decompose $\mathcal{I}^u=\mathcal{I}_1^{u/4}\cup \mathcal{I}_2^{u/4} \cup \mathcal{I}_3^{u/4}\cup \mathcal{I}_4^{u/4}$, where $\I_k^{u/4}$, $k\in{\{1,2,3,4\}}$, are independent interlacement sets at level $u/4$ each. Similarly, let $\widehat{\mathcal{I}}_k^{u/4}$ be obtained from $\mathcal{I}_k^{u/4}$ in the same manner as $\widehat{ \mathcal{I}}^u$ from $\mathcal{I}^u$, whence $\widehat{\mathcal{I}}^u=\widehat{\mathcal{I}}_1^{u/4}\cup \widehat{\mathcal{I}}_2^{u/4} \cup \widehat{\mathcal{I}}_3^{u/4}\cup \widehat{\mathcal{I}}_4^{u/4}$. For $k\in{\{1,2,3,4\}}$, we denote by $Z_1^k,\dots,Z_{N_k}^k$ the (equivalence classes of) trajectories in the Poisson point process $\I_k^{u/4}$ which hit $B^z,$ and for each $i\in{\{1,\dots,N_k\}},$ we decompose $Z_i^k$ canonically into its ($M_i^k$ many) excursions $Z_{i,1}^k,\dots,Z_{i,M_i^k}^k,$ each started when hitting ${B}^z$ and ending when exiting $B^z_{\lambda}.$

Combining \eqref{eq:capBd} and the fact that $N_k$ is $\text{Poisson}(u\mathrm{cap}_{\tilde{\G}_K}(B^z)/4)$-distributed it follows by a standard large deviation estimate for Poisson random variables that
\begin{equation}
\label{eq:Nsmall}
\overline{\P}_{\tilde{\G}_K}(N_k\geq cuR^{\nu})\leq\exp(-c'uR^{\nu}), \text{ for all $u>0$, $R\geq 1$}\text{ with } d(z,0)\geq \Cr{c:boundonNk}R.
\end{equation} 
We now derive a suitable upper bound on the tails of $M_i^k,$ $k \in \{1, 2,3,4\},$ $i\in{\{1,\dots,N_k\}}$. Using \eqref{eq:entranceGreenequi} on $\tilde{\G}_K$, one finds that for all $\lambda \geq c(\geq1)$  and $R \ge 1$ with $d(z,0)\geq \Cr{c:boundonNk}\lambda R \cdot 1\{K \neq \emptyset\},$
\begin{equation}
\label{eq:HBRfinite}
\begin{split}
\sup_{x\in{\partial B^z_{\lambda}}}P_x^{\tilde{\G}_K}(H_{{B}^z}<\infty)
&\stackrel{\eqref{eq:GreenK}}{\leq} c((\lambda-1) R)^{-\nu}\mathrm{cap}(B^z)\stackrel{\eqref{eq:capBd}}{\leq} \frac{c'}{(\lambda-1)^{\nu}}\leq\frac12,
\end{split}
\end{equation}
where $e_{B^z,\tilde{\G}_K}$ denotes the equilibrium measure of the set $B^z$ in $\tilde{\G}_K.$
As a consequence of \eqref{eq:HBRfinite}, for $\lambda \geq c$, the random variables $M_i^k,$ $i\in{\{1,\dots,N_k\}},$  are stochastically dominated by independent  geometric random variables with parameter $1/2$ each. Therefore, using a union bound, a standard concentration inequality entails that for such $\lambda$ and all $k$, if $d(z,0)\geq \Cr{c:boundonNk}\lambda R  \cdot 1\{K \neq \emptyset\}$,
\begin{equation}
\label{eq:Miissmall}
\overline{\P}_{\tilde{\G}_K}\big(N_k\leq cuR^{\nu},\exists\,i\leq N_k,\ M_i^k\geq cuR^{\nu}\big)
\leq (cuR^{\nu})e^{-cuR^{\nu}}.
\end{equation}
Henceforth, we simply fix a value of $\lambda$ such that \eqref{eq:Miissmall} holds. For each $k\in{\{1,2,3,4\}},$  we then denote by $A_{R,k}^{u}$ the event that $N_k\leq cuR^{\nu}$ and $M_{i}^k\leq cuR^{\nu}$ for all $i\in{\{1,\dots,N_k\}},$ and take $A_{R}^u=A_{R,1}^{u}\cap A_{R,2}^{u}\cap A_{R,3}^{u}\cap A_{R,4}^{u}.$ It follows from \eqref{eq:Nsmall} and \eqref{eq:Miissmall} that for all $R\geq1$ 
\begin{equation}
\label{eq:Aruislarge}
\overline{\P}_{\tilde{\G}_K}((A_R^u)^c)\leq cuR^{\nu}e^{-c'uR^{\nu}}\text{ if }d(z,0)\geq  
\Cr{c:boundonNk}\lambda R \cdot1\{ K \neq \emptyset\}
\end{equation} 
(for all $K\subset \tilde{B}(0,R)$). Let us define the sets $\widehat{\mathcal{C}}_{i,j}^{m,n}$ as consisting of the vertices $z$ visited by $Z_{i,j}^m,$ as well as the vertices $y\in{B^z_{\lambda}}$ connected to such $z$ by a path of edges in $B^z_{E, \lambda }\cap\widehat{\I}_n^{u/4}.$ In particular, if  $x\in \mathcal{I}^u\cap B^z$, then $x$ belongs to $\widehat{\mathcal{C}}_{i,j}^{m,n}$ for some $m\in \{1,2,3,4\}$, $i \in \{ 1,\dots ,N_m\}$ and $j \in \{1,\dots M_i^m\},$ and any $n\in{\{1,2,3,4\}}.$ For $v\geq1$ we then infer that
\begin{equation*}
\begin{split}
\bigcup_{x,y \in \mathcal{I}^u \cap B^z}\{ x \nleftrightarrow y &\text{ in }  \mathcal{I}^u \cap B^z_{E,v\lambda} \}
\\&\subset\bigcup_{p,m \in \{1,2,3,4\}}\bigcup_{i_1=1}^{N_p}\bigcup_{i_2=1}^{N_m}\bigcup_{j_1=1}^{M_{i_1}^p}\bigcup_{j_2=1}^{M_{i_2}^m} \big\{\widehat{\mathcal{C}}_{i_1,j_1}^{p,n}\nleftrightarrow \widehat{\mathcal{C}}_{i_2,j_2}^{m,n} \text{ in } \widehat{\I}_k^{u/4} \cap B_{E, v\lambda }^z\big\},
\end{split}
\end{equation*}
where to each $m,p\in{\{1,2,3,4\}}$ we associate $n$ the smallest element of $\{1,2,3,4\}\setminus\{m,p\}$ and $k$ the smallest element of $\{1,2,3,4\}\setminus\{m,p,n\}.$  Next, denote by $\mathcal{A}$ the $\sigma$-algebra generated by the point processes underlying  $\mathcal{I}_1^{u/4},$ $\mathcal{I}_2^{u/4}$ and $\mathcal{I}_4^{u/4}$. In view of \eqref{eq:lb1}, returning to the previous display, applying first a union bound over $p,m$, then conditioning suitably and applying a second union bound, one finds that for all $v\geq1$, 
\begin{equation}
\label{eq:localuniqsmaller}
\begin{split}
&\overline{\P}_{\tilde{\G}_K}\big( \text{LocUniq}_{u,R,\lambda v}(z)^{\mathsf{c}}, A_R^u\big) \\
& \  \leq 16 (cuR^{\nu})^4 \sup_{m=1,4} \overline{\E}^{\tilde{\G}_K}\Big[ \sup_{i_1,i_2,j_1,j_2}\overline{\P}_{\tilde{\G}_K}\big(\widehat{\mathcal{C}}_{i_1,j_1}^{1,2}\nleftrightarrow \widehat{\mathcal{C}}_{i_2,j_2}^{m,2} \text{ in } \widehat{\mathcal{I}}_3^{u/4} \cap B^z_{E,v\lambda} \, \big| \, \mathcal{A} \, \big)1_{A_{R,1}^u\cap A_{R,m}^u} \Big],
\end{split}
\end{equation}
where the indices $i_1,i_2,j_1,j_2$ are all $\mathcal{A}$-measurable, $i_1$ ranges over $\{1,\dots N_1\}$, $i_2$ over $\{1,\dots N_m\}$, $j_1 \in \{ 1,\dots, M_1^{i_1}\}$ and $j_2 \in \{ 1,\dots, M_m^{i_2}\}$.
We then choose $v\in{(1,\infty)}$ large enough such that, by an adaptation of Lemma 4.3 in~\cite{DrePreRod2} to the current setup (using \eqref{eq:GreenK}, \eqref{eq:capBd} and \eqref{eq:Miissmall}), the following holds: on the event 
\begin{equation}
\label{eq:lb2}
\Big\{\text{cap}_{\tilde{\G}_K}(\widehat{\mathcal{C}}_{i_1,j_1}^{1,2})\geq \frac{cR^{\nu}}{t\log(R)^{1\{\alpha=2\nu\}}}\Big\}
\cap\Big\{\text{cap}_{\tilde{\G}_K}(\widehat{\mathcal{C}}_{i_2,j_2}^{m,2})\geq \frac{cR^{\nu}}{t\log(R)^{1\{\alpha=2\nu\}}}\Big\},
\end{equation}
if $d(z,0)\geq \Cr{c:boundonNk}v\lambda R \cdot 1\{ K \neq \emptyset\}$, we have
   \begin{equation}
 \label{eq:lb3}
 \begin{split}
\overline{\P}_{\tilde{\G}_K}\big( \widehat{\mathcal{C}}_{i_1,j_1}^{1,2}\nleftrightarrow \widehat{\mathcal{C}}_{i_2,j_2}^{m,2} \text{ in } \widehat{\mathcal{I}}_3^{u/4} \cap B^z_{E,v\lambda} \, \big| \, \mathcal{A} \, \big) &\leq \exp\big(-cuR^{-\nu}\mathrm{cap}_{\tilde{\G}_K}(\widehat{\mathcal{C}}_{i_1,j_1}^{1,2})\mathrm{cap}_{\tilde{\G}_K}(\widehat{\mathcal{C}}_{i_2,j_2}^{m,2})\big)
 \\&\leq
  \exp\Big(-\frac{cuR^{\nu}}{t^2\log(R)^{2\cdot 1\{\alpha=2\nu\}}}\Big),
 \end{split}
  \end{equation}
  using the bounds from \eqref{eq:lb2} in the second line. Moreover for all $k\in{\{1,4\}},$
\begin{equation}
\label{eq:lb22}
\begin{split}
&\overline{\P}_{\tilde{\G}_K}\Big(A_{R,k}^u,\exists\,i\leq N_k, j\leq M_k^{i},\ \mathrm{cap}_{\tilde{\G}_K}(\hat{\mathcal{C}}_{i,j}^{k,2})\leq \frac{cR^{\nu}}{t\log(R)^{1\{\alpha=2\nu\}}}\Big)
\\&\quad\leq\overline{\P}_{\tilde{\G}_K}\Big(\exists i\leq (cuR^{\nu}),\ j\leq (cuR^{\nu}),\ \mathrm{cap}_{\tilde{\G}_K}(\hat{\mathcal{C}}_{i,j}^{k,2})\leq \frac{cR^{\nu}}{t\log(R)^{1\{\alpha=2\nu\}}}\Big),
\end{split}
\end{equation}
where for every $i\leq N_k$ and $j>M_i,$ and every $i> N_k$ and $j\geq1,$ we define $\hat{\mathcal{C}}_{i,j}^{k,2}=G.$ Since, conditionally on the starting point of the respective excursion, the random set $\widehat{\mathcal{C}}_{i,j}^{k,2}$ stochastically dominates $\widehat{\mathcal{C}}^{u/4}(x, (\lambda-1) R)$ (under $P_x \otimes \overline{\P}$) for a certain vertex $x \in G$ (cf.~below Lemma \ref{lem:4.7} for notation), and $\widehat{\mathcal{C}}^{u/4}(x, (\lambda-1) R)\supset \mathcal{C}(x,(\lambda-1)R)$, the probability of the event in \eqref{eq:lb22} can be estimated using Lemma \ref{lem:4.7} and a union bound. Due to \eqref{eq:Aruislarge}, \eqref{eq:localuniqsmaller} and \eqref{eq:lb3}, the desired bound \eqref{eq:lb2alphalarger2nu} thus follows from \eqref{eq:capifalpha>2nu} by taking $\Cr{c:locuniq1}=v\lambda,$ $\Cr{cvalphagnuK}=\Cr{c:boundonNk}v\lambda$ and $t=c( (u\wedge1) R^{\nu}/\log(R)^{2 \cdot 1\{\alpha=2\nu\}})^{\frac{\nu}{2\nu+1}},$ with $c$ such that $R^\nu \ge 2t$ is satisfied for all $R \geq 1$, as required for Lemma~\ref{lem:4.7} to apply.
 \end{proof}

\begin{Rk}[The regime $\alpha<2\nu$]\label{R:locuniqdiff}
Proceeding similarly as above, one deduces under the assumptions of Theorem~\ref{Thm:locuniq} that, for all $u\in{(0,1)},$ $R\geq1$ and $z \in G$ with $d(z,0)\geq\Cr{cvalphagnuK}R \cdot 1\{ K \neq \emptyset\},$   
  \begin{equation}
 \label{eq:lb2alphasmaller2nu}
\overline{\P}_{\tilde{\G}_K}\big(\textnormal{LocUniq}_{u, R,\Cr{c:locuniq1}}(z)^{\mathsf{c}}\big) \leq c\exp\bigg\{ -\Big(\frac{\Cr{c:locuniq1}uR^{\beta}}{\log(R)^{c'}}\Big)^{c}\bigg\}, \quad \text{  if $\alpha<2\nu$,}
  \end{equation}
where $\beta=\alpha-\nu.$
To obtain \eqref{eq:lb2alphasmaller2nu}, one replaces Lemma~\ref{lem:4.7} by the following estimate, valid for all $x\in G,$ $u>0$ and $R,t\geq1$ such that $R^{\nu}\geq 2t$ and $B(x,R)\subset\tilde{\G}_K$, with $M=R^{\beta}t^{-\frac\beta\nu},$
\begin{equation}
    \label{eq:capifalpha<2nu} \tag{5.8'}
    P_x^{\tilde{\G}_K}\otimes\overline{\P}_{\tilde{\G}_K}\Big(\mathrm{cap}_{\tilde{\G}_K}(\widehat{\mathcal{C}}^u(x,R))\leq \frac{cR^{\nu}}{t}\Big)\leq cuR^{\nu}\exp\Big\{-\frac{c'(t^{\frac1\nu}\wedge uM)}{\log(M)}\Big\}, \text{ if $\alpha<2\nu$}.
\end{equation}
Once \eqref{eq:capifalpha<2nu} is shown, the proof of~\eqref{eq:lb2alphasmaller2nu} proceeds exactly as that of Theorem~\ref{Thm:locuniq} above, but using \eqref{eq:capifalpha<2nu} instead of \eqref{eq:capifalpha>2nu}, cf.~below \eqref{eq:lb3}, which yields the choice $t=(uR^{\beta}\log(R)^{\tilde{c}})^{\tilde{c}'},$
 whereupon \eqref{eq:lb2alphasmaller2nu} follows using the inequality $\nu>\beta$ and noticing that $\log(M)\leq C\log(R)$ when $uR^{\beta}\geq1.$ The proof of \eqref{eq:capifalpha<2nu} is somewhat more technical and we omit it here. Thus, for $\alpha<2\nu$, our current methods do not provide us with a bound similar to \eqref{eq:lb2alphalarger2nu} (even for $K=\emptyset$), that is a decay of $\overline{\P}_{\tilde{\G}_K}(\text{LocUniq}_{u,R,\lambda}(z)^{\mathsf{c}})$ to $0$ as $uR^{\nu}$ increases to $\infty$; such a decay would in turn improve the bound \eqref{eq:corrlength2ld} from Remark \ref{R:LB},\ref{R:LB.3} below to a bound similar to \eqref{eq:corrlength2lb}. 

 In fact, we do not expect that $uR^{\nu} \to \infty$ is sufficient for $\overline{\P}_{\tilde{\G}_K}(\text{LocUniq}_{u,R,\lambda}(z)^{\mathsf{c}})$ as defined in \eqref{eq:lb1} to decay to $0$ in the regime $\alpha<2\nu$, essentially because the capacity of the range of one random walk in a box of linear size $R$ grows as $R^{\beta}$, cf.~for instance Lemma 4.6 in \cite{DrePreRod2}, and $\beta< \nu$. As a consequence, one may therefore seek alternative approaches in intermediate ``dimensions'' by modifying the event $\text{LocUniq}_{u,R,\lambda}(z)$ in order to produce local connections through $\mathcal{I}^u.$ Finding such a connection strategy remains an interesting question in this regime.
\end{Rk}

\section{Connectivity lower bounds}
\label{sec:LB}

With Theorem~\ref{Thm:locuniq} at our disposal, we now proceed to supply the proofs of the lower bounds in \eqref{eq:corrlength1} and \eqref{eq:corrlength2lbpsi}, see Proposition~\ref{prop:lowerbounds} below. The arguments presented here are quite flexible, and we will explain in Section~\ref{sec:denouement} how to adapt them to i) obtain the outstanding lower bounds for $\tilde{\psi}(a,r)$ in \eqref{eq:corrlength2lb} and ii) deduce all corresponding lower bounds for $\tau_a^{\textnormal{tr}}(0,x)$ in Theorem ~\ref{T2}. The assumption \eqref{eq:defb} below, which allows for a (possible) logarithmic correction to the radius function at criticality when $\nu=1$, is known to hold with $q(r)=c(\log r)^{1/2}$ for $\nu=1$ by  \eqref{eq:psi_0nu=1} and with $q(r)=c$ when $\nu< 1$ by \eqref{eq:psi_0nu<1}. Recall $\xi=\xi (a)$ from \eqref{eq:defxi}.

\begin{Prop}[under \eqref{eq:intro_Green}, \eqref{eq:intro_sizeball}, \eqref{eq:ellipticity} and \eqref{eq:intro_shortgeodesic}]
\label{prop:lowerbounds} If
\begin{equation}
\label{eq:defb}
  \text{there exists $q: [1,\infty)\to  [1,\infty)$ s.t. }      \psi(0,r)\leq q(r)r^{-\nu/2}\text{ for all }r\geq2,
\end{equation}
for some $\nu \in (0, \frac\alpha{2})$, then there exist positive constants $\tilde{c},\tilde{c'}$ and $\tilde{c''}$ such that with $\xi=\xi(a)$,
\begin{equation}
\label{eq:boundonpsiwithb}
    \psi\big(a,r)\geq \psi(0,r)\exp\left\{- \tilde{c}q(\xi)- \frac{\tilde{c}(r/\xi)^{\nu \wedge 1}}{\log(r/\xi)^{b}}\right\}\text{ for all }a\in{[-\tilde{c}',\tilde{c}']}\text{ and }r\geq \tilde{c}''\xi,
\end{equation}
where $b=0$ if $\nu< 1$, $b=1$ if $\nu=1$ and $b=- \Cr{c:lbdefect2}$ if $\nu>1$.
\end{Prop}

Throughout the remainder of this section, we suppose that the assumptions \eqref{eq:intro_Green}, \eqref{eq:intro_sizeball}, \eqref{eq:ellipticity} and \eqref{eq:intro_shortgeodesic} of Proposition~\ref{prop:lowerbounds} are in force. Figure~\ref{F:LB} gives an idea of how the relevant connection event for $\psi\big(a,r)$ will be implemented. The underlying construction will be gradually unveiled over the course of this section and the next.

\begin{figure}[h!]
  \centering 
  \includegraphics[scale=0.7]{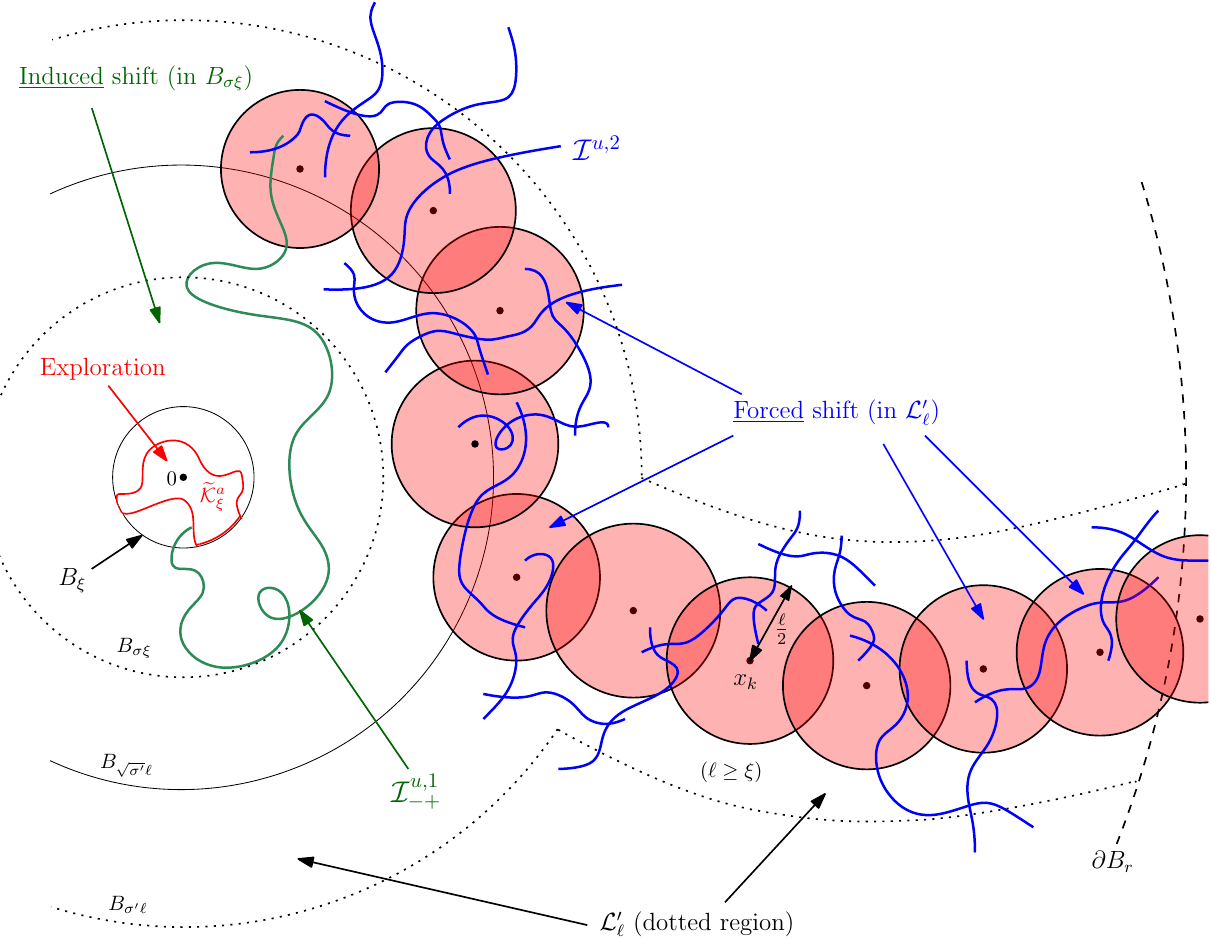}
  \caption{Connecting $0$ to $\partial B_r$ in three steps: 1) exploring the near-critical cluster $\tilde{\mathcal{K}}_{\xi}^a$ at scale $\xi$ (red), with associated cost controlled by~Lemma~\ref{lem:capacityislarge}; 2) connecting $\tilde{\mathcal{K}}_{\xi}^a$ when $\text{cap}(\tilde{\mathcal{K}}_{\xi}^a)$ is large enough to a multiple of that scale via a killed-surviving interlacement trajectory (green), cf.~Lemma~\ref{L:lb13} and its proof in Section~\ref{sec:lemmaProofs}; 3) bridging the remaining distance using optimal local uniqueness by means of Theorem~\ref{Thm:locuniq} (blue). \label{F:LB}}
\end{figure}

Our starting point is an estimate on the capacity of a piece of the cluster $\tilde{\mathcal{K}}^a$ truncated at the critical scale $\xi=\xi(a)$.
For $L\geq1,$ we abbreviate $\tilde{\K}_{L}^a = \tilde{\K}_{\tilde{B}(0,L)}^a $, see \eqref{eq:Ka} for notation. 
\begin{Lemme}[(Near-)critical estimate, $\nu \in (0, \frac\alpha{2})$] 
\label{lem:capacityislarge} If \eqref{eq:defb} holds, one has for all $0< a< c$,
\begin{align}
& \label{eq:capacityislargeneg}
    \P\big(\mathrm{cap}(\tilde{\K}_{\xi}^{-a})\geq \Cl[c]{ccaplarge}\xi^{\nu}\big)\geq \xi^{-\frac\nu 2} q(\xi)^{-1}, \\
&\label{eq:capacityislarge}
   \P\big(\mathrm{cap}(\tilde{\K}_{\xi}^{a})\geq \Cr{ccaplarge}\xi^{\nu}\big) \geq  \P\big(\mathrm{cap}(\tilde{\K}_{\xi}^{-a})\geq \Cr{ccaplarge}\xi^{\nu}\big) \exp\{-c q(\xi) \}.
\end{align}
\end{Lemme}

The proof of Lemma~\ref{lem:capacityislarge} is presented in Section \ref{sec:lemmaProofs}. In case $\nu=1$ this lemma applies with $q(\xi)= (\log \xi)^{\frac1{2}}$, thus yielding an effective regime in which the second term in the exponential of \eqref{eq:boundonpsiwithb} dominates when, say, $\frac{r}{\xi} \geq (\log \xi)^{\frac1{2}+\varepsilon}$ for some $\varepsilon > 0$. A weaker version of Lemma~\ref{lem:capacityislarge} can be obtained by simpler means, and yields a bound
with effective regime $\frac{r}{\xi} \geq (\log \xi)^{1+\eps},$ see Remark~\ref{R:nonoptimalL6.2} below. This is insufficient for later purposes, notably that of deducing Corollary~\ref{C:scalingrelation}.

\begin{proof}[Proof of Proposition~\ref{prop:lowerbounds}] We first observe that under \eqref{eq:intro_Green}, \eqref{eq:intro_sizeball}, \eqref{eq:ellipticity}, there exists $\Cl[c]{c:oneend}  \in (1,\infty)$ such that for all $R\geq c$,
\begin{equation}
\label{choicec18}
    \text{there exists a connected component of $B_{\Cr{c:oneend}^2 R}\setminus B_{R}$ which contains $\partial_{\text{in}}B_{\Cr{c:oneend} R}$}
\end{equation}
(indeed this follows from the first conclusion in the proof of Lemma 6.5 in \cite{DrePreRod2}). Henceforth, we tacitly assume that $0< a \leq c$ with $c$ chosen small enough so that \eqref{choicec18} holds whenever $R \geq \xi=\xi(a)$. 

Let $K \subset \tilde{B}_{\xi} (\subset \tilde{\G})$ be a compact set, soon to be chosen as $K= \tilde{\K}_{\xi}^{a}$, cf.~also the red region in Fig.~\ref{F:LB}. For such $K$, by \eqref{choicec18} and in view of \eqref{eq:G_K}, with $\tilde{B}_{R}$ as defined at the beginning of Section \ref{sec:radius}, one has
\begin{equation}
\label{eq:boundariesinG_K}
\tilde{\G} \setminus \tilde{B}_{\Cr{c:oneend} \xi} \subset \tilde{\G}_K. 
\end{equation}
Applying an argument akin to \eqref{eq:HBRfinite}, involving \eqref{eq:entranceGreenequi}, \eqref{eq:intro_Green} and \eqref{eq:capBd} (with $K=\emptyset$), one finds $\sigma \geq \Cr{c:oneend} $ suitably large such that for all compacts $K \subset \tilde{B}_{\xi}$,
\begin{equation}
\label{eq:applicationharnack}
    c\xi^{-\nu}\mathrm{cap}(K)\leq h_K(x) \leq \frac12\text{ for all }x\in{\partial_{\text{in}}{B}_{\sigma \xi}} \quad \text{ (see \eqref{eq:hdef} for notation).}
\end{equation}
Now, recalling that $\tilde{\K}_{\xi}^{a}= \tilde{\K}_{\tilde{B}_\xi}^{a} $, see \eqref{eq:Ka} for notation, 
writing $\P^{\tilde{\G}_K}$ for the canonical law of the Gaussian free field $\phi$ on the cable system $\tilde{\G}_K$ associated to the graph $\G_K,$ (cf.~below \eqref{eq:lb2alphalarger2nubis} for its definition), whence $\P^{\tilde{\G}} = \P$ as given by \eqref{eq:introGFF}, it follows from the Markov property \eqref{eq:Markov2} that 
\begin{equation}
\label{eq:markovKaphi}
    \big(\phi_x-h^{\phi}_{\tilde{\K}_\xi^a}(x)\big)_{x\in{\tilde{\G}_{\tilde{\K}_\xi^a}}}\text{ has the same law under }\P^{\tilde{\G}}(\cdot\,|\, \mathcal{A}^+_{\tilde{\K}_\xi^a})\text{ as }\phi\text{ under } 
    \P^{\tilde{\G}_{\tilde{\K}_\xi^a}}.
\end{equation}
 On the event $\{\mathrm{cap}(\tilde{\K}_\xi^a)\geq \Cr{ccaplarge}\xi^{\nu}\}$, the shift $h^{\phi}_{\tilde{\K}_\xi^a}(x)$, see \eqref{eq:hdef} for notation, can be bounded from below uniformly in $\tilde{B}_{\sigma \xi}$ using \eqref{eq:applicationharnack} as 
\begin{equation}
\label{eq:hkaphilarge}
    \inf_{x\in{\tilde{B}_{\sigma \xi }}}h^{\phi}_{\tilde{\K}_\xi^a}(x)\geq a\inf_{x\in{\tilde{B}_{\sigma \xi }}} P_x(H_{\tilde{\K}_\xi^a}<\infty)\geq a \inf_{x\in{\partial_{\text{in}}{B}_{\sigma \xi }}}P_x(H_{\tilde{\K}_\xi^a}<\infty)\geq ca \stackrel{\text{def.}}{=}2 \rho a.
    \end{equation}
In words, \eqref{eq:hkaphilarge} means that when its capacity is suitably large, the conditioning on $\tilde{\K}_\xi^a$ induces a shift $h^{\phi}_{\tilde{\K}_\xi^a}$ which is ``felt'' everywhere in $\tilde{B}_{\sigma \xi }$, see Fig.~\ref{F:LB}. In view of \eqref{eq:markovKaphi} and \eqref{eq:hkaphilarge}, we thus~obtain
\begin{equation}
\label{eq:psialphas}
    \psi(\rho  a ,r)\geq \E\big[1\big\{\mathrm{cap}(\tilde{\K}_\xi^a)\geq\Cr{ccaplarge}\xi^{\nu}\big\} \P^{\tilde{\G}_{\tilde{\K}_\xi^a}} \big(A(\tilde{\K}_\xi^a,a,r)\big)\big],
\end{equation}
where for $K\subset \tilde{B}_{\xi}$ we set
\begin{equation}
\label{eq:defALKar}
    A(K,a,r)=\left\{\begin{array}{c}\text{$\exists$ a continuous path }\pi\text{ in }\tilde{B}_r\text{ from }K\text{ to }\partial_{\text{in}}B_r
    \\\text{ with }\phi_x\geq -\rho  a \text{ for all }x\in{\pi\cap (\tilde{B}_{\sigma \xi }}\setminus K)
    \\\text{ and }\phi_x\geq \rho  a \text{ for all }x\in{\pi\cap (\tilde{B}_r\setminus\tilde{B}_{\sigma \xi }})\end{array}\right\}.
\end{equation}
In view of \eqref{eq:psialphas}, we aim at finding a suitable lower bound on the probability of $A(K,a,r)$ under $\P^{\tilde{\G}_K}$ for all admissible choices of ${K} \subset  \tilde{B}_{\xi}$ with large enough capacity. The desired result \eqref{eq:boundonpsiwithb} will then quickly follow from this bound and Lemma \ref{lem:capacityislarge}.

Consider a geodesic path $\gamma=(0=y_0,y_1,\dots)$ from \eqref{eq:intro_shortgeodesic} with $d_{\text{gr}}(y_k,y_{p})\leq \Cr{Cgeo}d(y_k,y_p)$ for all $k,p$ and recall from the beginning of Section~\ref{sec:radius} that  $d(x,y)\leq \Cr{c:ball1} d_{\text{gr}}(x,y)$ for all $x,y\in{G}.$ We now introduce a parameter $\sigma'$ used in the definition of $\mathcal{L}_{\ell}'$  in \eqref{eq:sausage} below and a length scale $\ell \geq 1$ which will play an important role in the construction that follows. From now on we assume that (see Theorem~\ref{Thm:locuniq} regarding $\Cr{c:locuniq1}$)
\begin{equation}
\label{eq:condLB}
\sigma' > 1, \ r \geq 100 \Cr{c:ball1} \text{ and }  (4 \Cr{c:ball1})\vee(\sigma \xi ) \leq  \ell \leq r/(10\Cr{c:locuniq1}\vee \sigma').
\end{equation}
In view of \eqref{eq:boundonpsiwithb}, \eqref{eq:defxi} and since $0< a < c \leq 1$, the condition on  $r$ is no loss of generality. For $k$ such that $1\leq k \leq 1+\lceil (2+4r/\ell)\Cr{c:ball1}\Cr{Cgeo}\rceil=:N_{\ell,r}$, fix a point $x_k \in \gamma\cap \big(B_{\text{gr}}(0,(k+1)\ell/(4\Cr{c:ball1}))\setminus B_{\text{gr}}(0,k\ell/(4\Cr{c:ball1}))\big)$ (such $x_k$ necessarily exists by assumption on $\ell$ and since $\gamma$ is a graph distance geodesic) and define, with $x_0=0,$
\begin{equation}
\label{eq:sausage}
\mathcal{L}_{\ell}=\bigcup_{0\leq k\leq N_{\ell,r}} B(x_k,\Cr{c:locuniq1}\ell)\quad \text{ and } \quad\mathcal{L}_{\ell}'= \big(B(x_0,\sigma'\ell)\cup\mathcal{L}_{\ell}\big)\setminus B(x_0,{\sigma \xi }).
\end{equation}
We write $\tilde{\mathcal{L}}_{\ell}$, $\tilde{\mathcal{L}}_{\ell}'$ for the corresponding sets obtained by replacing $B(x_k,\cdot)$ with $\tilde{B}(x_k,\cdot)$ everywhere in \eqref{eq:sausage} (see the beginning of Section~\ref{sec:radius} regarding continuous balls $\tilde{B}(x,r)$). Note that since $\sigma \geq \Cr{c:oneend},$ see above \eqref{eq:applicationharnack}, one has $\tilde{\mathcal{L}}_{\ell}'\subset \tilde{\G}_K$ for any $K\subset\tilde{B}_{\xi}$ by \eqref{eq:boundariesinG_K}.  One easily checks, using that $d_{\text{gr}}(x_k,x_{k+1})\leq \frac{\ell}{2\Cr{c:ball1}}$ for the first inclusion below and \eqref{eq:intro_shortgeodesic} for the second one, that
\begin{equation}
\label{eq:propertyofLl}
B(x_{k+1},\ell/2)\subset B(x_k,\ell)\text{ for all }k\leq N_{\ell,r}-1\text{ and } B(x_{N_{\ell,r}},\ell/2)\subset (G\setminus B(0,r)).
\end{equation}
The length scale $\ell$ in \eqref{eq:sausage} will be carefully chosen below (see \eqref{eq:lb8}). In the sequel, we always  assume that $K \subset \tilde{B}_{\xi}$ is compact but otherwise arbitrary. We now introduce the measure $\P_{a,\ell}^{\tilde{\G}_K}$, defined similarly as $\P_{a'}$ in \eqref{eq:df4}, but when considering $\P^{\tilde{\G}_K}$ instead of $\P=\P^{\tilde{\G}}$ and with the choices $a'=-2\rho a$ and $K= \tilde{\mathcal{L}}_{\ell}'$ in \eqref{eq:df4}. 
Thus, cf.~below \eqref{eq:df1},  
\begin{equation}
\label{eq:LBpsishift}
\text{$(\varphi_x)_{x \in \tilde{\mathcal{L}}'_{\ell}}$ has the same law under ${\P}_{  a,\ell }^{\tilde{\G}_K}$ as $(\varphi_x +2\rho  a )_{x \in \tilde{\mathcal{L}}'_{\ell}}$ under $\P^{\tilde{\G}_K}$.}
\end{equation}
Recall that $\text{cap}_{\tilde{\G}_K}(\cdot)$ denotes the capacity on $\tilde{\G}_K$, see the paragraph following \eqref{eq:lb2alphalarger2nubis} regarding its definition. Since $M_{\tilde{\mathcal{L}}_{\ell}'}$ (see \eqref{eq:df2} for notation) is centered under $\P^{\tilde{\G}_K},$ we have ${\E}_{  a,\ell }^{\tilde{\G}_K}[M_{\tilde{\mathcal{L}}_{\ell}'}] =2\rho  a  \text{cap}_{\tilde{\G}_K}(\tilde{\mathcal{L}}_{\ell}').$ As a consequence, due to \eqref{eq:df4}, we get for all $a>0$ that
\begin{equation}
\label{eq:lb11}
\begin{split}
{\E}_{  a,\ell }^{\tilde{\G}_K}\Big[\log\frac{{\rm d}{\P}_{  a,\ell }^{\tilde{\G}_K}}{{\rm d}\P^{\tilde{\G}_K}}\Big] = 2\rho  a  {\E}_{  a,\ell }^{\tilde{\G}_K}[M_{\tilde{\mathcal{L}}_{\ell}'}]   - 2(\rho  a) ^2 \text{cap}_{\tilde{\G}_K}(\tilde{\mathcal{L}}_{\ell}') =2(\rho  a) ^2\text{cap}_{\tilde{\G}_K}(\tilde{\mathcal{L}}_{\ell}').
\end{split}
\end{equation}
Using \eqref{eq:lb11}, a classical change-of-measure argument -- see for instance below (2.7) in \cite{BDZ95} --  yields that for all $K \subset \tilde{B}_{\xi}$ compact, all $0< a< c$ and $r$, $\ell$, $\sigma'$ such that \eqref{eq:condLB} holds,
\begin{equation}
\label{eq:lb12}
\P^{\tilde{\G}_{K}}\big(A(K,a,r)\big)\geq{\P}_{  a,\ell }^{\tilde{\G}_K}\big(A(K,a,r)\big)\exp\left\{-\frac{2(\rho  a)^2\text{cap}_{\tilde{\G}_K}(\tilde{\mathcal{L}}_{\ell}')+1/e}{{\P}_{  a,\ell }^{\tilde{\G}_K}\big(A(K,a,r)\big)}\right\}.
\end{equation}
It thus remains to find suitable bounds on the various quantities appearing on the right-hand side of \eqref{eq:lb12}. We collect these separately in two lemmas, the proofs of which will be supplied in Section \ref{sec:lemmaProofs}. The first lemma gives an upper bound on the capacity of $\mathcal{L}_{\ell}'$. Care is needed due to the presence of the ``boundary condition'' arising from the removal of $K$ in $\tilde{\G}_K$, see \eqref{eq:G_K}. Let $f(t)=t^{1-\nu}$ if $\nu<1,$ $f(t)=\log(t)$ if $\nu=1$ and $f(t)=1$ if $\nu>1$. 

\begin{Lemme}
\label{L:lb9}
For all $\nu>0$ and $r,\ell,\sigma'$ satisfying \eqref{eq:condLB},
\begin{equation}
\label{eq:lb9}
\sup_K \, \textnormal{cap}_{\tilde{\G}_K}( \tilde{\mathcal{L}}_{\ell}') \leq  \frac{c(\sigma')^{\nu}r\ell^{\nu-1}}{f(r/\ell)},
\end{equation}
where the supremum ranges over all compact sets $K \subset \tilde{B}_{\xi}$.
\end{Lemme}

We now bound ${\P}_{  a,\ell }^{\tilde{\G}_K}\big({A}_{L}(K,a,r)\big)$ suitably from below, which involves choosing the scale~$\ell.$ Recall that $\tilde{\mathcal{L}}_{\ell}'$ implicitly depends on the parameter $\sigma' > 1$, see \eqref{eq:sausage}. We refer to Remark~\ref{R:LB},\ref{R:LB.00} with regards to extending the following result to the case $\alpha=2\nu$.

\begin{Lemme}[$\nu \in (0, \frac{\alpha}{2})$] There exist $ \sigma' > 1$ and $M \geq \sigma (>1)$ such that, with $\xi=\xi(a)$ and 
\label{L:lb13}
\begin{equation}
\label{eq:lb8}
\ell= M \xi \Big( \log \frac{r}{\xi} \Big)^{\frac{2\nu +1}{\nu}}, 
\end{equation}
for all compacts $K\subset \tilde{B}_{\xi},$ all $a\in (0,c)$ and $r\geq \Cl[c]{C:finalLB2} \xi$, 
\begin{equation}
\label{eq:lb13}
{\P}_{  a,\ell }^{\tilde{\G}_K}\big(A(K,a,r)\big) \geq \tilde{c}\big(1-\exp(-\tilde{c}'(\rho  a) ^2\mathrm{cap}(K))\big).
\end{equation}
\end{Lemme}

Let us now explain how to conclude assuming Lemmas \ref{L:lb9} and \ref{L:lb13} to hold. Let $\nu \in (0, \frac{\alpha}{2})$. For $0<a < c (\leq \frac12)$, returning to \eqref{eq:psialphas} and applying \eqref{eq:capacityislargeneg}, \eqref{eq:capacityislarge}, one obtains with $\xi= \xi(a)$ that
\begin{equation}
\label{eq:LBfinal}
\psi(\rho  a ,r)\geq \xi^{-\frac\nu 2} q(\xi)^{-1} \exp\{ -cq(\xi) \}  \, \inf_{K} \, \P^{\tilde{\G}_{K}}\big(A(K,a,r)\big),
\end{equation}
where the infimum ranges over all compact subsets $K$ of $\tilde{B}_{\xi}$ satisfying $\mathrm{cap}(K)\geq\Cr{ccaplarge}\xi^{\nu}$. In order to apply \eqref{eq:lb12} and get a lower bound on the quantity $\P^{\tilde{\G}_{K}}(A(K,a,r))$ appearing on the right-hand side of \eqref{eq:LBfinal}, the conditions \eqref{eq:condLB} must be met. Fix $\sigma' > 1$ (and $M$) such that the conclusions of Lemma~\ref{L:lb13} hold and let $\ell$ be given by \eqref{eq:lb8}. 
For $0<a < c (\leq \frac12)$, recalling $\xi= \xi(a)$ from \eqref{eq:defxi}, we may assume that 
$\ell \geq 4 \Cr{c:ball1}$, as required by \eqref{eq:condLB} (note that the condition $\ell \geq \sigma \xi$ is automatically satisfied as $M \geq \sigma$ in \eqref{eq:lb8}). Letting $u(x)= x/(\log x)^{\frac{(2\nu +1)}{\nu}}$ and $\Cl[c]{C:finalLB1}= \Cr{C:finalLB2} \vee \inf \{x > 0 : u(x) \geq M (10 \Cr{c:locuniq1} \vee \sigma') \}$, we then see that $r$ satisfies all conditions in \eqref{eq:condLB} whenever $r \geq \Cr{C:finalLB1} \xi$, and moreover \eqref{eq:lb13} holds. For $K$ with $\mathrm{cap}(K)\geq\Cr{ccaplarge}\xi^{\nu}$, the latter implies that ${\P}_{  a,\ell }^{\tilde{\G}_K}(A(K,a,r)) \geq c$ (recall that $\rho \in (0,1)$ is fixed, see \eqref{eq:hkaphilarge}). 

Thus, going back to \eqref{eq:LBfinal}, applying \eqref{eq:lb12}, which is in force, and substituting the uniform lower bound for ${\P}_{  a,\ell }^{\tilde{\G}_K}(A(K,a,r))$ yields that for all $0< a < c$ and $r \geq \Cr{C:finalLB1} \xi $,
\begin{equation}
\label{eq:LBfinal2}
\psi(\rho  a ,r)\geq \xi^{-\frac\nu 2} q(\xi)^{-1} \exp\Big\{ -cq(\xi) -  \frac{\tilde{c}(r/\xi)^{\nu \wedge 1}}{\log(r/\xi)^{b}}  \Big\};
\end{equation}
in obtaining \eqref{eq:LBfinal2}, we also used \eqref{eq:defxi} and applied the capacity bound \eqref{eq:lb9} with $\ell$ as in \eqref{eq:lb8} to deduce that $\text{cap}_{\tilde{\G}_K}(\tilde{\mathcal{L}}_{\ell}') \leq c r^{\nu}$ when $\nu < 1$, $\text{cap}_{\tilde{\G}_K}(\tilde{\mathcal{L}}_{\ell}') \leq \frac{c r}{\log(r/\xi)}$ when $\nu = 1$ and $\text{cap}_{\tilde{\G}_K}(\tilde{\mathcal{L}}_{\ell}') \leq \frac{c ra^{-2}}{\xi} (\log \frac{r}{\xi})^{\Cr{c:lbdefect2}}$ when $\nu > 1$. Finally, to get \eqref{eq:boundonpsiwithb} for $a>0$ from \eqref{eq:LBfinal2}, one bounds
$$
 \xi^{-\frac\nu 2} q(\xi)^{-1} \stackrel{ \eqref{eq:defb}}{\geq} \psi(0,\xi) q(\xi)^{-2} \stackrel{r \geq \xi}{\geq} \psi(0,r) q(\xi)^{-2},
$$
and notes that the factor $q(\xi)^{-2}$ can be absorbed into $\exp\{ -cq(\xi)\}$ in \eqref{eq:LBfinal2}. The corresponding estimate in \eqref{eq:boundonpsiwithb} for $-c<a<0$ follows by symmetry, using Lemma~4.3 in \cite{DrePreRod3}, which applies in the present setting (this again follows from Theorem 1.1,2 in \cite{DrePreRod3}) and Lemma 6.1 in \cite{DrePreRod3}, recalling that \eqref{eq:intro_Green} implies in particular that \eqref{T1_sign} holds). This completes the proof of Proposition~\ref{prop:lowerbounds}, subject to Lemmas~\ref{lem:capacityislarge},~\ref{L:lb9} and~\ref{L:lb13}, which are proved in the next section.
\end{proof}

\section{Proofs of the three intermediate lemmas}
\label{sec:lemmaProofs}
We now supply the proofs of Lemmas~\ref{lem:capacityislarge}, \ref{L:lb9} and \ref{L:lb13}, which were assumed to hold in the previous section, thereby completing the proof of Proposition~\ref{prop:lowerbounds}. For intuition, we refer the reader to Figure~\ref{F:LB}. We begin with the

\begin{proof}[Proof of Lemma \ref{lem:capacityislarge}]
Let $0< a\leq 1/10$ and $\xi=\xi(a)$. We first show \eqref{eq:capacityislargeneg}. Let $\tilde{\K}= \tilde{\K}_{\xi/2}^0$ and consider the event $A = \{   \mathrm{cap}(\tilde{\K}) \geq s  q(\xi)^{-2} \xi^{\nu} \}$ for $s>0$.
Combining \eqref{eq:defb} and the tail estimate \eqref{eq:captails} from Corollary~\ref{C:captails} in case $a_N=0$, one sees upon choosing $s \in (0,1]$ small enough that
\begin{equation}
\label{eq:lem6.2.1}
\P(A)\geq  \P\big( \mathrm{cap}(\tilde{\K}^0)\geq s  q(\xi)^{-2} \xi^{\nu}  \big)- \P\big(\tilde{\K}^0\not\subset \tilde{B}_{\xi/2}\big) \geq q(\xi) \xi^{-\frac{\nu}{2}}.
\end{equation}
Deducing \eqref{eq:capacityislargeneg} from \eqref{eq:lem6.2.1} involves strengthening the capacity lower bound (in a twice larger box and at level $-a$) to reach order $\xi^\nu$. This will be achieved by forcing an interlacement trajectory onto $\tilde{\K}$ and using a refinement of the isomorphism theorem \eqref{eq:weakisom}, as follows. Applying (Isom) on p.4 of \cite{DrePreRod3}, which is in force under the present assumptions, one infers that
\begin{equation}
\label{eq:lem6.2.2}
\P\big(\mathrm{cap}(\tilde{\K}_{\xi}^{-a})\geq s'\xi^{\nu}\big)\geq \E \big[\, \overline{\P} \big( \text{cap}\big( \mathcal{I}_{\tilde{\K}}^{a^2/2} \cap \tilde{B}_{\xi} \big) \geq s'\xi^{\nu}  \big) 1_A\big];
\end{equation}
here, $\mathcal{I}^{a^2/2}$ refers to the interlacement set at level $u=\frac{a^2}{2}$ (with law $\overline{\P}=\overline{\P}_{\tilde{\G}}$) and $\mathcal{I}_{\tilde{\K}}^{a^2/2}$ to the set obtained as the trace of all the trajectories hitting $\tilde{\K}$ (governed by the independent probability $\P$), run from the time they first visit $\tilde{\K}$ until first exiting $\tilde{B}_{\xi}$. In particular, the event on the right-hand side of \eqref{eq:lem6.2.2} implies that $\mathcal{I}^{a^2/2} \cap \tilde{\K} \neq \emptyset$. 

We now derive a suitable lower bound on $ \overline{\P} ( \text{cap}( \mathcal{I}_{\tilde{\K}}^{a^2/2} \cap \tilde{B}_{\xi} ) \geq s'\xi^{\nu} )$. Conditioning on the number of trajectories visiting $\tilde{\K}$ (with respect to which the event $\{\mathcal{I}^{a^2/2} \cap \tilde{\K} \neq \emptyset\}$ is measurable) as well as their entrance points in $ \tilde{\K} $ and denoting the corresponding $\sigma$-algebra by $\mathcal{F}$, the following holds. On the event $\{\mathcal{I}^{a^2/2} \cap \tilde{\K} \neq \emptyset\}$, writing $x_0 (\in \partial \tilde{\K})$ for the starting point of the trajectory with (say) smallest label visiting $\tilde{\K}$, which is $\mathcal{F}$-measurable, and applying Lemma~\ref{lem:4.7} with $K=\emptyset$, $t=1$ and $R=\xi$, one sees that for $s' =  \Cr{ccaplarge}$ small enough,
$$
 \overline{\P}\big( \text{cap}\big( \mathcal{I}_{\tilde{\K}}^{a^2/2} \cap \tilde{B}_{\xi} ) \geq   \Cr{ccaplarge} \xi^{\nu} \, \big| \, \mathcal{F} \big) \geq  P_{x_0}\big(\mathrm{cap}({\mathcal{C}}(x_0,\xi)) \geq  \Cr{ccaplarge} \xi^{\nu}  \big) \geq c
$$
(on the event $\{\mathcal{I}^{a^2/2} \cap \tilde{\K} \neq \emptyset\}$). Returning to \eqref{eq:lem6.2.2}, it thus follows that
\begin{equation}
\label{eq:lem6.2.3}
\P\big(\mathrm{cap}(\tilde{\K}_{\xi}^{-a})\geq  \Cr{ccaplarge}  \xi^{\nu}\big)\geq c \E \big[\, \overline{\P} \big( \mathcal{I}^{a^2/2} \cap \tilde{\K} \neq \emptyset  \big) 1_A\big]\stackrel{\eqref{eq:defRI}}{=} c \E \big[\, \big(1- e^{-\frac{a^2}{2} \mathrm{cap}(\tilde{\K}) } \big) 1_A\big].
\end{equation}
Now, inserting the lower bound $\mathrm{cap}(\tilde{\K}) \geq s  q(\xi)^{-2} \xi^{\nu}$ valid on the event $A$, recalling that $\xi^{\nu}=a^{-2}$, using that $1-e^{-x} \geq cx$ for $x\in [0,1]$ and \eqref{eq:lem6.2.1}, \eqref{eq:lem6.2.3} is readily seen to imply \eqref{eq:capacityislargeneg}. The bound \eqref{eq:capacityislarge} follows by combining \eqref{eq:capacityislargeneg}, \eqref{eq:defxi}, \eqref{eq:capBd} and Proposition \ref{Prop:appendix} for $K=B_{\xi}.$
\end{proof}

\begin{Rk}\label{R:nonoptimalL6.2} 
Proceeding similarly as in \eqref{eq:lem6.2.1} but at level $a>0$ directly while taking advantage of \eqref{eq:rad_5}, \eqref{eq:rad_6} and \eqref{eq:defb}, one obtains for small enough $s\in{(0,1]}$ that
\begin{equation*}
\begin{split}
     \P(\mathrm{cap}\big(\tilde{\K}_\xi^a)\geq s \xi^\nu q(\xi)^{-2}\big)&\geq \P\big(\mathrm{cap}\big(\tilde{\K}^a)\geq s \xi^\nu q(\xi)^{-2}\big)- \P\big(\tilde{\K}^a\not\subset \tilde{B}_{\xi}\big)
    \geq q(\xi)\xi^{-\nu/2}.
     \end{split}
\end{equation*}
 In comparison with \eqref{eq:capacityislarge} (combined with \eqref{eq:capacityislargeneg}), the present lower bound is easier to prove since it does not require the change-of-measure \eqref{newlb4}, and gives a better estimate but for a weaker event. Crucially, when $\nu=1,$ in which case one can choose $q(\xi) = \sqrt{\log(\xi)}$ due to \eqref{eq:psi_0nu=1}, the above argument only produces $\mathrm{cap}\big(\tilde{\K}_\xi^a)$ of order $\xi (\log \xi)^{-1}$ rather than $\xi.$ This has rather dramatic effects. Indeed, retracing the arguments in the proof of Proposition~\ref{prop:lowerbounds}, one arrives at the bound $\P_{a,\ell}^{\tilde{\G}_K}\big(A(K,a,r)\big) \geq c  (\log \xi)^{-1}$ obtained from \eqref{eq:lb13} (with $K=\tilde{\K}_\xi^a$), which manifests itself unfavorably in the exponential \eqref{eq:lb12}.
\end{Rk}

\medskip
We now turn to the proofs of Lemmas \ref{L:lb9} and \ref{L:lb13} which require some further preparation due mainly to the presence of the Dirichlet boundary condition on $K$ inherent to $\tilde{\G}_K$. In the sequel, let $K \subset \tilde{\G}$ be compact. Recall $\overline{\P}_{\tilde{\G}_K}$, the canonical law of the interlacement process on $\widetilde{\mathcal{G}}_K$ from \eqref{eq:RIdefG_K}. Its intensity measure $\nu^K$ is defined on a space of continuous doubly-infinite trajectories $w^*$ on $\tilde{\G}_K\cup \{ \Delta\},$ where $\Delta$ is a cemetery state. Denoting by $\pi^*$ the canonical projection identifying equivalent trajectories up to time-shift reparametrisations, one may assume that $w^*=\pi^*( w)$ for some doubly-infinite trajectory $w = (w(t))_{t \in{\R}}$ with $w(0)\neq\Delta,$ and both its forwards and backwards parts $ (w(\pm t))_{t \geq 0}$ can either be \textit{killed} -- that is reach $\Delta$ after a finite time, which corresponds to exiting $\tilde{\G}_K$ through $\partial K$ -- or \textit{survive}, i.e.~escape to infinity (in possibly finite time) without reaching $K$. Henceforth, we call a trajectory $w$ (and a fortiori $w^* = \pi^*(w)$) \textit{killed-surviving} if its backwards part is killed ($-$) and its forwards part surviving ($+$). We denote by $W^*_{-+}$ the set of these trajectories. Similarly, $W^*_{--}$ consists of all trajectories whose backwards and forwards parts are both killed. We then define two intensity measures 
\begin{equation}
\label{eq:killedintensities}
 \nu^K_{-+} = 1_{W^*_{-+}}\nu^K, \quad  \nu^K_{-} = (1_{W^*_{-+}} + 1_{W^*_{--}}) \nu^K,
\end{equation}
which induce two processes on $\tilde{\G}_K (\cup \Delta)$ (under the measure $\overline{\P}_{\tilde{\G}_K}$) with respective intensities $u \nu^K_{-+}$ and $u \nu^K_{-}$, for $u>0$. We refer to them as \textit{killed-surviving} and \textit{backwards-killed} interlacement processes, respectively, and write $\mathcal{I}^u_{-+} $ and $\mathcal{I}^u_{-} $ for the corresponding interlacement sets at level $u$ (cf.~also Section 2 in \cite{Pre1} and Corollary 3.4 therein for an alternative description of these processes). These processes will play a central role below.

We begin with the following useful lemma, which in particular determines the total intensity of the killed-surviving process; see the beginning of Section~\ref{S:diffform} for notation.

\begin{Lemme}[$K \subset \tilde{\G}$ compact, $\tilde{\G}_K$ as in~\eqref{eq:G_K}]
\label{le:numberofbackwardkilled}

For all compacts $\mathcal{L}\subset \tilde{\G}_K$ such that every unbounded continuous path on $\tilde{\G}$ starting in $K$ intersects $\mathcal{L}$ (when viewing $\tilde{\G}_K$ as a subset of~$\tilde{\G}$),
\begin{equation}
\label{eq:numberofbackwardkilled}
    \mathrm{cap}_{\tilde{\G}}(K)=\langle e_{\mathcal{L},\tilde{\G}_K}-e_{\mathcal{L},\tilde{\G}},1-h_{K}\rangle = \int \textnormal{d}  \nu^K_{-+} .
\end{equation}
\end{Lemme}
\begin{proof}
By Lemma 2.1 in \cite{DrePreRod3}, we can assume without loss of generality that $K,\mathcal{L}\subset G.$ As in the proof of Lemma~\ref{lem:4.7}, let $(Z_n)_{n \geq 0}$ denote the discrete-time skeleton of the trace of the diffusion $X$
on $G \cup \{\Delta \}$ under $P_x^{\tilde{\G}}$ or $P_x^{\tilde{\G}_K},$ for $x\in ( \tilde{\G}_K \cap G)$. We write $H_U(Z)=\inf\{n\geq0:\,Z_n\in{U}\}$ and $\tilde{H}_U(Z)=\inf\{n\geq1:\,\,Z_n\in{U}\}$ for the first hitting and return times of $Z$ in $U$ with the convention $\inf\emptyset=\infty.$ In view of (2.4) in \cite{DrePreRod3}, $Z$ is a Markov chain which jumps from $y$ to $z$ with probability $\lambda_{y,z}/\lambda_y$ under $P_x^{\tilde{\G}},$ and $Z$ under $P_x^{\tilde{\G}_K}$ has the same law as $Z$ killed at time $H_K(Z)$ under $P_x^{\tilde{\G}}.$  Due to (2.6) and (2.16) in \cite{DrePreRod3}, we thus have, for all $x\in{\mathcal{L}}$,
\begin{equation*}
\begin{split}
& \lambda_x^{-1}(e_{\mathcal{L},\tilde{\G}_K}- e_{\mathcal{L},\tilde{\G}})(x)  \\
&\qquad \quad = P_x^{\tilde{\G}_K}(\tilde{H}_{\mathcal{L}}(Z)=\infty)-P_x^{\tilde{\G}}(\tilde{H}_{\mathcal{L}}(Z)=\infty)=P_x^{\tilde{\G}}({H}_{K}(Z)\leq \tilde{H}_{\mathcal{L}}(Z))-P_x^{\tilde{\G}}(\tilde{H}_{\mathcal{L}}(Z)=\infty)
    \\&\qquad \quad =P_x^{\tilde{\G}}({H}_{K}(Z)<\infty,{H}_{K}(Z)\leq \tilde{H}_{\mathcal{L}}(Z)) =P_x^{\tilde{\G}_K}(\tilde{H}_{\mathcal{L}}(Z)=\infty,Z\text{ is killed})
\end{split}
\end{equation*}
where in the third equality we used the fact that $H_K(Z)=\tilde{H}_{\mathcal{L}}(Z)(=\infty)$ when $\tilde{H}_{\mathcal{L}}(Z)=\infty$ since every connected and unbounded paths starting in $K$ hits $\mathcal{L},$  and in the last equality that the event $H_K(Z)<\infty$ under $P_x^{\tilde{\G}}$ corresponds to the event that $Z$ is killed, i.e.~that $H_{\Delta}(Z)<\infty$, under $P_x^{\tilde{\G}_K}$. Hence,
\begin{equation}
\label{eq:numberoftraj}
\begin{split}
    \langle e_{\mathcal{L},\tilde{\G}_K}-e_{\mathcal{L},\tilde{\G}},1-h_{K}\rangle
        &=\sum_{x\in{\mathcal{L}}}\lambda_xP_x^{\tilde{\G}_K}(Z\text{ survives})P_x^{\tilde{\G}_K}(\tilde{H}_{\mathcal{L}}(Z)=\infty,Z\text{ is killed})
    \\&=\sum_{x\in{\mathcal{L}}}e_{\mathcal{L},\tilde{\G}_K}(x)P_x^{\tilde{\G}_K}(Z\text{ survives})P_x^{\tilde{\G}_K}(Z\text{ is killed}\,|\,\tilde{H}_{\mathcal{L}}(Z)=\infty).
\end{split}
\end{equation}
But by definition of the intensity measure $\nu^K$, see for instance (2.9), (2.11) and (3.9) in \cite{Pre1}, the second line of \eqref{eq:numberoftraj} is precisely the measure of the set trajectories in $W_{-+}^*$ hitting $\mathcal{L},$ that is all of $W_{-+}^*$ by assumption on $\mathcal{L}.$ Thus, the second equality in \eqref{eq:numberofbackwardkilled} holds.

Since $\nu^K$ is invariant under time reversal, see for instance Remark 3.3,1) in \cite{Pre1}, \eqref{eq:numberoftraj} equals $\nu^K(W^*_{+-})$, the intensity of trajectories whose backwards parts survive and forwards parts are killed, which is equal to
\begin{equation*}
\begin{split}
    &\sum_{x\in{\mathcal{L}}}e_{\mathcal{L},\tilde{\G}_K}(x)P_x^{\tilde{\G}_K}(Z\text{ is killed})P_x^{\tilde{\G}_K}(Z\text{ survives}\,|\,\tilde{H}_{\mathcal{L}}(Z)=\infty)
    \\&\quad=\sum_{x\in{\mathcal{L}}}\lambda_xP_x^{\tilde{\G}}(H_K<\infty)P_x^{\tilde{\G}}(\tilde{H}_{\mathcal{L}}(Z)=\infty)
    =P_{e_{\mathcal{L},\tilde{\G}}}(H_K<\infty)=\mathrm{cap}_{\tilde{\G}}(K),
\end{split}
\end{equation*}
where we used \eqref{eq:balayage} in the last equality, and we conclude that \eqref{eq:numberofbackwardkilled} holds.
\end{proof}

We now proceed to the

\begin{proof}[Proof of Lemma \ref{L:lb9}]
By \eqref{eq:sausage} and definition of the continuous balls $\tilde{B}(x,r)$ at the beginning of Section~\ref{sec:radius}, one knows that $\text{cap}_{\tilde{\G}_K}(\tilde{\mathcal{L}}_{\ell}')=\text{cap}_{\tilde{\G}_K}({\mathcal{L}}_{\ell}')$.
Combining \eqref{eq:applicationharnack} and \eqref{eq:numberofbackwardkilled} for the choice $\mathcal{L}= \mathcal{L}_{\ell}'$, which satisfies the assumptions of Lemma~\ref{le:numberofbackwardkilled} since every unbounded path from $K \subset {\tilde{B}}_{\xi} $ intersects ${\partial_{\text{in}}{B}_{\sigma \xi}} \subset \mathcal{L}$ in view of \eqref{eq:sausage} and since $\ell\geq \sigma \xi,$ we obtain that
\begin{equation}
\label{eq:comparecap}
\begin{split}
   \textnormal{cap}_{\tilde{\G}_K}( \mathcal{L}_{\ell}')-\textnormal{cap}_{\tilde{\G}}( {\mathcal{L}}_{\ell}')&=\langle e_{\mathcal{L}_{\ell}',\tilde{\G}_K}-e_{\mathcal{L}_{\ell}',\tilde{\G}},1\rangle
    \\&\leq 2\langle e_{\mathcal{L}_{\ell}',\tilde{\G}_K}-e_{\mathcal{L}_{\ell}',\tilde{\G}},1-h_{K}\rangle=2\textnormal{cap}_{\tilde{\G}}(K)
    \leq c\xi^\nu,
\end{split}
\end{equation}
where we used \eqref{eq:capBd} in the last inequality. Moreover it follows from \eqref{eq:capBd} and \eqref{eq:sausage} that
\begin{equation}
\label{compareLlLandLl}
    \textnormal{cap}_{\tilde{\G}}( {\mathcal{L}}_{\ell}')\leq \textnormal{cap}_{\tilde{\G}}( {\mathcal{L}}_{\ell})+c(\sigma')^{\nu}\ell^{\nu}.
\end{equation}
Let us now bound $\textnormal{cap}_{\tilde{\G}}( \mathcal{L}_{\ell}).$ Using \eqref{eq:intro_Green}, it follows that for all $x \in \mathcal{L}_{\ell}$, assuming $x \in B(x_{k_0}, \Cr{c:locuniq1}\ell)$ and letting $I$ consist of all indices $k\geq0$ with $|k-k_0|\geq 2+8\Cr{c:locuniq1}\Cr{c:ball1}\Cr{Cgeo}$ divisible by $\lceil 2+8\Cr{c:locuniq1}\Cr{c:ball1}\Cr{Cgeo}\rceil$ (note that the corresponding balls $B(x_k, \Cr{c:locuniq1}\ell)$, $k \in I$, are disjoint, and that $|I| \geq c r/\ell$), 
$$
\sum_{y \in \mathcal{L}_{\ell}} g(x,y) \geq \sum_{k \in I} \sum_{y \in B(x_k, \Cr{c:locuniq1}\ell)} g(x, y) \geq  \sum_{k \in I}  c\ell^{\alpha}\inf_{y\in{B(x_k,\Cr{c:locuniq1}\ell)}}d(x,y)^{-\nu}.
$$ 
Since $d(x,y)\leq c\ell|k+1-k_0|$ for all $y\in{B(x_k,\Cr{c:locuniq1}\ell)}$ and 
\begin{equation}
\label{eq:defoff}
\sum_{k \in I}  |k+1-k_0|^{-\nu} \geq \sum_{\substack{1\leq k \leq \lceil 2r\Cr{c:ball1}\Cr{Cgeo}/\ell\rceil \\ k=0 \text{ mod }\lceil2+8\Cr{c:locuniq1}\Cr{c:ball1}\Cr{Cgeo}\rceil}} (k+1)^{-\nu} \geq c f(r/\ell),
\end{equation}
we obtain that
\begin{equation*}
\sum_{y \in \mathcal{L}_{\ell}} g(x,y)\geq c\ell^{\alpha-\nu}f(r/\ell), \quad \text{for all } x \in \mathcal{L}_{\ell}.
\end{equation*}
Clearly $|\mathcal{L}_{\ell}| \leq c \ell^{\alpha}\frac{r}{\ell}$ by \eqref{eq:intro_sizeball},  and so using the bound $\textnormal{cap}( \mathcal{L}_{\ell}) \leq |\mathcal{L}_{\ell}| / \inf_{x \in \mathcal{L}_{\ell}} \sum_{y \in \mathcal{L}_{\ell}} g(x,y)$, which follows by summing \eqref{eq:entranceGreenequi} for $K= \mathcal{L}_\ell$ over $\mathcal{L}_{\ell}$ (note to this effect that, $\mathcal{L}_{\ell}$ being a subset of $G$, $e_{\mathcal{L}_{\ell}}$ coincides with the equilibrium measure for the discrete chain generated by \eqref{eq:generator}, see (2.16) in \cite{DrePreRod3}), we obtain
\begin{equation}
\label{eq:capLell}
    \mathrm{cap}_{\tilde{\G}}(\mathcal{L}_{\ell})\leq \frac{cr\ell^{\nu-1}}{f(r/\ell)}\text{ for all }  4 \Cr{c:ball1} \leq  \ell \leq r/10\Cr{c:locuniq1}.
\end{equation}
Noting that $\ell^{\nu}\leq r\ell^{\nu-1}/f(r/\ell)$ since $ \ell\leq \frac{r}{10}$, \eqref{eq:lb9} follows from  \eqref{eq:comparecap}, \eqref{compareLlLandLl} and \eqref{eq:capLell}.
\end{proof}

It remains to give the

\begin{proof}[Proof of Lemma \ref{L:lb13}]
The following considerations hold for any $r, \ell$ (and $0< a< c$) satisfying the conditions appearing in \eqref{eq:condLB}, which we assume to hold in the sequel. The specific choice of $\ell$ in \eqref{L:lb13} will only be made at the very end (see below \eqref{eq:finalboundonG}). Recall that $K \subset \tilde{B}_{\xi}$. Define $\tilde{\mathcal{L}}_{\ell}^K$ to be the union of $\tilde{B}(0,\sigma'\ell)\cap\tilde{\G}_K$ and $\tilde{B}(x_k,\Cr{c:locuniq1}\ell)\cap\tilde{\G}_K,$ $0\leq k\leq N_{\ell,r}$, cf.~\eqref{eq:sausage}. Under $\P_{a,\ell}^{\tilde{\G}_K}$ as defined below \eqref{eq:propertyofLl} and due to \eqref{eq:LBpsishift}, $(\varphi_x)_{x \in \tilde{\mathcal{L}}_\ell^K}$ has the same law as $(\varphi_x + \chi )_{x \in \tilde{\mathcal{L}}_\ell^K}$ under $\P^{\tilde{\G}_K}$, where $ \chi \geq 0$ and $\chi = 2\rho  a$ on $\tilde{\mathcal{L}}_{\ell}^K \setminus \tilde{B}_{\sigma \xi}$, and thus by \eqref{eq:defALKar},
\begin{equation}
\label{eq:lb14}
\P_{a,\ell}^{\tilde{\G}_K}(A(K,a,r)) \geq {\P}^{\tilde{\G}_K}\big(K \leftrightarrow \partial_{\text{in}} {B}_r \text{ in }\{ \varphi \geq -\rho  a \} \cap  \tilde{\mathcal{L}}_{\ell}^K\big)
\end{equation}
where $K \leftrightarrow \partial_{\text{in}} {B}_r$ in $A\subset\tilde{\G}$ means that there exists a continuous path $\pi$ in $A$ from $K$ to $\partial_{\text{in}} {B}_r.$ 

We now further delimit the region in which we will construct the path achieving the connection in \eqref{eq:lb14}.
To this end we first choose $\sigma'  \geq (1+ \Cr{c:ball1}\Cr{Cgeo}\Cr{cvalphagnuK}')^2 $ large enough so that 
\begin{equation}
\label{choicec20}
    \text{there exists a connected component of $B_{(\sigma'-\Cr{c:locuniq1}-1)\ell}\setminus B_{(\Cr{cvalphagnuK}'+1)\ell}$  containing $\partial_{\text{in}}B_{\sqrt{\sigma'}\ell}$},
\end{equation}
where $\Cr{cvalphagnuK}' = \Cr{cvalphagnuK} \vee \Cr{c:boundonNk} \vee (\Cr{c:locuniq1} +\Cr{c:oneend} + 1)$
(see Theorem~\ref{Thm:locuniq} and \eqref{eq:capBd} regarding $\Cr{c:locuniq1}$, $\Cr{cvalphagnuK}$ and $\Cr{c:boundonNk}$), which exists using \eqref{choicec18} with $R=\ell\sqrt{\sigma'}/\Cr{c:oneend}$ and taking $\sigma' > 1$ sufficiently large. The specific choice of $\Cr{cvalphagnuK}' $ and the explicit lower bound on $\sigma'$ will ensure that various sets, e.g.~all the vertices pertaining to the set $S$ which we introduce next, are sufficiently distant from $0$, as required in \eqref{eq:defS} and \eqref{eq:distancerequirementLB} below. With $\sigma'$ fixed, by Lemma 6.1 in \cite{DrePreRod2}, there exists a set $S$ with $|S|\leq c$ for some constant $c$ (independent of $\ell$) such that 
\begin{equation}
\label{eq:defS}
    B_{(\sigma'-\Cr{c:locuniq1}-1)\ell}\setminus B_{(\Cr{cvalphagnuK}'+1)\ell}\subset\bigcup_{z\in{S}}B(z,\ell/4),\ \tilde{B}(z,\Cr{c:locuniq1}\ell)\subset \tilde{\mathcal{L}}_{\ell}^K\text{ and }d(z,0)\geq \Cr{cvalphagnuK}'\ell \text{ for all }z\in{S}
\end{equation}
(this follows by considering $\{z\in{\Lambda(\ell/4)}:\,z\in{B_{(\sigma'-\Cr{c:locuniq1})\ell}\setminus B_{\Cr{cvalphagnuK}'\ell}}\},$ where $\Lambda(\ell/4)$ is the set defined in Lemma 6.1 of \cite{DrePreRod2}). Let $S'=\{z\in{S}:\,B(z,\ell/4)\cap\partial_{\text{in}}B_{\sqrt{\sigma'}\ell}\neq\emptyset\}.$ There exists $z_0\in{S'}$ such that one can find some vertex $y\in{B(z_0,\ell/4)\cap \gamma\cap\partial_{\text{in}}B_{\sqrt{\sigma'}\ell}},$ where $\gamma$ is the geodesic from \eqref{eq:intro_shortgeodesic}. By definition of the vertices $x_k,$ $k\leq N_{\ell,r}$ above \eqref{eq:sausage} in terms of $\gamma$, and since $r\geq \sigma'\ell,$ there exists also some $N_{\ell,r}\leq N_{\ell,r}$ such that $d_{\text{gr}}(y,x_{N_{\ell,r}})\leq \ell/(4\Cr{c:ball1}),$ and therefore $d(z_0,x_{N_{\ell,r}})\leq \ell/2.$ Consequently, by \eqref{choicec20} and \eqref{eq:defS}, for all $z\in{S'}$ there exists a nearest-neighbor path $\pi=(\pi_0,\dots,\pi_p)\subset \bigcup_{w\in{S}}B(w,\ell/4)$ of vertices joining $\pi_0\in{B(z_0,\ell/4)}$ and $\pi_p\in{B(z,\ell/4)},$ and if we fix $z_i\in{S}$ so that $\pi_i\in{B(z_i,\ell/4)}$ for all $i\in{\{1,\dots,p\}}$ (with $z_p=z$) we have
\begin{equation}
\label{eq:connectionviazi}
    B(z_0,\ell/2)\subset B(x_{N_{\ell,r}},\ell)\text{ and }B(z_{i},\ell/2)\subset B(z_{i-1},\ell)\text{ for all }1\leq i\leq p.
\end{equation}

We proceed to define a suitable event implementing the desired connection in \eqref{eq:lb14} and refer to Fig.~\ref{F:LB} for visualization. Let $u=(\rho  a) ^2/4$ and $\mathcal{P}=S\cup \{x_k,k\in{\{N_{\ell,r},\dots,N_{\ell,r}\}}\}.$ By \eqref{eq:weakisom} applied to $\tilde{\G}_K$, there exists a coupling $\mathbb{Q}$ of $\varphi$ under ${\P}^{\tilde{\G}_K}$ with $\mathcal{I}^{2u}$ under $\overline{\P}_{\tilde{\G}_K}$ such that $\{ \varphi \geq -\rho  a \} \supset \mathcal{I}^{2u},$ and $\I^{2u}$ splits into two independent interlacements $\I^{u,1}$ and $\I^{u,2}$ at level $u$ such that $\I^{2u}=\I^{u,1}\cup\I^{u,2}.$ We denote by $\text{LocUniq}_{u,\ell}^{(2)}(x)$ the same event as in \eqref{eq:lb1} with the choice $\lambda = \Cr{c:locuniq1}$, but for the interlacements $\I^{u,2},$ and note that for every $x\in{\mathcal{P}},$ all the edges in $\hat{\I}^{u,2}\cap B_{E_K}(x,\Cr{c:locuniq1}\ell)$ (where $E_K$ is the set of edges associated with $\G_K$, cf.~the paragraph following \eqref{eq:lb2alphalarger2nubis}) have their respective cables included in  $\I^{2u}\cap \tilde{\mathcal{L}}_{\ell}^K$ by definition of $\hat{\I}^{2u},$ $B_{E_K}(x,\Cr{c:locuniq1}\ell),$ see above \eqref{eq:lb1}, $\tilde{\mathcal{L}}_{\ell}^K$ and $S,$ see \eqref{eq:defS}. We then define $\I^{u,1}_{-}$ as the set of vertices hit by any trajectories in the interlacements process associated to $\I^{u,1}$ whose backwards parts are killed on $K$, which has intensity $u\nu^K_-$, see \eqref{eq:killedintensities}, and $\I^{u,1,\ell}_{-}$ as the set of vertices visited by any trajectories in $\I^{u,1}_-$ before their first exit time of $B(0,\sigma'\ell)$. Let us consider the (good) event
\begin{equation}
\label{eq:lb15}
G= \Big\{\I^{u,1,\ell}_{-}\cap\I^{u,2}\cap\bigcup_{z\in{S'}}B(z,\ell/2)\neq\emptyset\Big\}  \cap \bigcap_{x\in{\mathcal{P}}} \big( \big\{ \text{LocUniq}^{(2)}_{u,\ell}(x)\big\} \cap \{\I^{u,2} \cap B(x,\ell/2) \neq \emptyset\} \big).
\end{equation}
By \eqref{eq:lb15}, the definition of the local uniqueness event in \eqref{eq:lb1} and by the construction of $\mathcal{L}_{\ell}^K,$ $S$ as well as $S',$ see in particular \eqref{eq:sausage}, \eqref{eq:propertyofLl}, \eqref{eq:defS} and \eqref{eq:connectionviazi}, the occurrence of $G$ entails that $K$ is connected to $\partial_{\text{in}} {B}_r$ by a continuous path in $\mathcal{I}^{2u}\cap\tilde{\mathcal{L}}_{\ell}^K,$ and hence the event on the right-hand side of \eqref{eq:lb14} occurs under~$\mathbb{Q}$. Defining for all $s\in{[0,1]}$ the event $G_s'=\{\exists\,z\in{S'}:\,\mathrm{cap}_{\tilde{\G}_K}(\I^{u,1,\ell}_{-}\cap B(z,\ell/2))\geq s\ell^{\nu}
\},$ we therefore have
\begin{equation}
\label{eq:lb16}
\P_{a,\ell}^{\tilde{\G}_K}(A(K,a,r)) \stackrel{\eqref{eq:weakisom}}{\geq}\Q(G)\geq \E^{\mathbb{Q}}\big[\mathbb{Q}(G\,|\,\I^{u,1})1_{G_s'}\big].
\end{equation}
We will bound $\Q(G_s')$ for suitable $s$ and $\mathbb{Q}(G\,|\,\I^{u,1})$ on the event $G_s'$ separately from below.

Let us first derive a bound on $\Q(G_s')$. To this end, fixing an arbitrary ordering of $S'$ and whenever $\bigcup_{z\in{S'}}B(z,\ell/4)\cap\I^{u,1,\ell}_{-}$ is not empty, we denote by $Z^{u,1}\in{S'}$ the smallest vertex $z\in{S'}$ such that $B(z,\ell/4)$ is hit by the trajectory in $\I^{u,1,\ell}_{-}$ with smallest label, and by $X^{u,1}$ the first entrance point in $B(Z^{u,1},\ell/4)$ of this trajectory; otherwise, i.e.~if $\bigcup_{z\in{S'}}B(z,\ell/4)\cap\I^{u,1,\ell}_{-}=\emptyset$ we set $X^{u,1}=Z^{u,1}=0.$ Recalling the definition of $\mathcal{C}(x,\ell/4)$ from above Lemma \ref{lem:4.7}, we have that, conditionally on $Z^{u,1}$ and $X^{u,1}$ and on the event that $(Z^{u,1},X^{u,1}) \neq (0,0)$, the set $\I^{u,1,\ell}_{-}\cap B(Z^{u,1},\ell/2)$ stochastically dominates $\mathcal{C}(X^{u,1},\ell/4)$ under $P_{X^{u,1}}^{\tilde{\G}_K}$. Hence, by Lemma \ref{lem:4.7}, which applies due to \eqref{eq:defS}, \eqref{eq:boundariesinG_K} and by choice of $\Cr{cvalphagnuK}'$ below \eqref{choicec20}, together implying that $B(X^{u,1},\ell/4)\subset\tilde{\G}_K$, we obtain for all $s$ small enough that
\begin{equation}
\label{eq:1stboundonG's}
\begin{split}
    \Q(G_s')&\geq \E^{\Q}\big[P_{X^{u,1}}^{\tilde{\G}_K}\big(\mathrm{cap}_{\tilde{\G}_K}({\mathcal{C}}(X^{u,1},\ell/4))\geq s\ell^{\nu}
    \big)
    1\{\exists\,z\in{S'}:\,B(z,\ell/4)\cap\I^{u,1,\ell}_{-}\neq\emptyset\}\big]
    \\&\geq \big(1-c\exp(-c's^{-1/\nu})\big)\Q\big(\exists\,z\in{S'}:\,B(z,\ell/4)\cap\I^{u,1,\ell}_{-}\neq\emptyset\big).
\end{split}
\end{equation}
Let $\I^{u,1}_{-+}\subset \I^{u,1}_{-} (\subset \I^{u,1})$ refer to the killed-surviving interlacement set corresponding to $ \I^{u,1}$, see below \eqref{eq:killedintensities}. By definition, $\I^{u,1}_{-+}$ comprises the range of all trajectories in a Poisson process of intensity $u \nu^K_{-+}$, and if $\I^{u,1}_{-+}\neq \emptyset$, we have that $\I^{u,1,\ell}_{-}\cap \partial_{\text{in}}B_{\sqrt{\sigma'}\ell}\neq\emptyset.$ By definition of $S',$ see below \eqref{eq:defS}, this means in turn that there exists $z\in{S'}$ such that $B(z,\ell/4)\cap \I^{u,1,\ell}_{-}\neq\emptyset.$ Now, Lemma \ref{le:numberofbackwardkilled} implies that the number of trajectories in the process underlying $\I^{u,1}_{-+}$ is a Poisson variable with parameter $u \mathrm{cap}_{\tilde{\G}}(K)$. Thus, returning to \eqref{eq:1stboundonG's}, we infer that for sufficiently small $s_0\in (0,1)$ (henceforth fixed),
\begin{equation}
    \label{eq:G'shaslargeproba}
    \Q(G_{s_0}')\geq\frac12\Q\big(\exists\,z\in{S'}:\,B(z,\ell/4)\cap\I^{u,1,\ell}_{-}\neq\emptyset \big)\geq\frac12\big(1-\exp\big(-((\rho  a)^2/4)\mathrm{cap}_{\tilde{\G}}(K)\big)\big),
\end{equation}
where we used $u=(\rho  a)^2/4$ in the last equality. 

We now bound $\mathbb{Q}(G\,|\,\I^{u,1})$ on the event $G_s',$  cf.~\eqref{eq:lb16}. By \eqref{eq:intro_shortgeodesic}, our choice of $N_{\ell,r},$ see above \eqref{eq:connectionviazi}, and of $\sigma',$ see above \eqref{choicec20}, we have, for all $k\geq N_{\ell,r}$,
\begin{equation}
\label{eq:distancerequirementLB}
    d(x_k,0)\geq\frac{1}{\Cr{Cgeo}}d_{\text{gr}}(x_k,0)\geq\frac{1}{\Cr{Cgeo}}d_{\text{gr}}(x_{N_{\ell,r}},0)\geq \frac{1}{\Cr{Cgeo}\Cr{c:ball1}}d(x_{N_{\ell,r}},0)\geq\frac{\sqrt{\sigma'}-1}{\Cr{Cgeo}\Cr{c:ball1}}\ell\geq\Cr{cvalphagnuK}'\ell \geq\Cr{cvalphagnuK}\ell.
\end{equation}
In view of \eqref{eq:distancerequirementLB}, \eqref{eq:defS} and since $\alpha> 2\nu$ by assumption, Theorem~\ref{Thm:locuniq} (with, say, $u_0=1$) applies with $R=\ell$ and any $z\in{\mathcal{P}}$ (note also that $K \subset \tilde{B}_{\xi} \subset \tilde{B}_{\ell}$ for any $\ell$ satisfying \eqref{eq:condLB}). Thus, using \eqref{eq:RIdefG_K} and the bound $|\mathcal{P}|\leq cr/\ell,$ for any $r$, $\ell$ satisfying \eqref{eq:condLB} we obtain that 
\begin{equation}
\label{eq:finalboundonG}
\mathbb{Q}(G^c\,|\,\I^{u,1})1_{G'_{s_0}}\stackrel{\eqref{eq:lb2alphalarger2nu}, \eqref{eq:capBd}}{\leq} e^{-us_0\ell^{\nu}
}+ \frac{cr}{\ell} \Big(e^{-(\Cr{clocuniq}
u\ell^{\nu})^{\frac1{2\nu+1}} } + e^{-cu\ell^{\nu}}\Big).
\end{equation}
Finally, choosing $\ell$ as in \eqref{eq:lb8}, since $u=ca^2$ and by \eqref{eq:defxi}, we find that for all $M \geq 1$,
\begin{equation*}
    u\ell^{\nu}\geq  c M^{\nu} \Big(\log \frac{r}{\xi}\Big)^{2\nu+1}, \quad \frac{r}{\ell} \leq  \frac{1}{M}\frac{r}{\xi},   \end{equation*}
whence both terms on the right-hand side of \eqref{eq:finalboundonG} tend to $0$ as $M \to \infty$. Hence choosing $M \geq \sigma'$ large enough, we can arrange for $\mathbb{Q}(G^c\,|\,\I^{u,1})1_{G'_{s_0}} \leq \frac12$. Combining this with \eqref{eq:lb16}, \eqref{eq:G'shaslargeproba} and noting that the present choice of $\ell$ implies $r \geq (10\Cr{c:locuniq1}\vee \sigma')\ell $, as required by \eqref{eq:condLB}, under the condition $r \geq \Cr{C:finalLB2} \xi$ and $a<c,$ we obtain \eqref{eq:lb13}.
\end{proof}

\section{Denouement}
\label{sec:denouement}

Combining the upper and lower bounds derived in Sections \ref{sec:radius} and \ref{sec:LB}, respectively, we now complete the proof of Theorem \ref{T2}, and explain in particular how to adapt the arguments from Section \ref{sec:LB} which yield a lower bound for $\psi$, to deduce similar bounds for $\tilde{\psi}$ and, importantly,~$\tau_a^{\textnormal{tr}},$ cf.~\eqref{eq:intro_psitilde} and \eqref{eq:eq:deftauhf} for their respective definitions. The proofs of Theorem~\ref{T:psitilde} as well as those of Corollaries \ref{C:scalingrelation} and \ref{Cor:upperboundvolume} are presented at the end of this section.

\begin{proof}[Proof of Theorem \ref{T2}]
The proofs of \eqref{eq:corrlength1} and \eqref{eq:corrlength2} appear in Section \ref{sec:radius} (following Remark~\ref{R:rhobounds}). As we now explain, the lower bound \eqref{eq:corrlength2lbpsi} follows from Proposition \ref{prop:lowerbounds}. First observe that the condition $\alpha> 2\nu$ (with $\alpha$ from \eqref{eq:intro_sizeball}) appearing in Proposition \ref{prop:lowerbounds} always holds when $\nu=1$ due to \eqref{eq:condonalphanu}. Now, in view of \eqref{eq:psi_0nu=1}, the proof of which is given in Remark \ref{R:rad}.\ref{R:twopointUB}, the condition \eqref{eq:defb} holds with $q(r)= c(\log r)^{1/2}$ when $\nu=1$ and thus the asserted lower bound \eqref{eq:corrlength2lbpsi} follows from \eqref{eq:boundonpsiwithb} when $r\geq \xi(a)(\log\xi(a))^{\Cr{c:defect3}}$ and $|a|< c$, for any choice of $\Cr{c:defect3}\in{(\frac12,1)}$. In the near-critical regime $r \leq \xi$, the lower bound \eqref{eq:corrlength2lbpsi} follows from \eqref{eq:rad_113} with the choice $t=1$ (see also Remark \ref{rk:appendix},\ref{rk:appendix1} for an alternative approach). Finally, in case $ \psi(0,r) \asymp r^{-1/2}$, Proposition \ref{prop:lowerbounds} applies with $q(r)=c$ (cf.~\eqref{eq:defb}) and gives \eqref{eq:corrlength2lbpsi} for all $r\geq c\xi,$ which is complemented in the near-critical regime by means of \eqref{eq:rad_113} with $t=c$ sufficiently large, thus yielding overall that \eqref{eq:corrlength2lbpsi} holds for all $r \geq 1$ and $|a|< c$. 

Note that Proposition \ref{prop:lowerbounds} also provides an alternative proof of the lower bound in \eqref{eq:corrlength1} in the regime $r/\xi \geq c$. This is relevant for pending adaptations of this proof to deduce the corresponding lower bounds for $\tau_a^{\textnormal{tr}},$ for which the easier arguments of Section~\ref{sec:radius} are not available, cf.~Remark~\ref{R:rad}.\ref{R:twopointUB}. We return to duly discuss matters around $\tau_a^{\textnormal{tr}}$ further below.

We now turn to the bounds on the truncated two-point function $\tau_a^{\textnormal{tr}}$ asserted as part of Theorem~\ref{T2}. The (analogues for $\tau_a^{\textnormal{tr}}$ of the) upper bounds in \eqref{eq:corrlength1} and \eqref{eq:corrlength2} are detailed in Remark~\ref{R:rad},\ref{R:twopointUB}. It remains to explain how to adapt the arguments of Sections~\ref{sec:LB} and~\ref{sec:lemmaProofs} to obtain the desired lower bounds on $\tau_a^{\textnormal{tr}}$. We highlight the significant changes.

Assuming \eqref{eq:intro_Green}, \eqref{eq:intro_sizeball}, \eqref{eq:ellipticity} and $d=d_{\text{gr}}$ -- the latter renders \eqref{eq:intro_shortgeodesic} superfluous -- to hold, and for $\nu \in (0, \frac\alpha{2})$, we will argue how to deduce lower bounds similar to the ones in \eqref{eq:boundonpsiwithb} for $\tau_a^{\textnormal{tr}}$ under the assumption~\eqref{eq:defb}, from which the analogues
 of \eqref{eq:corrlength2lbpsi} and of the lower bound in \eqref{eq:corrlength1} for $\tau_a^{\textnormal{tr}}$ will then be deduced. We focus on $a>0$ as the remaining cases follow by symmetry.
 
Defining $\tilde{\K}_\xi^a(x)$ similarly as in \eqref{eq:Ka} to be the connected component of $x$ in $\{\phi\geq a\}\cap\tilde{B}(x,\xi),$  by the FKG-inequality and Lemma \ref{lem:capacityislarge} we get for $0<a<c$ that
\begin{equation}
\label{eq:capacityislargetau}
    \P\big(\mathrm{cap}(\tilde{\K}_\xi^a)\geq \Cr{ccaplarge}\xi^{\nu},\,\mathrm{cap}(\tilde{\K}_\xi^a(x))\geq \Cr{ccaplarge}\xi^{\nu}\big)\geq \xi^{-\nu}q(\xi)^{-2}\exp\left(-cq(\xi)\right).
\end{equation}
Let $r=d(0,x).$ For all compacts $K\subset\tilde{B}_{\xi }$ and $K'\subset\tilde{B}(x,{\xi })$, consider (cf.~\eqref{eq:defALKar})
\begin{equation*}
    A(K,K',a,r)=\left\{\begin{array}{c}\text{there exists a continuous path }\pi\text{ in }\tilde{B}(0,r)\text{ from }K\text{ to }K'
    \\\hspace{-2mm}\text{with }\phi_x\geq -\rho  a \text{ for all }x\in{\pi\cap ((\tilde{B}(0,{\sigma \xi })\setminus K)\cup(\tilde{B}(x,{\sigma \xi })\setminus K'))}
    \\\text{ and }\phi_x\geq \rho  a \text{ for all }x\in{\pi\cap (\tilde{B}(0,r)\setminus(\tilde{B}(0,{\sigma \xi })}\cup\tilde{B}(x,{\sigma \xi })).\end{array}\right\}.
\end{equation*}
One easily verifies that \eqref{eq:psialphas} still holds when adding $\{\mathrm{cap}(\tilde{\K}_\xi^a(x))\geq \Cr{ccaplarge}\xi^{\nu}\}$ in the indicator function, replacing $\psi(\rho  a ,r)$ by $\tau_{\rho a}^{\textnormal{tr}}(0,x),$ $\tilde{\G}_{\tilde{\K}_\xi^a}$ by $\tilde{\G}_{\tilde{\K}_{\xi}^a\cup \tilde{\K}_\xi^a(x)}$ and $A(\tilde{\K}_\xi^a,a,r)$ by $A(\tilde{\K}_\xi^a,\tilde{\K}_\xi^a(x),a,r).$ Next, one repeats the sausage construction around \eqref{eq:condLB}--\eqref{eq:sausage} (in which the path $\pi$ is eventually built) but replacing the geodesic $\gamma=(0=y_0,y_1,\dots)$ originating in \eqref{eq:intro_shortgeodesic} and considered above \eqref{eq:condLB} by a geodesic $\gamma$ joining $0$ and $x$. One then sets $\mathcal{L}_{\ell}''=(\mathcal{L}_{\ell}'\cup B(x,\sigma'\ell))\setminus B(x,\sigma \xi )$.

With this setup, an analogue of \eqref{eq:lb12} for $A(K,K',a,r)$ holds when replacing $\tilde{\mathcal{L}}_{\ell}'$ by $\tilde{\mathcal{L}}_{\ell}''.$ One then proves an analogue of Lemma \ref{L:lb9} for the quantity $\sup_{K,K'} \, \textnormal{cap}_{\tilde{\G}_{K\cup K'}}( \tilde{\mathcal{L}}_{\ell}'')$ with the supremum ranging over compact sets $K\subset\tilde{B}_{\xi }$ and $K'\subset\tilde{B}(x,{\xi })$, which yields the same upper bound as in \eqref{eq:lb9}. The proof is similar and relies on Lemma~\ref{le:numberofbackwardkilled}, applied directly to $K\cup K'$ (instead of $K$). Note to this effect that $\tilde{\mathcal{L}}_{\ell}''$ has the required ``insulation'' property, i.e.~any unbounded path starting in $K\cup K'$ intersects $\tilde{\mathcal{L}}_{\ell}''$.

Next, one shows under the assumptions of Lemma \ref{L:lb13} that
\begin{equation}
\label{eq:lb13tau}
\P_{a,\ell}^{\tilde{\G}_{K\cup K'}}\big(A(K,K',a,r)\big) \geq {c}\big(1-\exp(-{c}'(\rho  a) ^2\mathrm{cap}(K))\big)\big(1-\exp(-{c}'(\rho  a) ^2\mathrm{cap}(K'))\big),
\end{equation}
for all compacts $K\subset \tilde{B}(0,\xi),$ $K'\subset \tilde{B}(x,\xi)$, which replaces \eqref{eq:lb13}. Here the measure $\P_{a,\ell}^{\tilde{\G}_{K\cup K'}}$ refers to the free field on $\tilde{\G}_{K\cup K'}$, shifted by $2\rho a$ in the region $\tilde{\mathcal{L}}_{\ell}''$ and extended harmonically outside (with a Dirichlet boundary condition on $K\cup K'$). We return to the proof of \eqref{eq:lb13tau} shortly. Combining the above results, following the line of argument leading up to \eqref{eq:LBfinal2}, one deduces that if~\eqref{eq:defb} holds, 
\begin{equation}
\label{eq:lowerboundtaua}
    \tau_{\rho a}^{\textnormal{tr}}(0,x)\geq \xi^{-\nu}q(\xi)^{-2}\exp\left\{-\tilde{c}q(\xi)-\frac{\tilde{c}(r/\xi)^{\nu\wedge1}}{\log(r/\xi)^{b}}\right\}\text{ for all }a\in{[-\tilde{c}',\tilde{c}']}\text{ and }r\geq \tilde{c}''\xi,
\end{equation}
the only difference between \eqref{eq:LBfinal2} and \eqref{eq:lowerboundtaua} coming from the discrepancy between the bounds \eqref{eq:capacityislarge} and \eqref{eq:capacityislargetau}.
In view of \eqref{eq:2ptatcriticality}, if $r/\xi$ is large enough then $\xi^{-\nu}\leq \tau_a^{\text{tr}}(0,x),$ and \eqref{eq:lowerboundtaua} yields lower bounds for $\tau_a^{\textnormal{tr}}$ similar to \eqref{eq:boundonpsiwithb}. 

This readily translates into lower bounds for $\tau_a^{\text{tr}}$ akin to \eqref{eq:corrlength1} when $\nu<1$, $a \in (0, \Cl[c]{c:tau11})$ and $r\geq \Cl[c]{c:tau12}\xi(a),$ and \eqref{eq:corrlength2} when $\nu=1$ and to $r\geq \xi(a)(\log\xi(a))^{\Cr{c:defect3}}.$ Moreover, when $\nu<1,$ $a \in (0, \Cr{c:tau11})$ and $r_0 < r< \Cr{c:tau12}\xi(a)$, where $r_0 =\Cr{c:tau12}\Cr{c:tau11}^{-\frac{2}{\nu}} ,$ defining $b\geq a$ such that $r=\Cr{c:tau12}\xi(b)$ one obtains for $r > r_0$ (whence $b< \Cr{c:tau11}$) that
\begin{equation*}
\tau_{a}^{\text{tr}}(0,x)\geq\tau_b^{\text{tr}}(0,x)\geq \tau_0^{\text{tr}}(0,x)\exp(-\Cr{C2_xi}(r/\xi(b))^{\nu})\geq c\tau_0^{\text{tr}}(0,x)\geq  c\tau_0^{\text{tr}}(0,x)\exp(-\Cr{C2_xi}(r/\xi(a))^{\nu}).
\end{equation*}
Finally, the analogue of \eqref{eq:corrlength1} for $\tau_a^{\text{tr}}(0,x)$ when $r\leq r_0$ and $a \in (0, \Cr{c:tau11})$ is trivial.  

Let us now go back and comment on the proof of \eqref{eq:lb13tau}. We proceed similarly as in the proof of Lemma~\ref{L:lb13}, but take $u=(\rho  a) ^2/6$ and split the interlacements $\I^{3u}$ on $\tilde{\G}_{K\cup K'}$ into three independent interlacements $\I^{u,1},$ $\I^{u,2}$ and $\I^{u,3}$ instead. We then define $S_x$ and $S'_x$ similarly as around \eqref{eq:defS} but replacing all the balls centered at $0$ by balls centered at $x,$ and let
\begin{equation*}
    G^{\tau}=G\cap\Big\{\I_{-}^{u,3,\ell}\cap \I^{u,2}\cap\bigcup_{z\in{S'_x}}B(z,\ell/2)\neq\emptyset\Big\},
\end{equation*}
with $G$ given by \eqref{eq:lb15} and where $\I_{-}^{u,3,\ell}$ is the set of of vertices hit before their first exit time of $B(x,\sigma'\ell)$ by any trajectories in the interlacement process associated to $\I^{u,3}$ whose backwards parts are killed on $K'$. Then under a coupling $\Q$ operating on the cable system $\tilde{\G}_{K\cup K'}$, the event $G^{\tau}$ implies that $K\leftrightarrow K'$ in $\{\phi\geq-\rho  a \}\cap(\tilde{\mathcal{L}}_{\ell}\cup \tilde{B}(0,\sigma'\ell)\cup \tilde{B}(0,\sigma'\ell))\cap\tilde{\G}_{K\cup K'},$ and so its probability is a lower bound for $\P_{a,\ell}^{\tilde{\G}_{K\cup K'}}\big(A(K,K',a,r)\big).$ To bound the probability of $G^{\tau},$ we then proceed as in \eqref{eq:lb16}--\eqref{eq:finalboundonG}, but now conditioning on $\I^{u,1}$ and $\I^{u,3},$ and adding the event $G'_{s_0}(x)=\{\exists\,z\in{S'_x}:\,\mathrm{cap}(\I_{-}^{u,3,\ell}\cap B(z,\ell/2))\geq s_0\ell^{\nu}\}.$ The probability of the event $G_{s_0}'(x)$ is bounded from below as in \eqref{eq:G'shaslargeproba} by a constant times the probability that $\I_{-}^{u,3,\ell}\cap\bigcup_{z\in{S'_x}}B(z,\ell/4)\neq\emptyset.$

It remains to argue that this term and the one corresponding to $G'_{s_0}$ produce the two factors present in \eqref{eq:lb13tau}. Since $d(x,y)\leq d(x,K)$ for all $y\in{\bigcup_{z\in{S'_x}}B(z,\ell/4)},$ $\I_{-}^{u,3,\ell}\cap\bigcup_{z\in{S'_x}}B(z,\ell/4)\neq\emptyset$ as soon as $\I^{u,3}$ contains a trajectory whose backwards part is killed on $K'$ but forwards part is not killed on $K'.$ In order to determine the average number of such trajectories by means of~\eqref{eq:numberofbackwardkilled}, we introduce long one-dimensional chains on the boundary edges of $K$. These chains absorb the effect of the Dirichlet boundary condition on $K$ and present the advantage of giving rise to an augmented graph which is completely massless (and to which Lemma \ref{le:numberofbackwardkilled} applies).

Specifically, we define a discrete set $K_{\infty}$ as in Remark 2.1 in \cite{Pre1}, see also Remark 2.2 in~\cite{DrePreRod3}, which contains for each cable touching $K$ an infinite sequence of vertices converging to $K.$  Adding $K_{\infty}$ to the vertex set $\G_K$ defines a new graph $ \mathcal{G}' \stackrel{\text{def.}}{=} \G_{K}^{K_{\infty}},$ see Lemma 2.1 in \cite{DrePreRod3}, so one can view $\G_K$ as a subset of $\G' $ and the killing measure vanishes on $\G' .$  Since $\G'_{K'}=\G_{K\cup K'}^{K_{\infty}}$ (see below \eqref{eq:lb2alphalarger2nubis} for notation) one can then see $\G_{K\cup K'}$ as a subset of $\G_{K'}',$ and the trajectories on $\G_{K'}'$ can only be killed on $K'$ (and not $K$). The number of trajectories in $\I^{u,3}$ whose backwards part are killed on $K'$ and whose forwards parts are not is then equal to the number of killed-surviving trajectories for the interlacement on $\G_{K'}'$, which by \eqref{eq:killedintensities} and Lemma \ref{le:numberofbackwardkilled} is a Poisson random variable with parameter $u\mathrm{cap}_{\tilde{\G}'}(K')=u\mathrm{cap}_{\tilde{\G}_K}(K'),$ see (2.16) in \cite{DrePreRod3}. Using the inequality $\mathrm{cap}_{\tilde{\G}_K}(K')\geq \mathrm{cap}_{\tilde{\G}}(K'),$ we obtain that 
\begin{equation*}
    \Q(G_{s_0}'(x))\geq c\big(1-\exp(-u\mathrm{cap}_{\tilde{\G}}(K')\big).
\end{equation*}
A similar result holds for the event $G_{s_0}',$ but replacing $K'$ by $K,$
and since $G_{s_0}'$ and $G_{s_0}'(x)$ are independent \eqref{eq:lb13tau} follows. 
\end{proof}

We now first prove~Theorem~\ref{T:psitilde}, which also directly relies on the findings of Sections~\ref{sec:radiusLB}--\ref{sec:lemmaProofs}.

\begin{proof}[Proof of Theorem~\ref{T:psitilde}]
Assume that $\nu\in{(0,\alpha/2]}$  and recall the definition of $\mathcal{L}_{\ell}$ from \eqref{eq:sausage}. A change-of-measure similar to \eqref{eq:lb12} but directly for the unconditioned Gaussian free field on $\tilde{\G}$ gives, for all $0< a< c$ and $r,\ell$ satisfying \eqref{eq:condLB} (with $\sigma'=1$),
\begin{equation}
\label{eq:lb12tilde}
\tilde{\psi}(a,r)=\P^{\tilde{\G}}\big(\tilde{A}(a,r)\big)\geq\P_{a,\ell}^{\tilde{\G}}\big(\tilde{A}(a,r)\big)\exp\left\{-\frac{2a^2\text{cap}_{\tilde{\G}}(\mathcal{L}_{\ell})+1/e}{\P_{a,\ell}^{\tilde{\G}}\big(\tilde{A}(a,r)\big)}\right\},
\end{equation}
where $\tilde{A}(a,r)=\{{B}_{\xi(a)} \stackrel{\geq a}{\longleftrightarrow} \partial_{\textnormal{in}} {B}_r\}$ and $\P_{a,\ell}^{\tilde{\G}}$ now corresponds to the measure from \eqref{eq:df4} with $K\stackrel{\text{def.}}{=}\tilde{\mathcal{L}}_{\ell}$ and $-2a$ in place of $a$. Thus, in particular, $(\varphi_x)_{x \in \tilde{\mathcal{L}}_{\ell}}$ under $\P_{a,\ell}^{\tilde{\G}}$ has the same law as $(\varphi_x+ 2a)_{x \in \tilde{\mathcal{L}}_{\ell}}$ under $\P^{\tilde{\G}}$. One then estimates $\text{cap}_{\tilde{\G}}(\mathcal{L}_{\ell})$ and $\P_{a,\ell}^{\tilde{\G}}\big(\tilde{A}(a,r)\big)$ separately. 
The relevant upper bound for $\mathrm{cap}_{\tilde{\G}}(\mathcal{L}_{\ell})$ (which replaces Lemma~\ref{L:lb9}) has been derived in \eqref{eq:capLell}. The analogue of Lemma~\ref{L:lb13} in the present context is the claim that, with 
\begin{equation}
\label{eq:lb8_extended}
\ell= M \xi \Big( \log \frac{r}{\xi} \Big)^{\frac{2\nu +1}{\nu}} \cdot \Big[( \log \xi)^2 \Big( \log \frac{r}{\xi} \Big) \Big]^{\frac{1}{\nu}\cdot 1\{\alpha=2\nu\}},
\end{equation}
which extends the definition \eqref{eq:lb8}, for all $0<a< c$ and $r/\xi \geq c  (\log\xi)^{\frac{3}{\nu}\cdot 1\{\alpha=2\nu\}}$,
\begin{equation}
    \label{eq:lb13tilde}
\P_{a,\ell}^{\tilde{\G}}\big(\tilde{A}(a,r)\big) \geq c.
\end{equation}
Assuming \eqref{eq:lb13tilde} to be true, \eqref{eq:corrlength2lb} follows in a regime $0<a<\Cl[c]{c:tilde11}$ and $r \geq \Cl[c]{c:tilde12} \xi  (\log\xi)^{\frac{3}{\nu}\cdot 1\{\alpha=2\nu\}} $ using \eqref{eq:lb12tilde}, the upper bound on $\text{cap}_{\tilde{\G}}(\mathcal{L}_{\ell})$ (with $\ell$ as above) and \eqref{eq:lb13tilde} much in the same way as \eqref{eq:boundonpsiwithb} is derived below Lemma~\ref{L:lb13}, with small modifications in case $\alpha= 2\nu$. In case $\alpha> 2\nu$, the regime where \eqref{eq:corrlength2lb} holds extends to all $r \geq 1$ by reducing the region of $a$ to the interval $0< a < {\Cr{c:tilde11}}{\Cr{c:tilde12}^{-\frac\nu2}} $. For such $a$, and $\xi< r < \Cr{c:tilde12} \xi $ (the case $r\leq \xi$ is trivial, cf.~\eqref{eq:intro_psitilde}), one picks $b> a$ such that $r= \Cr{c:tilde12} \xi(b) $ (note that $b =( \Cr{c:tilde12}r^{-1})^{\frac{\nu}{2}} \leq \Cr{c:tilde12}^{\frac{\nu}{2}} a < \Cr{c:tilde11}$) whence
\begin{equation*}
\tilde{\psi}(a,r)\geq\tilde{\psi}(b,r)\geq \tilde{c}\exp\big(-\Cr{C2_xi}(r/\xi(b))\log ((r/\xi(b))\vee2)^{\Cr{c:lbdefect2}}\big)\geq  \tilde{c}',
\end{equation*}
as desired. Finally the case $a<0$ in \eqref{eq:corrlength2lb} follows by symmetry, as explained at the end of the proof of Proposition \ref{prop:lowerbounds}.

It remains to explain how to adapt the proof of Lemma \ref{L:lb13} to obtain \eqref{eq:lb13tilde}. We replace the event $G$ from \eqref{eq:lb15} by 
\begin{equation*}
\tilde{G}= \{ {B}_{\xi(a)} \cap \mathcal{I}^u \neq \emptyset\} \cap \bigcap_{k=0}^{N_{\ell,r}} \big( \big\{ \text{LocUniq}_{u,\ell}(x_k)\big\} \cap \{ \mathcal{I}^u \cap B(x_k,\ell/2) \neq \emptyset\} \big)\text{ for }u=a^2/2.
\end{equation*}
Then under the coupling $\Q$ from \eqref{eq:weakisom}, the event $\tilde{G}$ implies that $B_{\xi(a)}$ is connected to $\partial_{\text{in}} B_r \text{ in }\K^{-a} \cap  {\mathcal{L}}_{\ell},$ and so its probability is smaller than $\P_{a,\ell}^{\tilde{\G}}\big(\tilde{A}(a,r)\big).$ We thus only need to show that $\Q(\tilde{G})\geq c$ in the relevant regime of $r/\xi$. When $\alpha < 2 \nu$, the bound $\Q(\tilde{G})\geq c$ follows for $r/\xi \geq c$ by combining \eqref{eq:defRI} and Theorem~\ref{Thm:locuniq}, which yield a similar bound on $\Q(\tilde{G}^c)$ as the upper bound in \eqref{eq:finalboundonG}, and noting that $u \text{cap}({B}_{\xi(a)}) \geq c$ by \eqref{eq:defxi} and \eqref{eq:capBd}. When $\alpha = 2\nu$, the analogue of the first bound appearing below \eqref{eq:finalboundonG} (which eventually guarantees that $\Q(\tilde{G}^c) \leq c <1$) becomes $u\ell^{\nu}(\log \ell)^{-2}\geq  c\frac{M^{\nu}}{(\log M)^2} (\log \frac{r}{\xi})^{2\nu+1}$ due to the presence of the logarithm in \eqref{eq:lb2alphalarger2nu}. The additional factor present in \eqref{eq:lb8} in case $\alpha=2\nu$ compensates the term $(\log \ell)^{-2}$. Then, \eqref{eq:lb13tilde} follows upon observing that the condition $r \geq c\ell $ with $\ell$ as in \eqref{eq:lb8_extended} is met when $\alpha=2\nu$ and $r \geq c \xi  (\log\xi)^{\frac{3}{\nu}}$.

Finally, one can obtain lower bounds similar to \eqref{eq:corrlength2lb} but for
\begin{equation}
\label{tautilde}
\tilde{\tau}_a^{\textnormal{tr}}(0,x)\stackrel{\text{def.}}{=} \P\big(\{{B}(0,{\xi(a)}) \stackrel{\geq a}{\longleftrightarrow}  {B}(x,\xi(a))\} \setminus \{ {B}(0,{\xi(a)}) \stackrel{\geq a}{\longleftrightarrow} \infty\}\big).
\end{equation}
by changing the event $\tilde{A}(a,r)$ appearing in \eqref{eq:lb12tilde} into $\{{B} (0,\xi(a)) \stackrel{\geq a}{\longleftrightarrow} B(x,\xi(a))\}$ and adapting the above arguments. We omit further details.
\end{proof}

\begin{Rk} \label{R:LB} 
 \begin{enumerate}[label=\arabic*)]
\item\label{R:LB.00}  In case $\alpha=2\nu$, the attentive reader will have noticed that the condition $r\geq c \xi(\log\xi)^{\frac{3}{\nu}\cdot 1\{\alpha=2\nu\}}$ appearing above \eqref{eq:lb13tilde} (as well as the corresponding one in \eqref{eq:corrlength2lb}) can in fact be replaced by the requirement that $r/\xi \geq c (\log\xi)^{\frac{2}{\nu}} (\log \log \xi)^{c}$, for large $c$. Along similar lines, the conclusions of Lemma~\ref{L:lb13} can also be extended to include the case $\alpha=2\nu$ with the choice of $\ell$ from \eqref{eq:lb8_extended}, thus yielding \eqref{eq:lb13} (and \eqref{eq:lb13tau}) for the same regime of $r/\xi$. However, the (best-possible, i.e.~with $q(r)=c$) resulting estimate in Proposition~\ref{prop:lowerbounds} is implied by \eqref{eq:corrlength2lb}. This is ultimately due to the additional presence of the logarithm in Theorem~\ref{Thm:locuniq} when $\alpha=2\nu$.
 
\item \label{R:LB.3} When $\alpha<2\nu,$ one could proceed as in the above proof of Theorem~\ref{T:psitilde} by taking $\ell=a^{-\frac{2}{\beta}}\big(\log a^{\frac{2}{\beta}}r\big)^c |\log a|^{c'}$, for suitable choice of $c,c'$ (cf.~\eqref{eq:lb8}).
In view of \eqref{eq:lb2alphasmaller2nu}, one can then show that the conclusions of Lemma \ref{L:lb13} still hold for this choice of $\ell,$ which leads to (note that Lemma~\ref{L:lb9} holds for any value of $\nu >0$)
\begin{equation}
\label{eq:corrlength2ld}
\tilde{\psi}(a,r) \geq  \exp\big\{ -  \Cr{C2_xi}a^{\frac{2(\beta-\nu+1)}\beta}r\log (a^{2/\beta}r)^{c} |\log a|^{c'} \big\}, \end{equation}
if $a>0$ and $ra^{2/\beta}/|\log a|^{c''}$ is large enough. Proceeding similarly as in \eqref{eq:intro_nu}, the best result one could hope to obtain is thus $\nu_{c}\in{[\frac{2(\beta-\nu+1)}\beta,2]}$ when $\alpha<2\nu$. This lower bound would however only be of interest  when $\beta-\nu+1>0,$ that is $\alpha>2\nu-1,$ and $\alpha<2\nu$ (which e.g.~is never the case of $\Z^d,$ $d\geq3$; it would require $d\in (4,5)$).
\item \label{rk:shortgeodesicsisneeded}
Let us explain why one needs to assume that \eqref{eq:intro_shortgeodesic} holds when $\nu\geq1$ to obtain lower bounds as in \eqref{eq:corrlength2lb}.
Suppose that \eqref{eq:intro_Green} and \eqref{eq:intro_sizeball}  hold for some distance $d$ and some $\nu<1$ and $\alpha>2,$ and let $d_b=d^b$ for some $b\leq1.$ It is easy to check that $d_b$ is a distance, and that $(G_{\nu_b})$ and $(V_{\alpha_b})$ hold for $d_b$ with $\nu_b=\nu/b$ and $\alpha_b=\alpha/b$. Let us define $\tilde{\psi}_b$ as in  \eqref{eq:intro_psitilde}, but for the distance $d_b.$ It then follows from \eqref{eq:corrlength1} for the distance $d$ and a union bound that (note that $\xi(1)=1$) 
\begin{equation*}
\tilde{\psi}_b(1,r) = \tilde{\psi}(1,r^{1/b}) \leq c \exp \big(-c'r^{\nu_b}\big), \quad r \geq 1.
\end{equation*}
In particular, \eqref{eq:corrlength2lb} cannot hold for the distance $d_b$ when $b=\nu,$ that is $\nu_b=1$ nor when $b<\nu,$ that is $\nu_b>1$. The only hypothesis which is not verified for the distance $d_b$ is \eqref{eq:intro_shortgeodesic}. Note that \eqref{eq:corrlength1} is however equivalent for the distance $d$ and the distance $d_b$ when $b>\nu,$ that is $\nu_b<1.$
\end{enumerate}
\end{Rk}

With the full strength of Theorems \ref{T2} and \ref{T:psitilde} at our disposal, we proceed to show their corollaries. We begin by comparing the results of Theorem \ref{T2} with the expected two-sided estimate \eqref{eq:defcorrlength}, thereby relating $\xi$ from \eqref{eq:defxi} to the quantity $\xi'$ defining \eqref{eq:defcorrlength}. By \eqref{eq:corrlength1}, if \eqref{eq:intro_Green} holds for some $\nu < 1$ then so does \eqref{eq:defcorrlength} with $f_{\nu}(t)\asymp t^{\nu}$ and ${\xi}'\asymp \xi$, with $\xi$ as in \eqref{eq:defxi}. Although \eqref{eq:defcorrlength} is not fully determined for larger values of $\nu$ by Theorem~\ref{T2}, the upper bounds on $\psi$ from \eqref{eq:corrlength2} and the lower bounds on $\tilde{\psi}$ from \eqref{eq:corrlength2lb} are already sufficient to obtain rigorous information about $\xi'$, in the following sense.
\begin{Cor}
\label{Corcorrlength} If \eqref{eq:intro_Green}, \eqref{eq:intro_sizeball}, \eqref{eq:ellipticity}, \eqref{eq:intro_shortgeodesic} and \eqref{eq:defcorrlength} hold, then for all $ t \geq c'$ and $ |a|\leq c$, 
\begin{align}
&f_{\nu}(t)\asymp \frac{t}{\log(t)}\text{  and  }\xi'(a)\asymp \xi(a), \text{ if $\nu=1$}, \label{eq:corrlengthnu=1}
\\
&f_{\nu}(t)\asymp t\text{  and  } 
\frac{c \xi(a)}{(\log \xi(a))^{\Cr{c:lbdefect2}+ \Cr{c:lbdefect4}}} \leq\xi'(a)\leq c'a^{-2}, \text{ if $1< \nu \leq \alpha/2$}. \label{eq:corrlengthnu>1}
\end{align}
with the convention $\Cr{c:lbdefect4}=0$ if $\nu < \alpha/2$ $($cf.~\eqref{eq:corrlength2lb}$)$.
\end{Cor}

\begin{proof}[Proof of Corollary \ref{Corcorrlength}]
If $\nu=1$ and \eqref{eq:defcorrlength} is verified, then fixing $a=c,$ it is clear from \eqref{eq:corrlength2} and \eqref{eq:corrlength2lbpsi} that $f_{\nu=1}(r)\asymp r/\log(r)$ for $r\geq c'.$ Now (still assuming \eqref{eq:defcorrlength}) \eqref{eq:corrlength2} implies that 
\begin{equation*}
\xi'(a)\leq ca^{-2}\frac{\log(r)}{\log(r)-\log(\xi'(a))}, \quad |a| >0, \, r \geq 2.
\end{equation*}
Taking the limit as $r\rightarrow\infty,$ we obtain $\xi'(a)\leq ca^{-2}.$ Finally, \eqref{eq:corrlength2lbpsi} directly yields $\xi'(a)\geq c'a^{-2}.$

Let us now assume that $\nu\in{(1,\alpha/2]}$ and \eqref{eq:defcorrlength} is verified. On account of \eqref{eq:intro_shortgeodesic}, one can lower bound $\psi(1,r)$ by the probability that the cable corresponding to each edge along a path from $0$ to $\partial_{\text{in}} B(0,r)$ of length at most $\Cr{Cgeo}r$ for the graph distance is entirely included in $\{ \varphi \geq a\},$ and so by the FKG-inequality (and using e.g.~\eqref{eq:cond_proba}), it follows that $\psi(1,r)\geq \exp(-cr)$. This implies directly that $f_{\nu}(r)\leq c'r.$ Moreover, \eqref{eq:corrlength2} with $a=1$ implies that $f_{\nu}(r)\geq cr$ and $\xi(a)\leq c'a^{-2}.$ For the reverse inequality, combining the FKG-inequality, \eqref{eq:intro_sizeball}, \eqref{eq:intro_Green} (to obtain that $c \leq \E[\varphi_x^2] \leq c'$ for all $x \in G$) and \eqref{eq:cond_proba}, we have the (crude) bound
\begin{equation}
\label{eq:corFKG}
\psi(a,r)\geq \P(\tilde{B}_{\xi(a)}\subset \tilde{\mathcal{K}}^a)\tilde{\psi}(a,r)\geq (\Cl[c]{c:crude})^{ca^{-2\alpha/\nu}} \tilde{\psi}(a,r), 
\end{equation}
valid for all $0\leq |a|\leq1.$ Therefore by \eqref{eq:corrlength2lb} and \eqref{eq:corFKG} we have, for all $0< |a|\leq c$ and $r \geq \xi(a) (\log \xi(a))^{3/\nu}$ (and even all $r \geq 1$ when $\nu < \alpha/2$),
\begin{equation}
\label{eq:needlabelfornextremark}
\xi'(a)\geq \big[c\xi^{-1}(a)(\log r/\xi(a))^{\Cr{c:lbdefect2}} (\log \xi(a))^{\Cr{c:lbdefect4}}+c' \log(1/\Cr{c:crude}) a^{-2\alpha/\nu}r^{-1}\big]^{-1},
\end{equation}
from which $\xi'(a) \geq c\xi(a)/ |\log a|^{\Cr{c:lbdefect2}+\Cr{c:lbdefect4}}$ follows upon choosing $r=ca^{-2(\alpha+1)/\nu},$ which satisfies $r \geq \xi(a) (\log \xi(a))^{3/\nu}$ as required in order for \eqref{eq:corrlength2lb} to apply.
\end{proof}

\begin{Rk}\label{R:radiusasymp}
As implicit in the previous proof, some care is needed for large values of $a$, i.e.~when $a$ reaches the size of typical fluctuations for the local observable $\varphi_{\cdot}$, and this may affect the behavior of $\xi'(a)$ in \eqref{eq:defcorrlength} in this regime. Indeed, refining slightly the lower bound above \eqref{eq:needlabelfornextremark} by using the Gaussian tail estimate $P(X \geq x) \geq cx^{-1}e^{-x^2/2}$ valid for $x \geq 2$, where $X$ is a standard Gaussian variable, one obtains that  $\psi(a,r) \geq (ca^{-1}e^{-c'a^2})^r $ whenever $a> 2\sqrt{\bar g }$ and for all $r \geq 1$, where $\bar g = \sup_{x\in G} g(x,x)$, which is finite under \eqref{eq:intro_Green}. Together with the upper bound from \eqref{eq:corrlength2} this yields in case $\nu >1$ that
 \begin{equation}
 \label{eq:squareasymp}
 -r^{-1} \log \psi(a,r) \asymp a^2  \text{ as $r \to \infty$, whenever $a> 2\sqrt{ \bar g }$ !}
 \end{equation}
 On the basis of Table~\ref{tb:exponents}, see p.~\pageref{tb:exponents}, and the conjectured mean-field behavior for large values of $\nu$, one may expect these asymptotics to fail when $a\ll 1$, and \eqref{eq:squareasymp} represents an obstacle in obtaining any improvement.
\end{Rk}

Next, we turn to Corollary~\ref{C:scalingrelation}, which will quickly follow from the following result, the proof of which relies on the bounds for $\tau_a^{\text{tr}}$ derived in Theorem~\ref{T2}. 

\begin{Prop}[$\nu \leq 1$] \label{P:scalinggamma}
\noindent \begin{enumerate}[label=(\roman*)]
\item If  \eqref{eq:intro_Green}, \eqref{eq:intro_sizeball} and \eqref{eq:ellipticity} hold, then with $\tilde{c}\,= (\alpha-\nu)1\{\nu=1\}$,
\begin{equation}
\label{eq:boundvolumeoffUB}
\E[|\K^a|1\{|\K^a|<\infty\}]\leq c|a|^{- \frac{2\alpha}\nu+2}\log(|a|^{-1}\vee2)^{\tilde{c}}, \text{  for all $|a| \leq c$.}
\end{equation}
\item If in addition $d=d_{\textnormal{gr}}$, then with $\tilde{c}'= \Cr{c:defect3} \cdot 1\{\nu=1\} \,(\leq 1/2)$,
\begin{equation}
\label{eq:boundvolumeoffLB}
\E[|\K^a|1\{|\K^a|<\infty\}]\geq c'|a|^{-\frac{2\alpha}\nu+2}e^{-(\log( |a|^{-1}))^{\tilde{c}'}},
 \text{  for all $ |a| \leq c$.}
\end{equation}
(with the convention that the right-hand side is $+\infty$ when $a=0$).
\end{enumerate}
\end{Prop}

In particular, if $ \psi(0,r) \asymp r^{-\nu/2}$, then \eqref{eq:boundvolumeoffLB} holds with $\tilde{c}'=0$, cf.~below \eqref{eq:corrlength2lbpsi}.

\begin{proof}
Let $f(r,a)=\exp(-\Cr{c4_xi}(r/\xi(a))^{\nu}/\log(r\vee2)^{1\{\nu=1\}})$ if $\nu\leq1.$ It follows from the versions of \eqref{eq:corrlength1} and \eqref{eq:corrlength2} for $\tau_a^{\text{tr}}$, \eqref{eq:2ptatcriticality} and \eqref{eq:intro_Green} that for all $a\in\R,$
\begin{align*}
\E[|\K^a|1\{|\K^a|<\infty\}]&=\sum_{x\in{G}}\tau_a^{\textnormal{tr}}(0,x)\leq c + c'\sum_{x \neq 0}d(0,x)^{-\nu}f(d(0,x),a).
\end{align*}
Therefore, using \eqref{eq:intro_sizeball} we obtain that
\begin{equation}
\label{eq:upperboundonKa}
\begin{split}
\E[|\K^a|1\{|\K^a|<\infty\}]&\leq c+\sum_{n=2}^{\infty}c 2^{-n\nu}f(2^n,a)|B(0,2^{n+1})\setminus B(0,2^{n})|
\\&\leq c+\sum_{n=2}^{\infty}c2^{n(\alpha-\nu)}f(2^n,a)\leq c+\int_0^{\infty}c2^{t(\alpha-\nu)}f(2^t,a)\, \mathrm{ d}t
\end{split}
\end{equation}
where we used that $2^{n(\alpha-\nu)}f(2^n,a)\leq 2^{(t+1)(\alpha-\nu)}f(2^t,a)$ for all $t\in{(n-1,n]}$ in the last inequality. Substituting $u=2^t/(\xi \log(\xi \vee 2)^{1\{\nu=1\}})$ with $\xi= \xi(a)$, the last integral in \eqref{eq:upperboundonKa} is bounded from above by
\begin{equation}
\label{eq:upperboundonKa2}
c\big(\xi(a)\log(\xi(a)\vee2)^{1\{\nu=1\}}\big)^{\alpha-\nu}\int_0^{\infty}u^{\alpha-\nu-1}\exp\left\{-\Cr{c4_xi}u^{\nu}\Big(\frac{\log(\xi(a)\vee2)}{\log(cu\xi(a)^2\vee2)}\Big)^{1\{\nu=1\}}\right\}\, \mathrm{ d}u.
\end{equation}
Since $\log(\xi(a)\vee2)/\log(cu\xi(a)^2\vee2)\geq c'/\log(c''u\vee2)$ and $\alpha-\nu-1\geq 1$ by \eqref{eq:condonalphanu}, the integral in \eqref{eq:upperboundonKa2} is upper-bounded by a finite constant uniformly in $a$ for all $0< |a| \leq 1$. Combining this with \eqref{eq:upperboundonKa} and \eqref{eq:defxi}, \eqref{eq:boundvolumeoffUB} follows.

To deduce \eqref{eq:boundvolumeoffLB}, one proceeds similarly as in \eqref{eq:upperboundonKa}, using instead the lower bound from \eqref{eq:corrlength1} (for $\tau_a^{\text{tr}}$) when $\nu< 1$ and \eqref{eq:corrlength2lbpsi} when $\nu =1$, along with \eqref{eq:2ptatcriticality} and \eqref{eq:intro_Green}, to find that
\begin{equation}
\label{eq:lowerboundonKa}
\begin{split}
\E[|\K^a|1\{|\K^a|<\infty\}]& \geq  c\int_{0}^{\infty}2^{t(\alpha-\nu)}\tilde{f}(2^t,a) 1\{ 2^t > c'\xi(1 \vee (\log \xi)^{\tilde{c}'}) \}\, \mathrm{ d}t, \text{ for all $0< |a| \leq c$},
\end{split}
\end{equation}
where $\tilde{f}(r,a)=\exp(-\Cr{C2_xi}(r/\xi(a))^{\nu}/\log((r/\xi(a))\vee2)^{1\{\nu=1\}})$ and $\tilde{c}'= \Cr{c:defect3} \cdot 1\{\nu=1\}$. For $a=0$ \eqref{eq:lowerboundonKa} holds without indicator function. Since $\tilde{f}(r,0)=1$ and $\alpha > \nu $, \eqref{eq:lowerboundonKa} immediately yields $\E[|\K^0|1\{|\K^0|<\infty\}]=\infty$. For $0< |a| \leq c$, substituting $u= 2^t/\xi(a)$ in \eqref{eq:lowerboundonKa} and recalling \eqref{eq:defxi}, \eqref{eq:boundvolumeoffLB} readily follows.
\end{proof}

\begin{proof}[Proof of Corollary~\ref{C:scalingrelation}]
The assertion \eqref{eq:gamma1} follows immediately from \eqref{eq:boundvolumeoffUB} and \eqref{eq:boundvolumeoffLB}.
\end{proof}

\begin{Rk}\label{rk:1stscalingrelation}
 \begin{enumerate}[label=\arabic*)]
  \item \label{rk:1stscalingrelationvolumerenormalized} Further to $|\mathcal{K}^a|$, one can consider a coarse-grained (renormalized) volume observable $|\mathcal{K}^a|_{\textnormal{ren}}$, which is instructive. Assume that \eqref{eq:intro_Green}, \eqref{eq:intro_sizeball}, \eqref{eq:ellipticity} and $d=d_{\text{gr}}$ hold, and let $\mathbb{L}_a \ni 0$ be a lattice in $G$ at scale $\xi=\xi(a)$ for $a \neq 0$. That is,  $\bigcup_{x \in \mathbb{L}_a} B(x,\xi)= G$ and $d(x,y) \geq c\xi$ for any pair of points $x \neq y \in \mathbb{L}_a$ (see Lemma 6.1 in \cite{DrePreRod2} regarding their existence). Then let
\begin{equation}
\label{eq:volumerenorm}
|\mathcal{K}^a|_{\textnormal{ren}} \stackrel{\text{def.}}{=} \sum_{x \in \mathbb{L}_a} 1\{ B(0,\xi) \stackrel{\geq a}{\longleftrightarrow}  {B}(x,\xi)\}, \quad a \neq 0. 
\end{equation}
In view of~\eqref{tautilde}, one has that $\E[|\K^a|_{\textnormal{ren}}1\{|\K^a|_{\textnormal{ren}}<\infty\}]=\sum_{x\in{ \mathbb{L}_a}}\tilde{\tau}_a^{\textnormal{tr}}(0,x)$, and following the arguments leading to the lower bound~\eqref{eq:boundvolumeoffLB} but using \eqref{eq:corrlength2lb} in its form for $\tilde{\tau}_a^{\textnormal{tr}}$ rather than \eqref{eq:corrlength2lbpsi}, one finds when $0< \nu<\frac{\alpha}{2}$ that 
\begin{equation} 
\label{eq:LBvolumerenorm}
\E[|\K^a|_{\textnormal{ren}}1\{|\K^a|_{\textnormal{ren}}<\infty\}] \geq c \int_1^\infty  v^{\alpha-1} e^{-c'v}\, {\rm d}v , \text{ for $0< |a| \leq c$}.
\end{equation}
(note that the integral is roughly $\Gamma(\alpha)$, where $\Gamma(\cdot)$ denotes the Euler-Gamma function).
Comparing with the derivation of \eqref{eq:boundvolumeoffLB}, the constant order lower bound (uniform in $a$!) in \eqref{eq:LBvolumerenorm} may a-priori suffer from the absence of a correct pre-factor corresponding to ``$\tilde{\tau}_0^{\textnormal{tr}}(0,x)$'' in \eqref{eq:corrlength2lb}. But on account of~\eqref{eq:lb12tilde}, such a pre-factor is expectedly of order unity uniformly in $x$, essentially because
$$
\liminf_{a \downarrow 0} \inf_{x \in G} \P\big(B(0,\xi(a)) \stackrel{\geq -a}{\longleftrightarrow}  {B}(x,\xi(a))\big) \geq c
$$
by a similar argument as in \eqref{eq:lb13tilde}. Overall, \eqref{eq:LBvolumerenorm} is thus plausibly sharp, and it intuitively signals that the length scale needed to correctly measure $|\K^a|$ on the event $\{|\K^a|<\infty\}$ extends to (a few units of) scale $\xi$, but not beyond. This is further confirmation of $\xi(a)$ in \eqref{eq:defxi} as a correct lower bound for the correlation length in this problem. We refer to Remark \ref{rk:appendix},\ref{rk:appendix2} for similar considerations in case $\alpha=2\nu.$

\item\label{rk:1stscalingrelationvolumerenormalizednot} Assume \eqref{eq:intro_Green}, \eqref{eq:intro_sizeball}, \eqref{eq:ellipticity} to hold for some $\nu>1.$ Then proceeding similarly as in the proof of Proposition \ref{P:scalinggamma}(i) one can use the version of \eqref{eq:corrlength2} for $\tau_a^{\text{tr}}(0,x)$ to prove that 
\begin{equation*}
\E[|\K^a|1\{|\K^a|<\infty\}]\leq c|a|^{- 2(\alpha-\nu)}, \text{  for all $|a| \leq c$ and $\nu>1$.}
\end{equation*}
This yields the upper bound $\overline{\gamma}\leq 2(\alpha-\nu),$ where $\overline\gamma$ (resp.~$\underline\gamma$) is the limsup (resp.~liminf) of the right-hand side of \eqref{eq:gamma1} as $a \searrow 0$ or $a \nearrow 0.$

Now assume additionally that $d=d_{\text{gr}}$ and $\nu\in{(1,\alpha/2)}.$ One can adapt the proof of Proposition \ref{prop:lowerbounds} as in the proof of Theorem \ref{T2} to obtain that, under \eqref{eq:defb}, with $r=d(0,x)$,
\begin{equation*}
\tau_a^{\text{tr}}(0,x)\geq \frac{cr^{-\nu}}{q(\xi)^2}\exp\left\{-\tilde{c}(r/\xi)\log(r/\xi(a))^{\Cr{c:lbdefect2}}-c'q(\xi)\right\}\text{ for all }a\in{[-\tilde{c}',\tilde{c}']}\text{ and }r\geq \tilde{c}''\xi(a).
\end{equation*}
Proceeding similarly as in the proof of Proposition \ref{P:scalinggamma}(ii), one then readily deduces that
\begin{equation*}
\E[|\K^a|1\{|\K^a|<\infty\}]\geq c|a|^{-\frac{2\alpha}{\nu}+2}\exp(-c'q(\xi)), \text{  for all $|a| \leq c$ and $\nu\in{(1,\alpha/2)}$.}
\end{equation*}
In particular, if \eqref{eq:defb} hold for some function $q$ verifying $q(r)=o(\log(r)),$ $r\rightarrow{\infty},$ this yields the lower bound $\underline{\gamma}\geq \frac{2\alpha}{\nu}-2$.
\end{enumerate}
\end{Rk}

Finally, we present the 
\begin{proof}[Proof of Corollary \ref{Cor:upperboundvolume}]
We first note that
\begin{equation}
\label{upperboundvolumeeq1}
\begin{split}
\P(|\K^0|\geq n)&\leq \P(\text{rad}(\K^0)\geq n^{\frac2{2\alpha-\nu}})+\P(|\K^0\cap B(0,n^{\frac2{2\alpha-\nu}})|\geq n)
\\&\leq cn^{-\frac\nu{2\alpha-\nu}}\log(n)^{\frac{1\{\nu=1\}}2}+\frac{1}{n}\E[|\K^0\cap B(0,n^{\frac2{2\alpha-\nu}})|],
\end{split}
\end{equation}
where in the second line we used \eqref{eq:psi_0nu=1}, as well as Markov's inequality. Moreover, by \eqref{eq:2ptatcriticality}, \eqref{eq:intro_Green} and \eqref{eq:intro_sizeball}, we have (recall that $B_r=B(0,r)$),
\begin{align*}
&\E[|\K^0\cap B(0,n^{\frac2{2\alpha-\nu}})|]\leq \sum_{k=0}^{\lfloor n^{2/(2\alpha-\nu)}\rfloor}ck^{-\nu}|B_k \setminus B_{k-1}| \\
&\qquad  \leq c\sum_{k=1}^{\frac2{2\alpha-\nu}\log_2(n)}\sum_{p=2^{k-1}}^{2^k}p^{-\nu} |B_p \setminus B_{p-1}|
\leq c\sum_{k=1}^{\frac2{2\alpha-\nu}\log_2(n)}2^{-(k-1)\nu}2^{\alpha k} \leq cn^{\frac{2(\alpha-\nu)}{2\alpha-\nu}}.
\end{align*}
Combining this with \eqref{upperboundvolumeeq1}, we obtain \eqref{eq:upperboundvolume}.
\end{proof}

\appendix

\section{Appendix: An enhanced change-of-measure formula}
\label{sec:appendix}

An essential tool in obtaining lower bound on various probabilities for the Gaussian free field has been a certain change-of-measure formula, see  \cite{BDZ95}, \cite{Sz15}, \cite{GRS20} or \eqref{eq:lb12} for instance. In this appendix, we present another version of this formula when studying events only depending on the cluster $\tilde{\K}^a$ from \eqref{eq:introKa}, which is useful in the proof of Lemma \ref{lem:capacityislarge}. Recall the definition of $\tilde{\K}^a_K$ for a compact $K$ from \eqref{eq:Ka}, and the definition $g=g(0,0)$ from below \eqref{T1_beta}.

\begin{Prop}[under \eqref{T1_sign}]
\label{Prop:appendix}
Let $K\subset\tilde{\G}$ be a compact set containing $0$ and $B\subset 2^{K}$ be such that the events $\{\tilde{\K}^a_K(\phi)\in{B}\}$, $a\in\R$, are measurable. Then for all $a\in{\R}$ and~$b>0,$
\begin{equation}
\label{newlb4}
\P\big(\tilde{\K}^{a+b}_K\in{B}\big)\geq \P\big(\tilde{\K}^a_K\in{B}\big)\exp\left\{-\frac{b^2}{2}\mathrm{cap}(K)-\frac{2b\big(1+(a\wedge0)^2\mathrm{cap}(K)\big)}{\sqrt{2\pi g}\P\big(\tilde{\K}^a_K\in{B}\big)}\right\}.
\end{equation}
\end{Prop}

\begin{proof}
Let $A(h)=\{\tilde{\K}_K^h\in{B}\}$ for all $h\in\R$ and $\P_{-b}$ be the measure defined in \eqref{eq:df4}. By Jensen's inequality, one has
\begin{equation}
\label{newlb1}
\begin{split}
\log\left(\frac{\P(A(a+b))}{\P_{-b}(A(a+b))}\right)&=\log\left(\E_{-b}\left[\left(\frac{{\rm d}{\P}_{-b}}{{\rm d}\P}\right)^{-1}\frac{1\{A(a+b)\}}{\P_{-b}(A(a+b))}\right]\right)
\\&\geq -\E_{-b}\left[\log\left(\frac{{\rm d}{\P}_{-b}}{{\rm d}\P}\right)\frac{1\{A(a+b)\}}{\P_{-b}(A(a+b))}\right].
\end{split}
\end{equation}
Now since $\phi+h^{b}$ has the same law under $\P$ as $\phi$ under $\P_{-b},$ one has $\P_{-b}(A(a+b))=\P({A}(a))$ and hence, in view of \eqref{eq:df4},
\begin{equation}
\label{newlb2}
\begin{split}
\E_{-b}\left[\log\left(\frac{{\rm d}{\P}_{-b}}{{\rm d}\P}\right)1\{A(a+b)\}\right]&=b\E_{-b}\left[M_{K}1\{A(a+b)\}\right]-\frac{b^2}{2}\text{cap}(K)\P_{-b}(A(a+b))
\\&=b\E\left[M_{K}1\{{A}(a)\}\right]+\frac{b^2}{2}\text{cap}(K)\P({A}(a)).
\end{split}
\end{equation}
We will now rewrite the expectation in the second line of \eqref{newlb2}, and we start with the case $a<0.$ Using $\E[M_{K}]=0$ (cf.\ \eqref{eq:df2}), we infer that
\[
  \E\left[M_{K}1\{{A}(a)\}\right]
=-\E\left[M_{K}1\{{A}(a)^c,\phi_0\geq a\}\right]-\E\left[M_{K}1\{{A}(a)^c,\phi_0< a\}\right].
\]
We now introduce a conditioning on $\mathcal{A}^+_{\tilde{\K}^{a}_K}$ in the first expectation on the right-hand side and notice that $A(a)^c\cap\{\phi_0\geq a\}\in{\mathcal{A}^+_{\tilde{\K}^{a}_K}}.$ Analogously, within the second expectation we condition on $\mathcal{A}^+_{\{0\}}$ and observe that $A(a)^c\cap\{\phi_0< a\}\in{\mathcal{A}^+_{\{0\}}}.$ As a consequence,  by \eqref{eq:df5} and the equality $\E[M_K\,|\,\mathcal{A}^+_{\{0\}}]=M_0,$ we deduce using \eqref{eq:Markov2} that
\begin{equation}
\label{appendix:1}
\begin{split}
    \E\left[M_{K}1\{{A}(a)\}\right]=-\E\big[M_{\tilde{\K}^{a}_{K}}1\{{A}(a)^c,\phi_0\geq a\}\big]-\E\big[M_{0}1\{{A}(a)^c,\phi_0< a\}\big].
\end{split}
\end{equation}
Since $\phi\geq a$ on the support of  $e_{\tilde{\K}_K^a},$ we further infer that on the event $\{\phi_0\geq a\}$, by \eqref{eq:cap} and \eqref{eq:df2} (cf.~the derivation of \eqref{eq:deriv1.1bis} for a similar argument),
\begin{equation}
\label{appendix:2}
    \E\big[M_{\tilde{\K}^{a}_{K}}1\{{A}(a)^c,\phi_0\geq a\}\big]\geq a\E\big[\mathrm{cap}(\tilde{\K}^{a}_{K})1\{A(a)^c,\phi_0\geq a\}\big]\geq a\E\big[\mathrm{cap}(\tilde{\K}^{a}_{K})1\{\phi_0\geq a\}\big],
\end{equation}
where we took advantage of the assumption $a< 0$ to deduce the last inequality.
It finally follows from \eqref{eq:dens4} and \eqref{T1_theta_0} that
\begin{equation}
\label{appendix:3}
\begin{split}
    \E\big[\mathrm{cap}(\tilde{\K}^{a}_{K})1\{\phi_0\geq a\}\big]&\leq \E\big[\mathrm{cap}(\tilde{\K}^{a})1\{\tilde{\K}^a\text{ is bounded},\phi_0\geq a\}) \big]+\mathrm{cap}(K)\P(\tilde{\K}^{a}\text{ is unbounded})
    \\&= \frac{f(a)}{-a}+(1-2\Phi(a))\mathrm{cap}(K)\leq -\frac{1}{\sqrt{2\pi g}}\Big(\frac{1}{a}+2a\mathrm{cap}(K)\Big).
\end{split}
\end{equation}
Combining \eqref{newlb1} with \eqref{appendix:3} and the inequality $\E[M_01\{\phi_0<a\}]\geq -(2\pi g)^{-1/2}$ (which follows by simple integration since $M_0=\phi_0/g$), we obtain \eqref{newlb4} when $a<0.$ Suppose now that $a\geq 0.$ Repeating the arguments leading to \eqref{appendix:1}, since $M_{\tilde{\K}^{a}_{K}}\geq0$ on the event $\{\phi_0\geq a\},$ we have that
\begin{equation}
\label{appendix:4}
\begin{split}
    \E\left[M_{K}1\{{A}(a)\}\right]=\E\big[M_{\tilde{\K}^{a}_{K}}1\{{A}(a)\}\big]&\leq \E\big[M_{\tilde{\K}^{a}_{K}}1\{\phi_0\geq a\}\big]+\E\big[M_01\{\phi_0< a,A(a)\}\big]
    \\&\leq \E[M_K1\{\phi_0\geq a\}]=\E[M_01\{\phi_0\geq a\}]\leq\frac{1}{\sqrt{2\pi g}},
\end{split}
\end{equation}
where, in deducing the inequality in the second line of \eqref{appendix:4}, we used the fact that $\E[M_01\{\phi_0< a,A(a)\}]= \E[M_01\{\phi_0< a\}]< 0$ in case $A(a) \supset \{\tilde{\K}^a_K=\emptyset \}$, and $\E[M_01\{\phi_0< a,A(a)\}]=0$ otherwise. Thus, \eqref{newlb4} follows from \eqref{newlb1}, \eqref{newlb2} and \eqref{appendix:4} for $a>0.$
\end{proof}

\begin{Rk}
\label{rk:appendix}
\begin{enumerate}[label=\arabic*)]
\item \label{rk:appendix1} The identity \eqref{newlb4} typically improves the usual change-of-measure formula, see for instance below (2.7) in \cite{BDZ95}, cf.~also \eqref{eq:lb12tilde} above, when $b$ and $\P(\tilde{\K}_K^a\in{B})$ are both small. For example, under \eqref{eq:intro_Green} and \eqref{eq:ellipticity}, assuming that $a=0,$ $K=\tilde{B}_r,$ and considering the event $B=\{r\leq\text{rad}(\K^b)\},$ we infer using \eqref{newlb4}, \eqref{eq:psi_0nu<1} and \eqref{eq:capBd} that under \eqref{eq:intro_Green}, for all $b >0$,
\begin{equation*}
\begin{split}
    \psi(b,r)\geq\psi(0,r)\exp\Big(-cb^2r^{\nu}-\frac{c'b}{\psi(0,r)}\Big)&\geq\psi(0,r)\exp(-cb^2r^{\nu}-c'br^{\nu/2})
    \\&\geq c\psi(0,r)\exp(-c'b^2r^{\nu}).
    \end{split}
\end{equation*}
This is a simple proof via the change-of-measure formula of the lower bound in \eqref{eq:corrlength1} when $\nu<1,$ and  of \eqref{eq:rad_113} when $r\leq\xi(b).$ 

\item \label{rk:appendix2}Another application of the identity \eqref{newlb4} is a bound akin to \eqref{eq:corrlength2lb} when $\alpha=2\nu,$ but for $\tilde{\tau}_a^{\text{tr}}$ instead, see \eqref{tautilde}, which is weaker but valid in the whole regime $r\geq\xi.$ Assume \eqref{eq:intro_Green} and \eqref{eq:ellipticity} hold  and fix some $a>0$ and $x\in{\G}$ with $r=d(0,x)\geq 10\xi(a).$ Proceeding similarly as in the proof of \eqref{eq:lb13tau}, we split $\I^{3u}$ into three independent interlacements $\I^{u,1},$ $\I^{u,2}$ and $\I^{u,3},$ with $u=a^2/6,$ and denote by $G'_s(0)$ the event that the set of vertices $\I^{u,1,r}_-$ hit by any trajectories of $\I^{u,1}$ between their first hitting time of $B(0,\xi(a))$ and their subsequent first exit time from $B(0,r)$ has capacity at least $sr^\nu/\log(r).$ We call $G'_s(x)$ the same event but for the the set of vertices $\I^{u,3,r}_-$ hit by any trajectories of $\I^{u,3}$ between their first hitting time of $B(x,\xi(a))$ and their suceeding first exit time from $B(0,r).$ Under the coupling \eqref{eq:weakisom} one knows that ${B}(0,{\xi(a)})$ is connected to ${B}(x,\xi(a))$ in $\{\phi\geq -a\}\cap B(0,tr)$ whenever the intersection of the events  $G'_s(0),$ $G'_s(x)$ and  $\{\I^{u,1,r}_-\leftrightarrow\I^{u,3,r}_-$ in $\I^{u,2}\cap B(0,tr)\}$ occurs. Using the bound $u\mathrm{cap}(B_{\xi(a)})\geq c,$ see \eqref{eq:defxi} and \eqref{eq:capBd} with $K=\emptyset,$ \eqref{eq:defRI} and Lemma~\ref{lem:4.7}, one readily sees that there exists $s>0$ such that the probabilities of $G'_s(0)$ and $G'_s(x)$ are of constant order. Taking $t>0$ large enough, it then follows from Lemma 4.3 in \cite{DrePreRod2} that
\begin{equation*}
\begin{split}
\P\big({B}(0,{\xi(a)}) \leftrightarrow  {B}(x,\xi(a))\text{ in }\{\phi\geq -a\}\cap B(0,tr)\big)&\geq c\big(1-\exp(-c'a^2r^{-\nu}(sr^{\nu}/\log r)^2)\big)
\\&\geq c\left(1\wedge\frac{a^2r^{\nu}}{\log (r)^{2}}\right).
\end{split}
\end{equation*}
Hence, for all $a\in{(0,1)}$ and $x\in{G}$ with $r=d(0,x)\geq 10\xi(a),$ it follows from \eqref{eq:capBd} and \eqref{newlb4} with $K=B(0,tr),$ as well as the inequality $a\leq 1\wedge (a^2r^{\nu}\log(r)^{-2})$ that
\begin{equation}
\label{tautildealpha=2nu}
\tilde{\tau}_a(0,x)\geq c\left(1\wedge\frac{a^2r^{\nu}}{\log(r)^{2}}\right)\exp(-ca^2r^{\nu}).
\end{equation}
The same bound for $a<0$ also holds by symmetry. The bound \eqref{tautildealpha=2nu} is worse than \eqref{eq:corrlength2lb} as $r\rightarrow\infty,$ but has the advantage to be of logarithmic order when $r=\xi.$ Recalling the definition of $|\K^a|_{\textnormal{ren}}$ from \eqref{eq:volumerenorm}, similarly as in \eqref{eq:LBvolumerenorm} it then follows from \eqref{tautildealpha=2nu} and the inequality $a^2r^{\nu}/\log(r)^2\geq 1/\log(\xi)^2$ for all $r\geq \xi,$ that when $\alpha=2\nu$, 
\begin{equation*} 
\E[|\K^a|_{\textnormal{ren}}1\{|\K^a|_{\textnormal{ren}}<\infty\}] \geq \frac{c}{\log(\xi)^2} \int_1^\infty  v^{\alpha-1} e^{-c'v^{\nu}}\, {\rm d}v , \text{ for $0< |a| \leq c$},
\end{equation*}
assuming \eqref{eq:intro_Green}, \eqref{eq:intro_sizeball} and \eqref{eq:ellipticity} hold. This further confirms that $\xi(a)$ in \eqref{eq:defxi} is a lower bound for the correlation length when $\alpha=2\nu$ (for instance when $\G=\Z^4$), up to logarithmic corrections, as indicated in \eqref{eq:corrlengthnu>1}. A similar approach when $\alpha< 2\nu$ recovers \eqref{eq:LBvolumerenorm} without the assumptions $d=d_{\text{gr}}$ or \eqref{eq:intro_shortgeodesic}. 
\end{enumerate}
\end{Rk}

\bigskip

\textit{Acknowledgments.} AD has been supported by  Deutsche Forschungsgemeinschaft (DFG) grant DR 1096/1-1, 
    AP  by the Engineering and Physical Sciences Research Council (EPSRC) grant EP/R022615/1 and Isaac Newton Trust (INT) grant G101121.

\bigskip    

\bibliography{bibliographie}
\bibliographystyle{abbrv}

\end{document}